\documentclass[a4paper,10pt]{article}

\usepackage{fullpage}
\usepackage{amsmath}
\usepackage{amsfonts}
\usepackage{amssymb}
\usepackage{amsthm}
\usepackage{mathrsfs}
\usepackage{textcomp}
\usepackage{stmaryrd}
\usepackage{latexsym} % for nicer \leadsto
\usepackage{mathpartir}
\usepackage{longtable}
\usepackage[utf8]{inputenc}
\usepackage[OT4]{fontenc}
\usepackage{hyperref}

\title{Higher-order illative combinatory logic}

\author{\L{}ukasz Czajka \medskip \\
  \small \vspace{-0.2em}
    Institute of Informatics, University of Warsaw\\
  \small \vspace{-0.2em}
    Banacha 2, 02-097 Warszawa, Poland\\
  \small
    lukaszcz@mimuw.edu.pl
}

\theoremstyle{plain}
\newtheorem{theorem}{Theorem}[section]

\newtheorem{lemma}[theorem]{Lemma}
\newtheorem{fact}[theorem]{Fact}
\newtheorem{corollary}[theorem]{Corollary}

\theoremstyle{definition}
\newtheorem{definition}[theorem]{Definition}

\newtheorem{remark}[theorem]{Remark}

\theoremstyle{plain}
\newtheorem{sstheorem}{Theorem}[subsection]

\newtheorem{sslemma}[sstheorem]{Lemma}
\newtheorem{ssfact}[sstheorem]{Fact}
\newtheorem{sscorollary}[sstheorem]{Corollary}

\theoremstyle{definition}
\newtheorem{ssdefinition}[sstheorem]{Definition}
\newtheorem{ssnotation}[sstheorem]{Notation}

\newtheorem{ssexample}[sstheorem]{Example}

\newcommand{\leftidx}[3]{{\;\vphantom{#2}}#1\!\!#2#3}

\newcommand{\from}{\ensuremath{\leftarrow}}

\newcommand{\reduces}{\ensuremath{\twoheadrightarrow}}
\newcommand{\contr}{\ensuremath{\rightarrow}}

\newcommand{\valuation}[3]{\ensuremath{\llbracket#1\rrbracket_{#2}^{#3}}}

\newcommand{\proves}{\ensuremath{\vdash}}
\newcommand{\forces}{\ensuremath{\Vdash}}
\newcommand{\notproves}{\ensuremath{\nvdash}}
\newcommand{\notforces}{\ensuremath{\nVdash}}
\newcommand{\pair}[2]{\langle#1,#2\rangle}

\newcommand{\transl}[1]{\ensuremath{\lceil#1\rceil}}

\newcommand{\rank}{\ensuremath{\mathrm{rank}}}

\newcommand{\Ps}{\ensuremath{\mathrm{\mathsf P}}}

\newcommand{\Y}{\ensuremath{{\Upsilon}}}
\newcommand{\T}{\ensuremath{{\mathbb T}}}
\newcommand{\Ts}{\ensuremath{{\mathscr{T}}}}

\newcommand{\Tc}{\ensuremath{{\cal T}}}

\newcommand{\B}{\ensuremath{{\cal B}}}
\newcommand{\I}{\ensuremath{{\cal I}}}

\newcommand{\Sc}{\ensuremath{{\cal S}}}
\newcommand{\F}{\ensuremath{{\cal F}}}
\newcommand{\C}{\ensuremath{{\cal C}}}
\newcommand{\D}{\ensuremath{{\cal D}}}
\newcommand{\M}{\ensuremath{{\cal M}}}
\newcommand{\N}{\ensuremath{{\cal N}}}

\makeatletter

\newcommand{\Rmnum}[1]{\expandafter\@slowromancap\romannumeral #1@}
\makeatother

\date{November 25, 2014}

\begin{document}
\maketitle

\begin{abstract}
  We show a model construction for a system of higher-order illative
  combinatory logic~$\I_\omega$, thus establishing its strong
  consistency. We also use a variant of this construction to provide a
  complete embedding of first-order intuitionistic predicate logic
  with second-order propositional quantifiers into the system $\I_0$
  of Barendregt, Bunder and Dekkers, which gives a partial answer to a
  question posed by these authors.
\end{abstract}

\noindent This paper is a revised version of~\cite{Czajka2013JSL}
which appeared in the Journal of Symbolic Logic, vol.~78, issue~3,
pp.~837-872. An error in Section~5 and some minor mistakes in
Section~4 are corrected. Also, the construction in Section~4 is
slightly simplified. \copyright~2013 by the Association for Symbolic
Logic.

%\tableofcontents
%\listoffigures
%\listoftables

\section{Introduction}\label{r:sec_intro}

Illative systems of combinatory logic or lambda-calculus consist of
type-free combinatory logic or lambda-calculus extended with
additional constants intended to represent logical notions. In fact,
early systems of combinatory logic and lambda calculus (by
Sch\"onfinkel, Curry and Church) were meant as very simple foundations
for logic and mathematics. However, the Kleene-Rosser and Curry
paradoxes led to this work being abandoned by most logicians.

It has proven surprisingly difficult to formulate and show consistent
illative systems strong enough to interpret traditional logic. This
was accomplished in \cite{illat01}, \cite{illat02} and \cite{illat03},
where several systems were shown complete for the
universal-implicational fragment of first-order intuitionistic
predicate logic.

The difficulty in proving consistency of illative systems in essence
stems from the fact that, lacking a type regime, arbitrary recursive
definitions involving logical operators may be formulated, including
negative ones. In early systems containing an unrestricted implication
introduction rule this was the reason for the Curry's paradox
\cite[\textsection8A]{illat01, CurryFeys1958}, where an arbitrary term
$X$ is derived using a term $Y$ satisfying $Y =_{\beta} Y \supset
X$. For an overview of and introduction to illative combinatory logic
see \cite{illat01}, \cite{Seldin2009} or \cite{CurryFeys1958}.

Systems of illative combinatory logic are very close to Pure Type
Systems. The rules of illative systems, however, have fewer
restrictions, judgements have the form $\Gamma \proves t$ where $t$ is
an arbitrary term instead of $\Gamma \proves N : C$. This connection
has been explored in \cite{BunderDekkers2005} where some illative-like
systems were proven equivalent to more liberal variants of PTSs from
\cite{BunderDekkers2001}. Those illative systems, however, differ
somewhat from what is in the literature.

In~\cite{Czajka2011} an algebraic treatment of a combination of
classical first-order logic with type-free combinatory logic was
given. On the face of it, the system of \cite{Czajka2011} seems to be
not quite like traditional illative combinatory logic, but the methods
used in the present paper are a (substantial) extension of those
from~\cite{Czajka2011}.

In this work we construct a model for a system of classical
higher-order illative combinatory logic~$\I_\omega^c$, thus
establishing a strong consistency result. We also use a variant of
this construction to improve slightly on the results
of~\cite{illat01}. We show a complete embedding of the system
$\mbox{PRED2}_0$ of first-order intuitionistic many-sorted predicate
logic with second-order propositional quantifiers into the
system~$\I_0$ which is an extension of~$\I \Xi$ from~\cite{illat01}.

To be more precise, we define a translation~$\transl{-}$ from the
language of~$\mbox{PRED2}_0$ to the language of~$\I_0$, and a
mapping~$\Gamma$ from sets of formulas of~$\mbox{PRED2}_0$ to sets of
terms of~$\I_0$. The embedding is proven to satisfy the following for
any formula~$\varphi$ of~$\mbox{PRED2}_0$ and any set of
formulas~$\Delta$ of~$\mbox{PRED2}_0$:
\[
\Delta \proves_{\mathrm{PRED2}_0} \varphi \mathrm{\ \ iff\ \ }
       \transl{\Delta}, \Gamma(\Delta, \varphi) \proves_{\I_0} \transl{\varphi}
\]
where $\Delta, \varphi$ stands for $\Delta \cup \{\varphi\}$. The
implication from left to right is termed soundness of the embedding,
from right to left -- completeness.

Our methods are quite different from those of \cite{illat01}, where an
entirely syntactic approach is adopted. We define a Kripke semantics
for illative systems and prove it sound and complete\footnote{In fact,
  for completeness of the embedding the easier soundness of the
  semantics would suffice, i.e., the completeness of the semantics is
  not necessary for the main results of this paper.}. Given a Kripke
model $\N$ for $\mbox{PRED2}_0$ we show how to construct an illative
Kripke model $\M$ for $\I_0$ such that exactly the translations of
statements true in a state of $\N$ are true in the corresponding state
of $\M$. This immediately implies completeness of the embedding.

The model constructions for $\I_0$ and $\I_\omega^c$ are similar, but
the latter is much more intricate. The basic idea is to define for
each ordinal $\alpha$ a relation $\leadsto_\alpha$ between terms and
so called ``canonical terms''. To every canonical term we associate a
unique type. In a sense, the set of all canonical terms of a given
type fully describes this type. Intuitively, $t \leadsto_\alpha \rho$
holds if $\rho$ is a ``canonical'' representant of $t$ in the type of
$\rho$. This relation encompasses a definition of truth when $\rho \in
\{\top, \bot\}$. Essentially, $\leadsto_\alpha$ is defined by
transfinite induction in a monotonous way. We show that there must
exist some ordinal~$\zeta$ such that $\leadsto_\alpha =
\leadsto_\zeta$ for $\alpha > \zeta$. We use the relation
$\leadsto_\zeta$ to define our model. Then it remains to prove that
what we obtain really is the kind of model we expect, which is the
hard part.

\section{Preliminaries}\label{r:sec_prelim}

In this section we define the system $\mbox{PRED2}_0$ of first-order
many-sorted intuitionistic predicate logic with second-order
propositional quantifiers, together with its (simplified) Kripke
semantics. We also briefly recapitulate the definition of full models
for a system of classical higher-order logic ${\mbox{PRED}\omega}^c$.

\begin{definition}\label{r:def_pred2}
  The system $\mbox{PRED}\omega$ of higher-order intutionistic logic
  is defined as follows.
  \begin{itemize}
  \item The \emph{types} are given by
    \[
    \Tc \;\; ::= \;\; o \;|\; \B \;|\; \Tc \rightarrow \Tc
    \]
    where $\B$ is a specific finite set of base types. The type $o$ is
    the type of propositions.
  \item The set of terms of $\mbox{PRED}\omega$ of type $\tau$,
    denoted $T_\tau$, is defined by the following grammar, where for
    each type $\tau$ the set $V_\tau$ is a countable set of variables
    and $\Sigma_\tau$ is a countable set of constants.
    \begin{eqnarray*}
      T_{\tau} &::=& V_{\tau} \;|\; \Sigma_{\tau} \;|\; T_{\sigma\rightarrow\tau} \cdot T_{\sigma}
      \mathrm{\ for\ } \sigma \in \Tc \;|\; \lambda V_{\tau_1}
      . T_{\tau_2} \mathrm{\ if\ } \tau = \tau_1\to\tau_2 \\
      T_{o} &::=& V_{o} \;|\; \Sigma_{o} \;|\; T_{\tau\rightarrow o}
      \cdot T_{\tau} \mathrm{\ for\ } \tau \in \Tc \;|\; T_{o} \supset T_{o} \;|\;
      \forall V_{\tau} . T_{o} \mathrm{\ for\ } \tau \in \Tc
    \end{eqnarray*}
    Terms of type $o$ are called \emph{formulas}.
  \item We identify $\alpha$-equivalent formulas, i.e., formulas
    differing only in the names of bound variables are considered
    identical.
  \item Every variable $x$ has an associated unique type, i.e., there
    is exactly one $\tau$ such that $x \in V_\tau$. We sometimes use
    the notation $x_\tau$ for a variable such that $x_\tau \in
    V_\tau$.
  \item The system $\mbox{PRED}\omega$ is given by the following rules
    and an axiom, where $\Delta$ is a finite set of formulas,
    $\varphi, \psi$ are formulas. The notation $\Delta, \varphi$ is a
    shorthand for $\Delta \cup \{\varphi\}$.

    \medskip

    {\bf Axiom}\\
    \(
    \Delta, \varphi \proves \varphi
    \)
    \medskip

    {\bf Rules}

    \begin{longtable}{cc}
      \(
      {\supset_i:}\; \inferrule{\Delta, \varphi \proves
        \psi}{\Delta \proves \varphi \supset \psi}
      \)
      &
      \(
      {\supset_e:}\; \inferrule{\Delta \proves \varphi \supset
        \psi \\ \Delta \proves \varphi}{\Delta \proves \psi}
      \)
      \\
      & \\
      \(
        {\forall_i:}\; \inferrule{\Delta \proves
          \varphi}{\Delta \proves \forall x_\tau . \varphi} \; x_\tau \notin FV(\Delta)
      \)
      &
      \(
        {\forall_e:}\; \inferrule{\Delta \proves \forall x_\tau . \varphi}{\Delta
          \proves \varphi[x_\tau/t]}\; t \in T_\tau
      \)
      \\
      & \\
      \multicolumn{2}{c}{
        \(
        {\mathrm{conv}:}\; \inferrule{\Delta \proves \varphi
          \\ \varphi =_{\beta\eta} \psi}{\Delta
          \proves \psi}
        \)
      }
    \end{longtable}
  \end{itemize}

  The classical variant ${\mbox{PRED}\omega}^c$ is defined by adding to
  $\mbox{PRED}\omega$ the law of double negation as an axiom
  \[
  \Delta \proves \left(\left(\varphi \supset \bot\right) \supset \bot\right)
  \supset \varphi
  \]
  where $\bot \equiv \forall x_o . x_o$ and $x_o \in V_o$.

  The system $\mbox{PRED2}_0$ is the fragment of second-order
  many-sorted predicate calculus restricted to formulas in which
  second-order quantifiers are only propositional. It is obtained from
  $\mbox{PRED}\omega$ by dropping the rule $\mbox{conv}$, restricting
  the types to
  \[
  \Tc \;\; ::= \;\; o \;|\; \B \;|\; \B \rightarrow \Tc
  \]
  and changing the definition of terms to
  \begin{eqnarray*}
    T_{\tau} &::=& V_{\tau} \;|\; \Sigma_{\tau} \;|\; T_{\sigma\rightarrow\tau} \cdot T_{\sigma}
    \;\; \mathrm{for\ all\ } \tau \in \Tc,\, \sigma \in \B \\
    T_{o} &::=& V_{o} \;|\; \Sigma_{o} \;|\;
    T_{\sigma\rightarrow o} \cdot T_{\sigma} \;|\; T_{o} \supset T_{o} \;|\; \forall V_{\tau} . T_{o}
    \mathrm{\ for\ } \tau \in \B \cup \{o\},\, \sigma \in \B
  \end{eqnarray*}
\end{definition}

For an arbitrary set~$\Delta$ we write $\Delta \proves_S \varphi$
if~$\varphi$ is derivable from a subset of~$\Delta$ in system~$S$. We
drop the subscript when obvious or irrelevant. Note that we trivially
have weakening with this definition, i.e., if $\Delta \proves \varphi$
then $\Delta' \proves \varphi$ for any $\Delta' \supseteq \Delta$.

In the rest of this section we assume a fixed set of base types and
fixed sets of constants~$\Sigma_\tau$ for each type~$\tau \in \Tc$. We
assume~$\Tc$,~$T_\tau$, etc. to refer either to~$\mbox{PRED}\omega$
or~$\mbox{PRED2}_0$, depending on the context.

The systems contain only~$\supset$ and~$\forall$ as logical
operators. However, it is well-known that all other connectives may be
defined from these with the help of the second-order propositional
universal quantifier.

We denote by~$t[x/t']$ a term obtained from~$t$ by simultaneously
substituting all free occurences of~$x$ with~$t'$.

\begin{definition}
  A \emph{full model} for ${\mbox{PRED}\omega}^c$ is a pair
  \[
  \M = \langle \{\D_\tau \;|\; \tau \in \Tc \}, I \rangle
  \]
  where each $\D_\tau$ is a nonempty set for $\tau \in \B$, $\D_o =
  \{\top, \bot\}$, each $\D_{\tau_1\to\tau_2}$ is the set of all
  functions from $\D_{\tau_1}$ to $\D_{\tau_2}$, and $I$ is a function
  mapping constants of type $\tau$ to $\D_\tau$. The interpretation
  function $\valuation{}{}{}$ and the satisfaction relation $\models$
  are defined in the standard way. It is well-known and easy to show
  that $\Delta \proves_{{\mathrm{PRED}\omega}^c} \varphi$ implies
  $\Delta \models \varphi$.
\end{definition}

The rest of this section is devoted to introducing a simplified
variant of Kripke semantics for~$\mbox{PRED2}_0$ and proving it sound
and complete. The development is mostly but not completely standard.

\begin{definition}\label{r:def_kripke_pred2}
  A \emph{Kripke pre-model} of $\mbox{PRED2}_0$ is a tuple
  \[
  \M = \langle \Sc, \le, \{\D_\tau \;|\; \tau \in \Tc \}, \cdot, I,
  \varsigma \rangle
  \]
  where~$\Sc$ is a set of~\emph{states}, $\le$ is a partial order on
  $\Sc$, the set $\D_\tau$ is the domain for type $\tau$, the function
  $\cdot$ is a binary application operation, $I$ is an interpretation
  of constants, and $\varsigma$ is a function assigning upward-closed
  (w.r.t. $\le$) subsets of $\Sc$ to elements of $\D_{o}$. A set $X
  \subseteq \Sc$ is upward-closed w.r.t. $\le$ when for all $s_1, s_2
  \in \Sc$, if $s_1 \in X$ and $s_1 \le s_2$, then $s_2 \in X$ as
  well. We sometimes write $\varsigma_\M$, $\Sc_\M$, etc., to stress
  that they are components of $\M$. Furthermore, the following
  conditions are imposed on a Kripke pre-model:
  \begin{itemize}
  \item $\D_\tau$ is nonempty for any $\tau$,
  \item for any $d_1 \in \D_{\tau_1\rightarrow\tau_2}$ and $d_2 \in
    \D_{\tau_1}$ we have $d_1 \cdot d_2 \in \D_{\tau_2}$,
  \item $I(c) \in \D_\tau$ for any $c \in \Sigma_\tau$.
  \end{itemize}

  A valuation is a function that, for all types~$\tau$, maps~$V_\tau$
  into~$D_\tau$.  When we want to stress that a valuation is
  associated with a structure $\M$, we call it an
  $\M$-valuation. If~$u$ is a valuation, $d \in D_\tau$ and $x_\tau$
  is a variable of type~$\tau$, then by~$u[x_\tau/d]$ we denote a
  valuation~$u'$ such that $u'(y) = u(y)$ for $y \ne x_\tau$ and
  $u(x_\tau) = d$. For a given structure~$\M$ and an
  $\M$-valuation~$u$, an interpretation~$\valuation{}{\M}{u}$
  (sometimes abbreviated by~$\valuation{}{}{}$) is a function mapping
  terms of type~$\tau$ to~$\D_\tau$, and satisfying the following:
  \begin{itemize}
  \item $\valuation{x}{}{u} = u(x)$ for a variable $x$,
  \item $\valuation{c}{}{u} = I(c)$ for $c \in \Sigma_\tau$,
  \item $\valuation{t_1 t_2}{}{u} = \valuation{t_1}{}{u} \cdot
    \valuation{t_2}{}{u}$.
  \end{itemize}

  For a formula $\varphi$, a state $s$ and a valuation $u$ we write
  $s,u \forces_\M \varphi$ if $s \in
  \varsigma(\valuation{\varphi}{\M}{u})$. Given a set of formulas
  $\Delta$, we use the notation $s, u \forces_\M \Delta$ if $s, u
  \forces_\M \varphi$ for all $\varphi \in \Delta$. We drop the
  subscript $\M$ when obvious or irrelevant.

  A \emph{Kripke model} is a Kripke pre-model $\M$ satisfying the
  following for any state~$s$ and any valuation~$u$:
  \begin{itemize}
  \item $s, u \forces \varphi \supset \psi$ iff for all $s' \ge s$
    such that $s', u \forces \varphi$ we have $s', u \forces \psi$,
  \item $s, u \forces \forall x_\tau . \varphi$ for $x_\tau \in
    V_\tau$ iff for all $s' \ge s$ and all $d \in \D_\tau$ we have
    $s', u[x_\tau/d] \forces \varphi$,
  \item $s, u \notforces \forall p . p$ for $p \in V_o$.
  \end{itemize}

  We write $\Delta \forces \varphi$ if for every Kripke model $\M$,
  every state $s$ of $\M$, and every valuation $u$, the condition $s,
  u \forces_\M \Delta$ implies $s, u \forces_\M \varphi$.
\end{definition}

\begin{remark}
  What we call Kripke semantics is in fact a somewhat simplified
  version of the usual notion. It is not much more than a
  reformulation of the inference rules. There are no conditions for
  connectives other than $\forall$ and $\supset$, so for instance with
  our definition $s, u \forces \varphi \vee \psi$ need not imply $s, u
  \forces \varphi$ or $s, u \forces \psi$, where $\varphi \vee \psi$
  is defined in the standard way as $\forall x_o . (\varphi \supset
  x_o) \supset (\psi \supset x_o) \supset x_o$. We also assume
  constant domains.\footnote{A reader concerned by this is invited to
    invent an infinite Kripke model (as defined in
    Definition~\ref{r:def_kripke_pred2}) falsifying the Grzegorczyk's
    scheme $\forall x (\psi \vee \varphi(x)) \supset \psi \vee \forall
    x \varphi(x)$. This scheme is not intuitionistically valid, but
    holds in all models with constant domains, in the usual
    semantics.} The resulting notion of a model is quite syntactic,
  which allows us to simplify the usual completeness proof
  considerably.

  Another peculiarity is the presence of the function~$\varsigma$. It
  may seem superfluous, but it is necessary in the Kripke semantics
  for illative systems in Section~\ref{r:sec_illat} where we do not
  know \emph{a priori} which terms represent propositions. For the
  sake of uniformity we already introduce it here.
\end{remark}

\begin{lemma}\label{r:lem_repres_0}
  If $\M$ is a Kripke model, $x \in V_\tau$, $t_0 \in T_\tau$, $t \in
  T_{\tau'}$ and $\tau' \ne o$, then:
  \[
  \valuation{t[x/t_0]}{\M}{u} = \valuation{t}{\M}{u'}
  \]
  where we use the notation $u'$ for $u[x/\valuation{t_0}{}{u}]$.
\end{lemma}

\begin{proof}
  Straightforward induction on the size of $t$.
\end{proof}

\begin{lemma}\label{r:lem_repres}
  If $\M$ is a Kripke model, $x \in V_\tau$, $t \in T_\tau$, $\varphi \in
  T_{o}$ and $u' = u[x/\valuation{t}{}{u}]$, then for all
  states~$s$:
  \[
  s, u \forces \varphi[x/t] \;\;\;\mathrm{iff}\;\;\; s, u' \forces
  \varphi
  \]
\end{lemma}

\begin{proof}
  We proceed by induction on the size of $\varphi$. If $\varphi$ is a
  constant, a variable, or $\varphi = t_1 t_2$, then the claim follows
  from Lemma~\ref{r:lem_repres_0}.

  Assume $\varphi = \varphi_1 \supset \varphi_2$. Suppose $s, u
  \forces \varphi_1[x/t] \supset \varphi_2[x/t]$ and let $s' \ge s$ be
  such that $s', u' \forces \varphi_1$. By the IH we have $s', u
  \forces \varphi_1[x/t]$, hence $s', u \forces
  \varphi_2[x/t]$. Applying the IH again we obtain $s', u' \forces
  \varphi_2$. This implies that $s, u' \forces \varphi$. The other
  direction is analogous.

  Assume $\varphi = \forall y . \varphi_0$. Without loss of generality
  $y \ne x$ and $y \notin FV(t)$. Suppose $s, u \forces \forall y
  . \varphi_0[x/t]$, and let $s' \ge s$ and $d \in \D_\tau$. We have
  $s', u[y/d] \forces \varphi_0[x/t]$. By the IH we obtain $s',
  u'[y/d] \forces \varphi_0$. This implies $s, u' \forces \forall y
  . \varphi_0$. The other direction is analogous.
\end{proof}

\begin{theorem}
  The conditions $\Delta \forces \varphi$ and $\Delta \proves \varphi$
  are equivalent.
\end{theorem}

\begin{proof}
  By induction on the length of derivation we first show that $\Delta
  \proves \varphi$ implies $\Delta \forces \varphi$. Note that it
  suffices to show this for finite $\Delta$. The implication is
  obvious for the axiom. Assume $\Delta \proves \varphi$ was obtained
  by rule $\forall_i$. Then $\varphi = \forall x . \psi$ for $x \in
  V_\tau$, $x \notin FV(\Delta)$. Let $\M,s,u$ be such that $s,u
  \forces_\M \Delta$. Hence for all $s' \ge s$ we have $s',u
  \forces_\M \Delta$, and $s',u[x/d] \forces_\M \Delta$ for any $d \in
  \D_\tau$ because $x \notin FV(\Delta)$. So by the inductive
  hypothesis we obtain $s',u[x/d] \forces_\M \psi$ for any $d \in
  \D_\tau$. By the definition of a Kripke model, this implies $s,u
  \forces_\M \forall x . \psi$. The remaining cases are equally
  straightforward. Lemma~\ref{r:lem_repres} is needed for the
  rule~$\forall_e$.

  To prove the other direction, we assume that $\Delta_0 \notproves
  \varphi_0$ and construct a Kripke model~$\M$ and a valuation~$u$
  such that for some state~$s$ of~$\M$ we have $s,u \forces_\M
  \Delta_0$, but $s, u \notforces_\M \varphi_0$.

  First, without loss of generality, we assume that there are
  infinitely many variables not occuring in the formulas
  of~$\Delta_0$. We can do this because extending the language with
  infinitely many new variables is conservative. The states of $\M$
  are consistent sets of formulas $\Delta' \supseteq \Delta_0$, i.e.,
  $\Delta' \not\proves \bot$, which differ from~$\Delta_0$ by only
  finitely many formulas. The ordering is by inclusion. For any
  type~$\tau$ as~$\D_\tau$ we take the set of terms of
  type~$\tau$. Let~$v$ be a valuation. Given a term~$t$, we denote
  by~$t^v$ a term obtained from~$t$ by simultaneously substituting any
  variable $x \in FV(t)$ by the term $v(x)$. We obviously assume that
  no variables are captured in these substitutions, which is possible
  because we treat formulas up to $\alpha$-equivalence. We define the
  interpretation~$I$ by $I(c) = c$. We also set $t_1 \cdot t_2 = t_1
  t_2$. Notice that now $\valuation{t}{}{v} = t^v$. Further, we define
  the function~$\varsigma$ of~$\M$ as follows: $\varsigma(\varphi) =
  \{ \Delta \;|\; \Delta \proves \varphi \}$ for a formula~$\varphi$,
  where~$\Delta$ ranges over sets of formulas which are valid
  states. Note that $\Delta, v \forces_\M \varphi$ is now equivalent
  to $\Delta \proves \varphi^v$. Finally, we set $u(x) = x$.

  Given a formula~$\phi$, a state~$\Delta$, and a valuation~$v$, we
  show by induction on the size of~$\phi$ that \mbox{$\Delta, v
    \forces_\M \phi$} satisfies the conditions required for a Kripke
  model. If $\phi = \varphi \supset \psi$, then we need to check that
  \mbox{$\Delta \proves \varphi^v \supset \psi^v$} iff for all
  $\Delta' \supseteq \Delta$ such that~$\Delta'$ is a valid state and
  $\Delta' \proves \varphi^v$, we have $\Delta' \proves
  \psi^v$. Suppose the right side holds and take $\Delta' = \Delta
  \cup \{ \varphi^v \}$. If~$\Delta'$ is a valid state then $\Delta'
  \proves \psi^v$, hence by rule~$\supset_i$ we obtain $\Delta \proves
  \varphi^v \supset \psi^v$. Because~$\Delta$ extends~$\Delta_0$ by
  finitely many formulas, so does~$\Delta'$. Hence if~$\Delta'$ is not
  a valid state, then it is inconsistent. Then obviously $\Delta'
  \proves \psi^v$ anyway, so we again obtain the left side by applying
  rule~$\supset_i$. The other direction follows by
  applying~$\supset_e$ and weakening finitely many times.

  Similarly, if $\psi = \forall x . \varphi$, then without loss of
  generality we assume $v(x) = x$, $x \in V_\tau$, and check that
  $\Delta \proves \forall x . \varphi^v$ iff for all valid states
  $\Delta' \supseteq \Delta$ and all $t_1 \in \D_\tau$ we have
  $\Delta' \proves \varphi^{v'}$ where $v' = v[x/t_1]$. If the right
  side of the equivalence holds, then it holds in particular for $t_1
  = y$ such that $y \notin FV(\Delta, \varphi^v)$, and $\Delta' =
  \Delta$. Such~$y$ exists, because we have assumed an infinite number
  of variables not occuring in the formulas of~$\Delta_o$,
  and~$\Delta$ extends~$\Delta_o$ by only finitely many formulas. By
  rule~$\forall_i$ we obtain $\Delta \proves \forall_y \varphi^v$,
  which is $\alpha$-equivalent to the left side, and we treat
  $\alpha$-equivalent formulas as identical. Conversely, if $\Delta
  \proves \forall x . \varphi^v$, then by rule~$\forall_e$ and
  weakening we obtain $\Delta' \proves \varphi^v[x/t_1]$. This is
  equivalent to $\Delta' \proves \varphi^{v'}$ where $v' = v[x/t_1]$.

  It is now a matter of routine to check that~$\M$ is a Kripke
  model. Obviously, in this model we have $\Delta_0, u \notforces
  \varphi_0$, i.e., $\Delta_0 \notin \valuation{\varphi_0}{}{u} =
  \varsigma(\varphi_0)$, because $\Delta_0 \notproves \varphi_0$. On
  the other hand, $\Delta_0, u \forces \psi$ for every $\psi \in
  \Delta_0$. This proves the theorem.
\end{proof}

\section{Illative systems}\label{r:sec_illat}

In this section we define the higher-order illative systems
$\I_\omega$, $\I_\omega^c$ and the second-order illative system
$\I_0$. We also define a semantics for these systems.

\begin{definition}\label{r:def_illat}
  By $\T(\Sigma)$ we denote the set of type-free lambda-terms over
  some specific set $\Sigma$ of primitive constants, which is assumed
  to contain $\Xi$, $L$ and $A_\tau$ for each $\tau \in \B$ where $\B$
  is some specific set of base types.

  We use the following abbreviations. The term $\supset$ is usually
  written in infix notation and is assumed to be right-associative.
  \begin{eqnarray*}
    I &=& \lambda x . x \\
    S &=& \lambda x y z . x z (y z) \\
    K &=& \lambda x y . x \\
    H &=& \lambda x . L (K x) \\
    \supset &=& \lambda x y . \Xi (K x) (K y) \\
%%    \bot &=& \Xi H I \\
%%    \top &=& \bot \supset \bot \\
    F &=& \lambda x y f . \Xi x \left(\lambda z . y \left(f
    z\right)\right)
%%     \wedge &=& \lambda x y . \Xi H \lambda z . (x \supset y \supset z)
%%     \supset z \\
%%     \vee &=& \lambda x y . \Xi H \lambda z . (x \supset z) \supset (y
%%     \supset z) \supset z \\
    %\nabla &=& \lambda x y . \Xi H \lambda z . \Xi x \lambda v . (y v
    %\supset z) \supset z
  \end{eqnarray*}

  The constant~$\Xi$ functions as a restricted quantification
  operator, i.e., $\Xi A B$ is intuitively interpreted as $\forall x
  . A x \supset B x$. The intended interpretation of~$L A$ is ``$A$ is
  a type'', or ``$A$ may be a range of quantification''. The term~$H$
  stands for the ``type'' of propositions, and $F A B$ denotes the
  ``type'' of functions from~$A$ to~$B$. The constants $A_\tau$ denote
  base types, i.e., different sorts of individuals. We use a notion of
  types informally in this section.

  For systems of illative combinatory logic, judgements have the form
  $\Gamma \proves t$ where $\Gamma$ is a finite subset of $\T(\Sigma)$
  and $t \in \T(\Sigma)$. The notation $\Gamma, t$ is an abbreviation
  for $\Gamma \cup \{t\}$.

  \medskip

  The system $\I_\omega$ is defined by the following axioms and rules.
  \medskip

  {\bf Axioms}
  \begin{enumerate}
  \item $\Gamma, t \proves t$
  \item $\Gamma \proves L H$
  \item $\Gamma \proves L A_\tau$ for $\tau \in \B$
  \end{enumerate}

  {\bf Rules}
  %%\begin{center}
  \begin{longtable}{cc}
    \(
    {\mbox{Eq}:}\; \inferrule{\Gamma \proves t_1 \\ t_1 =_{\beta\eta} t_2}{\Gamma \proves t_2}
    \)
    &
    \(
    {H_i:}\; \inferrule{\Gamma \proves t}{\Gamma \proves H t}
    \)
    \\
    & \\
    \multicolumn{2}{l}{
      \(
      {\Xi_e:}\; \inferrule{\Gamma \proves \Xi t_1 t_2 \\ \Gamma \proves t_1 t_3}{\Gamma \proves t_2 t_3}
      \)
    }
    \\
    & \\
    \multicolumn{2}{l}{
      \(
      {\Xi_i:}\; \inferrule{\Gamma, t_1 x \proves t_2 x \\ \Gamma \proves L t_1}{\Gamma
        \proves \Xi t_1 t_2}\; x \notin FV(\Gamma, t_1, t_2)
      \)
    }
    \\
    & \\
    \multicolumn{2}{l}{
      \(
      {\Xi_H:}\; \inferrule{\Gamma, t_1 x \proves H(t_2 x) \\ \Gamma \proves L
        t_1}{\Gamma \proves H(\Xi t_1 t_2)}\; x \notin FV(\Gamma, t_1, t_2)
      \)
    }
    \\
    & \\
    \multicolumn{2}{l}{
      \(
      {F_L:}\; \inferrule{\Gamma, t_1 x \proves L t_2 \\ \Gamma
        \proves L t_1}{\Gamma \proves L (F t_1 t_2)}\; x \notin
      FV(\Gamma, t_1, t_2)
      \)
    }
  \end{longtable}
  %%\end{center}

  The system $\I_\omega^c$ is $\I_\omega$ plus the axiom of double
  negation:
  \[
  \Gamma \proves \Xi H \left(\lambda x . \left(\left(x \supset \bot\right)
  \supset \bot\right) \supset x\right)
  \]
  where $\bot = \Xi H I$.\footnote{Note that here the symbol~$\bot$ is
    an abbreviation for a term in the syntax of~$\I_\omega$, which is
    distinct from previous uses of~$\bot$.}

  The system $\I_0$ is $\I_\omega$ minus the rule $F_L$. The rule
  $F_L$ allows us to quantify over functions and
  predicates. Obviously, the system becomes more useful if for $\tau
  \in \B$ we can add constants~$c$ representing some elements of
  type~$\tau$, axioms $A_\tau c$, and some axioms of the form e.g. $p
  (f c_1) (g c_2)$ where $f$, $g$ are constants representing functions
  and $p$ is a predicate constant (i.e.~of type $\tau_1\to\tau_2\to
  o$). That most such simple extensions are consistent with
  $\I_\omega$ is a consequence of the model construction in
  Section~\ref{r:sec_construction}.

  For an arbitrary set $\Gamma$, we write $\Gamma \proves_{\I} t$ if
  there is a finite subset $\Gamma' \subseteq \Gamma$ and a derivation
  of $\Gamma' \proves t$ in an illative system $\I$. The subscript is
  dropped when obvious from the context.
\end{definition}

\begin{lemma}\label{r:lem_illat_admissible}
  The following rules are admissible in $\I_\omega$ and $\I_0$.
  \begin{longtable}{cc}
    \(
    {P_e:}\; \inferrule{\Gamma \proves t_1 \supset t_2 \\ \Gamma
      \proves t_1}{\Gamma \proves t_2}
    \)
    &
    \(
    {P_i:}\; \inferrule{\Gamma, t_1 \proves t_2 \\ \Gamma \proves H
      t_1 }{\Gamma \proves t_1 \supset t_2}
    \)
    \\
    & \\
      \(
      {P_H:}\; \inferrule{\Gamma, t_1 \proves H t_2 \\ \Gamma \proves
        H t_1}{\Gamma \proves H (t_1 \supset t_2)}
      \)
      &
      \(
      {\mathrm{Weak}:}\; \inferrule{\Gamma \proves t}{\Gamma, t' \proves t}
      \)
  \end{longtable}
\end{lemma}

\begin{proof}
  Routine.
\end{proof}

%% \begin{lemma}\label{r:lem_var_subst}
%%   If $\Gamma \proves t$ then $\Gamma[x/t'] \proves t[x/t']$ for any
%%   term $t'$.
%% \end{lemma}

%% \begin{proof}
%%   Straightforward induction on the length of derivation.
%% \end{proof}

\begin{definition}
  A \emph{combinatory algebra} $\C$ is a tuple $\langle C, \cdot, S, K
  \rangle$, where $\cdot$ is a binary operation in $C$ and $S, K \in
  C$, such that for any $X, Y, Z \in C$ we have:
  \begin{itemize}
  \item $S \cdot X \cdot Y \cdot Z = (X \cdot Z) \cdot (Y \cdot Z)$,
  \item $K \cdot X \cdot Y = X$.
  \end{itemize}
  To save on notation we often write $X \in \C$ instead of $X \in
  C$. We assume $\cdot$ associates to the left, and sometimes omit it.

  A combinatory algebra is \emph{extensional} if for any $M_1, M_2 \in
  \C$, whenever for all $X \in \C$ we have $M_1 X = M_2 X$, then we
  also have $M_1 = M_2$.

  It is well-known that any combinatory algebra contains a fixed-point
  combinator and satisfies the principle of combinatory abstraction,
  so any equation of the form $z \cdot x = \Phi(z, x)$, where $\Phi(z,
  x)$ is an expression involving the variables $z$, $x$ and some
  elements of~$\C$, has a solution for $z$ satifying this equation for
  arbitrary $x$.
\end{definition}

\begin{definition}\label{r:def_ikm}
  An \emph{illative Kripke pre-model} for an illative system~$\I$ ($\I
  \in \{\I_\omega, \I_\omega^c, \I_0\}$) with primitive
  constants~$\Sigma$, is a tuple $\langle \Sc, \le, \C, I, \varsigma
  \rangle$, where $\Sc$ is a set of states, $\le$ is a partial order
  on the states, $\C$ is an extensional combinatory algebra, $I :
  \Sigma \rightarrow \C$ is an interpretation of primitive constants,
  and $\varsigma$ is a function assigning upward-closed (w.r.t. $\le$)
  subsets of~$\Sc$ to elements of~$\C$. We sometimes write
  $\sigma_\M$, $\Sc_\M$, etc., to stress that they are components of
  $\M$.

  Given an illative Kripke pre-model $\M$, the value
  $\valuation{t}{\M}{u}$ of term $t$ under valuation $u$, which is a
  function from variables to $\C$, is defined inductively:
  \begin{itemize}
  \item $\valuation{x}{}{u} = u(x)$ for a variable $x$,
  \item $\valuation{c}{}{u} = I(c)$ for a constant $c$,
  \item $\valuation{t_1 t_2}{}{u} = \valuation{t_1}{}{u} \cdot
    \valuation{t_2}{}{u}$,
  \item $\valuation{\lambda x . t}{}{u}$ is the element $d \in \C$
    satisfying $d \cdot d' = \valuation{t}{}{u'}$ for any $d' \in
    \C$, where $u' = u[x/d']$.
  \end{itemize}
  Note that the element in the last point is uniquely defined because
  of extensionality and combinatorial completeness of $\C$.

  To save on notation, we often confuse $\Xi$, $L$, etc. with
  $\valuation{\Xi}{\M}{u}$, $\valuation{L}{\M}{u}$, etc. The intended
  meaning is always clear from the context. The subscript $\M$ is also
  often dropped.

  Intuitively, for $X \in \C$ the set $\varsigma(X)$ is the set of all
  states~$s$ such that the element~$X$ is true in~$s$. The
  relation~$\le$ on states is analogous to an accessibility relation
  in a Kripke frame.

  An \emph{illative Kripke model} for $\I_\omega$ is an illative
  Kripke pre-model where $\varsigma$ satisfies the following conditions
  for any $X, Y \in \C$:
  \begin{enumerate}
  \item if $s \in \varsigma(L X)$ and for all $s' \ge s$ and all $Z \in
    \C$ such that $s' \in \varsigma(X Z)$ we have $s' \in \varsigma(Y
    Z)$, then $s \in \varsigma(\Xi X Y)$, \label{r:ikm_1}
  \item if $s \in \varsigma(\Xi X Y)$ then for all $Z \in \C$ such that
    $s \in \varsigma(X Z)$ we have $s \in \varsigma(Y Z)$, \label{r:ikm_2}
  \item if $s \in \varsigma(L X)$ and for all $s' \ge s$ and all $Z \in
    \C$ such that $s' \in \varsigma(X Z)$ we have $s' \in \varsigma(H(Y
    Z))$, then $s \in \varsigma(H(\Xi X Y))$, \label{r:ikm_3}
  \item if $s \in \varsigma(L X)$ and for all $s' \ge s$ such that $s' \in
    \varsigma(X Z)$ for some $Z \in \C$, we have $s' \in \varsigma(L Y)$, then
    $s \in \varsigma(L(F X Y))$, \label{r:ikm_4}
  \item if $s \in \varsigma(X)$ then $s \in \varsigma(H X)$, \label{r:ikm_5}
  \item $s \in \varsigma(L H)$, \label{r:ikm_6}
  \item $s \in \varsigma(L A_\tau)$ for $\tau \in \B$. \label{r:ikm_7}
  \end{enumerate}
  An illative Kripke model for $\I_0$ is defined analogously, but
  omitting condition~(\ref{r:ikm_4}). A model is a \emph{classical
    illative model} if it satisfies the law of double negation: if $s
  \in \varsigma(H X)$ and $s \in \varsigma((X \supset \bot) \supset
  \bot)$ then $s \in \varsigma(X)$, where $\bot = \Xi H I$. It is not
  difficult to see that every one-state illative Kripke model is a
  classical illative model. For a classical illative model with a
  single state $s$ we define the set $\Ts$ of \emph{true elements} by
  $\Ts = \{ X \in \C \;|\; s \in \varsigma(X) \}$. Note that
  $\varsigma(X)$ may be empty.

  For a term $t$ and a valuation~$u$, we write $s, u \forces_{\M} t$
  whenever $s \in \varsigma(\valuation{t}{\M}{u})$. For a set of terms
  $\Gamma$, we write $\Gamma \forces_{\I} t$ if for all Kripke models
  $\M$ of an illative system $\I$, all states $s$ of $\M$, and all
  valuations $u$ such that $s,u \forces_{\M} t'$ for all $t' \in
  \Gamma$, we have $s,u \forces_{\M} t$. Note that $s,u \forces_{\M}
  t$ implies $s',u \forces_{\M} t$ for $s' \ge s$, because $\varsigma(X)$
  is always an upward-closed subset of $\Sc$, for any argument $X$.
\end{definition}

Informally, one may think of illative Kripke models as combinatory
algebras with an added structure of a Kripke frame.

\begin{fact}\label{r:fact_supset_illat}
  In any illative Kripke model the following conditions are satisfied:
  \begin{enumerate}
  \item if $s \in \varsigma(H X)$ and for all $s' \ge s$ such that $s'
    \in \varsigma(X)$ we have $s' \in \varsigma(Y)$, then $s \in \varsigma(X
    \supset Y)$, \label{r:cond_supset_illat_01}
  \item if $s \in \varsigma(X \supset Y)$ then $s \in \varsigma(X)$ implies
    $s \in \varsigma(Y)$, \label{r:cond_supset_illat_02}
  \item if $s \in \varsigma(H X)$ and for all $s' \ge s$ such that $s'
    \in \varsigma(X)$ we have $s' \in \varsigma(H Y)$, then $s \in
    \varsigma(H(X \supset Y))$. \label{r:cond_supset_illat_03}
  \end{enumerate}
\end{fact}

%\vspace{-1.5em}

\begin{theorem}\label{r:thm_ikm_complete}
  The conditions $\Gamma \forces_{\I} t$ and $\Gamma \proves_{\I} t$
  are equivalent, where $\I = \I_\omega$ or $\I = \I_0$.
\end{theorem}

\begin{proof}
  We first check that $\Gamma \proves_{\I} t$ implies $\Gamma
  \forces_{\I} t$, by a simple induction on the length of
  derivation. It suffices to prove this for finite $\Gamma$. The
  implication is immediate for the axioms. Now assume $\Gamma \proves
  t_2 t$ was obtained by rule $\Xi_e$, and we have $s,u \forces_{\M}
  \Gamma$. Hence, by the inductive hypothesis $s,u \forces_{\M} \Xi
  t_1 t_2$ and $s,u \forces_\M t_1 t$, which by
  condition~(\ref{r:ikm_2}) in Definition~\ref{r:def_ikm} implies $s,u
  \forces_{\M} t_2 t$. Assume $\Gamma \proves \Xi t_1 t_2$ was
  obtained by rule~$\Xi_i$, and that $s,u \forces_{\M} \Gamma$. Let
  $s' \ge s$ and $Z \in \C$ be such that $s' \in
  \varsigma(\valuation{t_1}{\M}{u} \cdot Z)$. We therefore have $s', u'
  \forces_{\M} \Gamma, t_1 x$, where $u' = u[x/Z]$ and $x \notin
  FV(\Gamma, t_1, t_2)$. So by the inductive hypothesis we obtain $s',
  u' \forces_{\M} t_2 x$. Because $x \notin FV(t_2)$, this is
  equivalent to $s' \in \varsigma(\valuation{t_2}{\M}{u} \cdot Z)$. The
  inductive hypothesis implies also that $s \in \varsigma(L \cdot
  \valuation{t_1}{\M}{u})$. We therefore obtain by
  condition~(\ref{r:ikm_1}) in Definition~\ref{r:def_ikm} that $s, u
  \forces_{\M} \Xi t_1 t_2$. The other cases are equally
  straightforward and we leave them to the reader. In the case of
  rule~\mbox{Eq} the extensionality of~$\C$ is needed.

  To prove the other direction, we assume $\Gamma_0 \notproves_{\I}
  t_0$, and construct an illative Kripke model $\M$ and a valuation
  $u$ such that for some state $s$ of $\M$ we have $s, u \forces_{\M}
  \Gamma_0$, but $s,u \notforces_{\M} t_0$.

  We construct the model as follows. First of all, we assume without
  loss of generality that there are infinitely many variables not
  occuring in~$\Gamma_0$. As states we take all sets of terms $\Gamma'
  \supseteq \Gamma_0$ which extend~$\Gamma_0$ by only finitely many
  formulas. The ordering is by inclusion. The combinatory algebra $\C$
  is the set of equivalence classes of $\beta\eta$-equality on
  $\T(\Sigma)$. We denote the equivalence class of a term $t$ by
  $[t]_{\beta\eta}$. We define $I(c) = [c]_{\beta\eta}$ for $c \in
  \Sigma$. The function~$\varsigma$ is defined by the condition:
  $\Gamma \in \varsigma([t]_{\beta\eta})$ iff $\Gamma \proves_{\I} t$
  and $\Gamma$ is a valid state. This is well-defined because of
  $\beta\eta$-equality in rule~\mbox{Eq}. The valuation~$u$ is defined
  by $u(x) = [x]_{\beta\eta}$. Note that $\valuation{t}{\M}{u} =
  [t]_{\beta\eta}$.

  We now show that this is an illative Kripke model. We only need to
  check the conditions on $\varsigma$. It is obvious that
  $\varsigma(X)$ is upward-closed for any $X \in \C$ because of
  weakening. Assume that $\Gamma \proves L t_1$, and for all $\Gamma'
  \supseteq \Gamma$ and all terms $t_3$ such that $\Gamma' \proves t_1
  t_3$ we have $\Gamma' \proves t_2 t_3$. Then, in particular, this
  holds for $\Gamma' = \Gamma \cup \{t_1 x\}$ and $t_3 = x$, where $x$
  is a variable, $x \notin FV(\Gamma, t_1, t_2)$. Such a variable~$x$
  exists because~$\Gamma$ differs from~$\Gamma_0$ by only finitely
  many formulas, and there are infinitely many variables not occuring
  in the formulas of~$\Gamma_0$. Therefore, by rule $\Xi_i$ we have
  $\Gamma \proves \Xi t_1 t_2$, hence $\Gamma \in \varsigma([\Xi t_1
    t_2]_{\beta\eta})$. This verifies
  condition~(\ref{r:ikm_1}). Conditions~(\ref{r:ikm_2}),
  (\ref{r:ikm_3}), (\ref{r:ikm_4}) and (\ref{r:ikm_5}) are verified in
  a similar manner, using rules $\Xi_e$, $\Xi_H$, $F_L$ and $H_i$,
  respectively. Condition~(\ref{r:ikm_6}) is immediate from the axiom
  $\Gamma \proves L H$. Condition~(\ref{r:ikm_7}) follows from the
  axioms $\Gamma \proves L A_\tau$ for $\tau \in \B$.

  It is obvious that $\Gamma_0, u \notforces_{\M} t_0$, i.e., $\Gamma_0
  \notin \varsigma([t_0]_{\beta\eta})$, because $\Gamma_0
  \notproves_{\I} t_0$. Clearly, we also have $\Gamma_0, u
  \forces_{\M} t$ for all $t \in \Gamma_0$. This proves the theorem.
\end{proof}

\begin{remark}
  Note one subtlety here. The above theorem does not imply that $\I_0$
  or $\I_\omega$ is consistent. This is because we allow trivial
  Kripke models, i.e., ones such that $\varsigma(X) = \Sc$ for any $X
  \in \C$, and it is not obvious that nontrivial ones exist. Indeed,
  if we dropped the restriction $s \in \varsigma(L X)$ in
  condition~(\ref{r:ikm_1}) in Definition~\ref{r:def_ikm}, then all
  illative Kripke models would be trivial. To see this, let $X \in \C$
  and $s \in \Sc$ be arbitrary and consider the element $\Y \in \C$
  defined by the equation $\Y = \Y \supset X$. Note that dropping $s
  \in \varsigma(L X)$ in condition~(\ref{r:ikm_1}) in
  Definition~\ref{r:def_ikm} means dropping $s \in \varsigma(H X)$ in
  condition~(\ref{r:cond_supset_illat_01}) in
  Fact~\ref{r:fact_supset_illat}. For any $s' \ge s$ we obviously have
  $s' \in \varsigma(\Y \supset X)$ whenever $s' \in \varsigma(\Y)$. By
  condition~(\ref{r:cond_supset_illat_02}) in
  Fact~\ref{r:fact_supset_illat} we conclude that $s' \in
  \varsigma(X)$ whenever $s' \in \varsigma(\Y)$. Therefore, by
  condition~(\ref{r:cond_supset_illat_01}) in
  Fact~\ref{r:fact_supset_illat}, we have $s \in
  \varsigma(\Y)$. Hence, $s \in \varsigma(\Y \supset X)$ as well, so
  again $s \in \varsigma(X)$. Thus $\varsigma(X) = \Sc$. This argument
  is essentially Curry's paradox.
\end{remark}

%% We also conjecture that $\I_\omega^c$ is complete w.r.t one-state
%% classical illative models, but the proof would be more difficult than
%% the completeness proof given above. (This is OK, we don't have
%% boolean extensionality here).

For convenience of reference we state the following simple fact about
one-state classical illative models for $\I_\omega^c$, as we will be
constructing such a model in the next section. Recall that for a
classical illative model with a single state~$s$, the set~$\Ts$ of
true elements is defined by $\Ts = \{ X \in \C \;|\; s \in
\varsigma(X) \}$.

\begin{fact}\label{r:fact_classical_illat_model}
  For a one-state classical illative model for $\I_\omega^c$ the
  conditions on $\varsigma$ may be reformulated as follows:
  \begin{enumerate}
  \item if $L X \in \Ts$ and for all $Z \in \C$ such that $X Z \in \Ts$
    we have $Y Z \in \Ts$, then $\Xi X Y \in \Ts$, \label{r:cikm_1}
  \item if $\Xi X Y \in \Ts$ then for all $Z \in \C$ such that $X Z \in
    \Ts$ we have $Y Z \in \Ts$, \label{r:cikm_2}
  \item if $L X \in \Ts$ and for all $Z \in \C$ such that $X Z \in \Ts$
    we have $H(Y Z) \in \Ts$, then $H(\Xi X Y) \in
    \Ts$, \label{r:cikm_3}
  \item if $L X \in \Ts$, and either $L Y \in \Ts$ or there is no $Z
    \in \C$ such that $X Z \in \Ts$, then $L(F X Y) \in
    \Ts$, \label{r:cikm_4}
  \item if $X \in \Ts$ then $H X \in \Ts$, \label{r:cikm_5}
  \item $L H \in \Ts$, \label{r:cikm_6}
  \item $L A_\tau \in \Ts$ for $\tau \in \B$. \label{r:cikm_7}
  \end{enumerate}
\end{fact}

\section{The model construction}\label{r:sec_construction}

In this section we construct a model for~$\I_\omega^c$. The
construction is parametrized by a full model for classical
higher-order logic.

\subsection{Definitions}

In this subsection we give definitions necessary for the construction
and fix some notational conventions.

\begin{ssdefinition}
  We define the set of \emph{types} $\Tc^+$ by the following grammar:
  \begin{eqnarray*}
  \Tc^+ &::=& \Tc_1 \;|\; \omega \;|\; \varepsilon \\
  \Tc_1 &::=& \Tc \;|\; \Tc_1 \rightarrow \Tc_1 \;|\;
  \omega \rightarrow \Tc_1 \\
  \Tc &::=& o \;|\; \B \;|\; \Tc \rightarrow \Tc
  \end{eqnarray*}
  where $\B$ is a specific finite set of base types. Intuitively, the
  type $o$ is the type of propositions, $\omega$ is the type of
  arbitrary objects, $\varepsilon$ is the empty type.

  For the sake of simplicity we use the following notational
  convention: we sometimes write $\tau\rightarrow\varepsilon$ for
  $\varepsilon$ when $\tau \ne \varepsilon$,
  $\varepsilon\rightarrow\tau$ for $\omega$, and
  $\tau\rightarrow\omega$ for $\omega$. There is never any ambiguity
  because $\tau\rightarrow\varepsilon$ etc. are not valid types
  according to the grammar for~$\Tc^+$. This convention is only to
  shorten some statements later on. We also use the abbreviation
  $\tau_1^n\to\tau_2$ for $\tau_1\to\ldots\to\tau_1\to\tau_2$ where
  $\tau_1$ occurs $n$ times (possibly $n = 0$).\hfill $\Box$
\end{ssdefinition}

% Note that this definition of types is just a definition of some
% syntactic objects. We do not assume \emph{a priori} anything about
% their meaning or their connection to the informal notion of types used
% in the previous section.

From now on we fix a full model $\N = \langle \{\D_\tau \;|\; \tau \in
\Tc \}, I \rangle$ of classical higher-order logic and construct a
one-state classical illative model~$\M$ for~$\I_\omega^c$. We assume
that $\Tc \subset \Tc^+$ defined above corresponds exactly to the
types of~$\N$, and that the base types~$\B$ correspond exactly to the
base types used in the definition of the syntax of~$\I_\omega^c$.

We will define the universe of the model as the set of equivalence
classes of a certain relation on the set of type-free lambda-terms
over a set~$\Sigma^+$ of primitive constants, to be defined below. We
assume these terms to be different objects than the terms of the
syntax of~$\I_\omega^c$. We also treat lambda-terms up to
$\alpha$-equivalence, i.e., terms differing only in the names of bound
variables are considered identical.

\begin{ssdefinition}\label{r:def_canonical}
  We define a set of primitive constants $\Sigma^+$, and a set of
  \emph{canonical terms} as follows. First, for every type~$\tau \ne
  \omega$ we define by induction on the size of~$\tau$ a set of
  \emph{canonical terms of type~$\tau$}, denoted by~$\T_\tau$. We also
  define a set of constants~$\Sigma_\tau$ for every type~$\tau \notin
  \{\omega, \varepsilon\} \cup \{\omega\rightarrow\tau' \;|\; \tau'
  \in \Tc^+ \}$, i.e., we leave $\Sigma_\tau$ undefined if $\tau$ is
  not of the form required. First, we set $\T_\varepsilon =
  \emptyset$. In the inductive step we consider possible forms
  of~$\tau$. If~$\tau \in \Tc$ (i.e.~it does not contain~$\omega$
  or~$\varepsilon$) then we define~$\Sigma_\tau$ to contain a unique
  constant for every element~$d \in \D_\tau$. We set $\T_\tau =
  \Sigma_\tau$. If $\tau \notin \Tc$, $\tau = \tau_1\rightarrow\tau_2$
  and $\tau_1\ne\omega$, then denote by $\Sigma_\tau$ a set of new
  constants for every (set-theoretical) function from $\T_{\tau_1}$ to
  $\T_{\tau_2}$. Again we set $\T_{\tau} = \Sigma_\tau$. If $\tau =
  \omega\rightarrow\tau_2$ then $\T_\tau$ consists of all terms of the
  form $\lambda x . \rho$ where $\rho \in
  \T_{\tau_2}$.\footnote{Formally, terms are $\alpha$-equivalence
    classes of certain strings, i.e., by $\lambda x . \rho$ we mean the
    $\alpha$-equivalence class of the string $"\lambda x . \rho"$, so
    e.g. $\lambda x . \rho$ and $\lambda y . \rho$ are the same, which
    we denote by $\lambda x . \rho \equiv \lambda y . \rho$.}

  The symbol $\Sigma^A$ stands for a set consisting of distinct new
  constants $A_\tau$ for each base type $\tau \in \B$. Finally, we set
  $\Sigma^+ = \{\Xi, L\} \cup \Sigma^A \cup \bigcup_\tau \Sigma_\tau$
  where the index in the sum ranges over $\tau \notin \{\omega,
  \varepsilon\} \cup \{ \omega \to \tau' \;|\; \tau' \in \Tc^+
  \}$. For the sake of uniformity, we use the notation $\T_\omega$ for
  the set of all type-free lambda terms over~$\Sigma^+$. Note that
  terms in $\T_\omega$ are not necessarily canonical and all canonical
  terms are closed.

  Note that for $\tau \in \Tc$ the set~$\Sigma_\tau$ contains a unique
  constant for every element of~$\D_\tau$. Hence for each $\tau \in
  \Tc$ there is a natural bijection from~$\Sigma_\tau$
  onto~$\D_\tau$. We denote this bijection by~$\pi_\tau$.

  We now define a mapping~$\F$ such that for $\rho \in
  \T_{\tau_1\to\tau_2}$ we have $\F(\rho) :
  \T_{\tau_1}\to\T_{\tau_2}$, where $\tau_1\to\tau_2 \in \Tc_1$. If
  $\tau_1\to\tau_2 \in \Tc$ then $\tau_1, \tau_2 \in \Tc$,
  $\T_{\tau_1} = \Sigma_{\tau_1}$, and both~$\pi_{\tau_1}$
  and~$\pi_{\tau_2}$ are defined. In this case we set $\F(c)(c_1) =
  \pi_{\tau_2}^{-1}(\pi_{\tau_1\to\tau_2}(c)(\pi_{\tau_1}(c_1)))$ for
  $c \in \Sigma_{\tau_1\to\tau_2}$, $c_1 \in \Sigma_{\tau_1}$. If
  $\tau_1\to\tau_2 \notin \Tc$ and $\tau_1 \ne \omega$ then also
  $\T_{\tau_1\to\tau_2} = \Sigma_{\tau_1\to\tau_2}$ and by our
  construction to each $c \in \Sigma_{\tau_1\to\tau_2}$ corresponds a
  set-theoretical function~$f_c$ from~$\T_{\tau_1}$
  to~$\T_{\tau_2}$. In this case we set $\F(c) = f_c$. Finally, if
  $\rho \in \T_{\omega\rightarrow\tau}$ then $\rho = \lambda x
  . \rho'$ and by~$\F(\rho)$ we denote the constant function
  from~$\T_\omega$ to~$\T_\tau$ whose value is always $\rho'$. Note
  that because~$\N$ is assumed to be a full model, so by our
  construction if~$\tau_1\to\tau_2 \in \Tc_1$ and $\tau_1 \ne \omega$
  then for every set-theoretical function~$f$ from~$\T_{\tau_1}$
  to~$\T_{\tau_2}$ there exists a constant~$\rho_f \in
  \Sigma_{\tau_1\to\tau_2}$ such that $\F(\rho_f) = f$.

  By $\top \in \Sigma_o$ we denote the constant corresponding to the
  element $\top \in \D_o$, by $\bot \in \Sigma_o$ the one
  corresponding to $\bot \in \D_o$. Note that $\Sigma_o = \{\top,
  \bot\}$, because $\D_o = \{\top,\bot\}$.

  Note that if $\tau_1, \tau_2 \ne \omega$ and $\tau_1 \ne \tau_2$
  then $\T_{\tau_1} \cap \T_{\tau_2} = \emptyset$. Hence every
  canonical term $\rho$ may be assigned a unique type $\tau \ne
  \omega$ such that $\rho \in \T_\tau$. When talking about the
  \emph{canonical type}, or simply \emph{the type}, of a canonical
  term we mean the type thus defined.\hfill $\Box$
\end{ssdefinition}

An \emph{$n$-ary context} $C$ is a lambda-term over the set of
constants $\Sigma^+ \cup \{\Box_1,\ldots,\Box_n\}$, where
$\Box_1,\ldots,\Box_n \notin \Sigma^+$. The constants
$\Box_1,\ldots,\Box_n$ are the \emph{boxes} of $C$. If~$C$ is an
$n$-ary context then by $C[t_1,\ldots,t_n]$ we denote the term $C$
with all occurences of~$\Box_i$ replaced with~$t_i$ for
$i=1,\ldots,n$. Unless otherwise stated, we assume that the free
variables of $t_1,\ldots,t_n$ \emph{do not become bound} in
$C[t_1,\ldots,t_n]$. By a \emph{context} we usually mean a unary
context, unless otherwise qualified. In this case we write~$\Box$
instead of~$\Box_1$.

In what follows $\alpha$, $\beta$, etc. stand for ordinals; $t$,
$t_1$, $t_2$, $r$, $r_1$, $r_2$, $q$, $q_1$, $q_2$ etc. stand for
type-free lambda-terms over~$\Sigma^+$ from which we build the model;
$c$, $c_1$, $c_2$, etc. stand for constants from~$\Sigma^+$; $\tau$,
$\tau_1$, $\tau_2$, etc. stand for types; $\rho$, $\rho_1$, $\rho_2$
stand for canonical terms (i.e.~terms $\rho \in \T_\tau$ for
$\tau\ne\omega$); and $C$, $C'$, $C_1$, $C_2$, etc. denote contexts;
unless otherwise qualified.

The following simple fact states some easy properties of canonical
terms. It will sometimes be used implicitly in what follows.

\begin{ssfact}\label{r:fact_canonical}
  If $\rho$ is a canonical term then:
  \begin{enumerate}
  \item $\rho \equiv \lambda x_1 \ldots x_n . c$ where $n \ge 0$, $c
    \in \Sigma_\tau$ for some $\tau$ (so $\tau \ne \omega\to\tau_1$),
    and $\rho \in \T_{\omega^n\to\tau}$, \label{r:canon_1}
  \item if $\rho \equiv C[t]$ then either $C \equiv \lambda x_1 \ldots
    x_k . \Box$ and $t$ is a canonical term, or $C \equiv
    \rho$. \label{r:canon_2}
  \end{enumerate}
\end{ssfact}

For each ordinal~$\alpha$ we inductively define reduction
systems~$R_\alpha$ and~$\widehat{R_\alpha}$, a relation~$\sim_\alpha$
between terms and types in~$\Tc^+$, and a
relation~$\succ_\alpha$ between terms and canonical
terms. Formally, all these notions are defined by one induction in a
mutually recursive way, but we split up the definitions for the sake
of readability. These definitions are monotone with respect
to~$\alpha$, so the induction closes at some ordinal, i.e., the
relations do not get larger after this ordinal.

First, let us fix some notations. We write~$R_{<\alpha}$
for~$\bigcup_{\beta<\alpha}R_{\beta}$, $\succ_{<\alpha}$
for~$\bigcup_{\beta<\alpha}\succ_\beta$, $\sim_{<\alpha}$
for~$\bigcup_{\beta<\alpha}\sim_\beta$. We use the notation $\equiv$
for identity of terms up to
$\alpha$-equivalence. By~$\contr_{\le\alpha}$ we denote the reduction
relation of~$R_\alpha$, by~$\contr_{\le\alpha}^\equiv$ the reflexive
closure of~$\contr_{\le\alpha}$, by~$\reduces_{\le\alpha}$ the
transitive reflexive closure of~$\contr_{\le\alpha}$, and
by~$=_{\le\alpha}$ the transitive reflexive symmetric closure. We
write $[t]_\alpha$ for the equivalence class of a term $t$ w.r.t. the
relation~$=_{\le\alpha}$. Analogously, we use the
subscript~$_{<\alpha}$ for relations corresponding to $R_{<\alpha}$,
and~$_{=\alpha}$ for relations corresponding to
$\widehat{R_{\alpha}}$. We drop the subscripts when they are obvious
or irrelevant.

\begin{ssnotation}\label{r:convention_kt}
  In what follows a term of the form $K t$ should be read as $\lambda
  x . t$ where $x \notin FV(t)$, a term $H t$ as $L \lambda x . t$
  where $x \notin FV(t)$, and $F t_1 t_2$ as $\lambda f . \Xi t_1
  (\lambda x . t_2 (f x))$. We adopt this convention to shorten
  notations.\hfill $\Box$
\end{ssnotation}

Before embarking on the task of rigorously constructing the model we
explain the intuitive meaning of various notions formally introduced
later. This is necessarily informal and at points rather vague.

Informally speaking, we identify types with sets of terms. A base type
corresponds to the set of all constants of this type, the type~$o$ to
the set of all propositions, the type~$\omega$ to the set of all
terms, the type $\varepsilon$ to the empty set, and a function
type~$\tau_1\to\tau_2$ to the set of all terms~$t$ such that for all
terms~$t_1$ of type~$\tau_1$ the term $t t_1$ has type~$\tau_2$. It is
known at the beginning of the transfinite inductive definition exactly
which terms have base types, but not so for type~$o$ or function
types. During the course of the induction new terms may obtain
types. If~$r$ is a term, and~$\alpha$ an ordinal, then by $r
\leadsto_\alpha \top$ we mean that at stage~$\alpha$ in the induction,
$r$ has been shown to be ``true''. If $r \equiv F A_{\tau_1}
A_{\tau_2} t$, we interpret this as saying that, at stage~$\alpha$ in
the induction, the term~$t$ has been shown to have type
$\tau_1\to\tau_2$. It may be that for all $\beta < \alpha$ we may have
$F A_{\tau_1} A_{\tau_2} t \not\leadsto_\beta \top$, yet $F A_{\tau_1}
A_{\tau_2} t \leadsto_\alpha \top$. So the fact that~$t$ has
type~$\tau_1\to\tau_2$ becomes known only at stage~$\alpha$ of the
induction. Our induction stops when no new typings may be obtained and
no new terms may become true or false, i.e., when we have all
information we need to construct the model.

Note that canonical terms may obtain types different from their
canonical types. For instance, a term of the form $\lambda x . c$
where $c \in \Sigma_\tau$ will ultimately obtain the type~$\omega$ and
all of the types~$\tau'\to\tau$ for any type~$\tau'$. As far as
canonical terms are concerned, we mostly care about their canonical
types, and it is known beforehand what types these are.

In $R_\alpha$ we will have reduction rules of $\beta$- and
$\eta$-reduction, and rules of the form $c \rho \to \F(c)(\rho)$,
where $c \in \Sigma_{\tau_1\to\tau_2}$ and $\rho \in \T_{\tau_1}$. We
will also add some other rules to make certain terms
``indistinguishable'', as explained in the paragraph below.

Intuitively, $t \succ_\alpha \rho$ is intended to hold if~$\rho \in
\T_\tau$ is a ``canonical'' term which is ``equivalent'' to~$t$ in
type~$\tau$, basing on the information we have at stage $\alpha$. Let
us give some examples to elucidate what we mean by this. For instance,
suppose we have two distinct (hence disjoint) base types~$\tau_1$
and~$\tau_2$, and two functions $\mathrm{Id}_{\tau_1\to\tau_1} \in
\D_{\tau_1\to\tau_1}$ and $\mathrm{Id}_{\tau_2\to\tau_2} \in
\D_{\tau_2\to\tau_2}$ which are identities on~$\D_{\tau_1}$
and~$\D_{\tau_2}$ respectively. In~$\Sigma^+$ we will have two
canonical constants $\mathrm{id}_{\tau_1\to\tau_1}$ and
$\mathrm{id}_{\tau_2\to\tau_2}$ of type~$\tau_1\to\tau_1$ and
$\tau_2\to\tau_2$ respectively, associated with the functions
$\mathrm{Id}_{\tau_1\to\tau_1}$ and $\mathrm{Id}_{\tau_2\to\tau_2}$,
i.e., such that $\F(\mathrm{id}_{\tau_1\to\tau_1}) = \pi_{\tau_1}^{-1}
\circ \mathrm{Id}_{\tau_1\to\tau_1} \circ \pi_{\tau_1}$ and
$\F(\mathrm{id}_{\tau_2\to\tau_2}) = \pi_{\tau_2}^{-1} \circ
\mathrm{Id}_{\tau_2\to\tau_2} \circ \pi_{\tau_2}$. The reduction rules
associated with $\mathrm{id}_{\tau_1\to\tau_1}$ will be
$\mathrm{id}_{\tau_1\to\tau_1} c \to c$ for every canonical
constant~$c$ \emph{of~type~$\tau_1$}, and analogously for
$\mathrm{id}_{\tau_2\to\tau_2}$. Note that
$\mathrm{id}_{\tau_1\to\tau_1} c$ will not form a redex if~$c$ is a
canonical constant of type different from~$\tau_1$. Now we have both
$\lambda x . x \succ_1 \mathrm{id}_{\tau_1\to\tau_1}$ and $\lambda x
. x \succ_1 \mathrm{id}_{\tau_2\to\tau_2}$, because $\lambda x . x$
behaves exactly like $\mathrm{id}_{\tau_1\to\tau_1}$ when given
arguments of type~$\tau_1$, and exactly like
$\mathrm{id}_{\tau_2\to\tau_2}$ when given arguments of
type~$\tau_2$. In fact, we will define the reduction systems
$R_\alpha$ so as to make $\lambda x . x$ and
$\mathrm{id}_{\tau_1\to\tau_1}$ indistinguishable, for sufficiently
large~$\alpha$, wherever a term of type~$\tau_1\to\tau_1$ is
``expected''. For instance, for any reduction rule in $R_\alpha$ of
the form $\rho\, \mathrm{id}_{\tau_1\to\tau_1} \to c$, where~$\rho$ is
a canonical term of type~$(\tau_1\to\tau_1)\to\tau$ for some~$\tau$,
we will add to $R_{\alpha+1}$ a reduction rule $\rho\, (\lambda x . x)
\to c$.

In the case $\rho \in \{\top, \bot\}$, the relation $t \succ_\alpha
\rho$ encompasses a definition of truth. The condition $t \succ_\alpha
\top$ means that~$t$ is certainly true, basing on the information from
the earlier stages $\beta < \alpha$ of the inductive definition. So if
$t \succ_\alpha \top$ then $t$ should behave like $\top$ wherever a
truth-value is expected. If $t \succ_\alpha \bot$, then $t$ is
certainly not true.

If $t \ne \rho$ then we never have $t \succ_\alpha \rho$ for a
canonical term~$\rho$ of some base type~$\tau \in \B$, because no term
different from~$\rho$ behaves like $\rho$ if the type of~$\rho$ is an
atomic type different from~$o$.

\begin{ssnotation}
  We use the notation $t \leadsto_{\alpha} \rho$ when $t
  \reduces_{\le\alpha} t' \succ_\alpha \rho$. We
  write~$\leadsto_{<\alpha}$
  for~$\bigcup_{\beta<\alpha}\leadsto_\beta$.
\end{ssnotation}

Informally, $t \leadsto_{\alpha} \rho$ holds if we can reduce $t$,
using the rules of $R_\alpha$, to a term equivalent to a canonical
term~$\rho$ in the type of~$\rho$ basing on what we know at stage
$\alpha$ of the inductive definition. A careful reader will notice
that what we ultimately really care about is the relation
$\leadsto_{\alpha}$, not $\succ_\alpha$, because we want to identify
$R_\alpha$-equivalent terms. The relation $\succ_\alpha$ is needed
chiefly to facilitate the proofs.

The condition $t \sim_\alpha \tau$ is intended to hold if $t$
``represents'' the type~$\tau$ basing on what we know at
stage~$\alpha$, i.e., it is a ``predicate'' which is true when applied
to terms of type~$\tau$, and is never true when applied to terms which
are not of type~$\tau$. In other words, $L t \leadsto_\alpha \top$ and
for all terms~$r$ known to be of type~$\tau$ we have $t r
\leadsto_\alpha \top$, but we should not have $t r \leadsto_\alpha
\top$ for any~$r$ which is not of type~$\tau$. So for instance for
each type~$\tau \in \B$ we should have $A_\tau \sim_\alpha \tau$ for
sufficently large~$\alpha$. Because~$\varepsilon$ is the empty type,
if $t \sim_\alpha \varepsilon$ then we should never have $t r
\leadsto_{<\alpha} \top$ for any term~$r$. Since~$\omega$ is the type
of arbitrary objects we should have $t \sim_\alpha \omega$ if for all
terms~$r$ we have $t r \leadsto_{<\alpha} \top$.

Having explained the intuitive meaning of the relations, we may
proceed to formal definitions. The definition below depends on the
definition of~$\succ_{<\alpha}$, and thus on~$\succ_\beta$ for~$\beta
< \alpha$.

\begin{ssdefinition}\label{r:def_reductions}
  A \emph{reduction system} is a set of reduction rules over a
  specified set of terms, i.e., a set of pairs of terms.  In all
  reduction systems we consider we assume the set of terms to be the
  type-free lambda-terms over $\Sigma^+$. Instead of writing $\langle
  t_1,t_2\rangle \in R$ we usually say that $t_1 \to t_2$ is a
  reduction rule of~$R$. Given a reduction system~$R$ we define its
  associated \emph{reduction relation} $\contr_R$ by: $t_1 \contr_R
  t_2$ iff there exists a context~$C$ with exactly one box and
  terms~$r_1, r_2$ such that $t_1 \equiv C[r_1]$, $t_2 \equiv C[r_2]$
  and $r_1 \to r_2$ is a rule of~$R$. In contrast to all subsequent
  uses of contexts, here we allow the free variables of $r_1$ and
  $r_2$ to become bound in $C[r_1]$ or $C[r_2]$.

  We define $\widehat{R_\alpha}$ to contain the following reduction
  rules:
  \begin{itemize}
  \item for $\alpha = 0$: rules of $\beta$- and $\eta$-reduction,
  \item for $\alpha > 0$: rules $c t \rightarrow \rho_2$ for
    every $c \in \Sigma_{\tau_1\rightarrow\tau_2}$ (so $\tau_1 \ne
    \omega$), every $\rho_2 \in \T_{\tau_2}$ and every term $t$ such
    that $t \succ_{<\alpha} \rho_1$ and $\F(c)(\rho_1) \equiv \rho_2$.
  \end{itemize}
  We set $R_\alpha = R_{<\alpha} \cup \widehat{R_\alpha}$.\hfill
  $\Box$
\end{ssdefinition}

\begin{ssdefinition}\label{r:def_sim}
  The relation $\sim_\alpha$ is defined by the following rules. Recall
  that $\tau_1\rightarrow\varepsilon = \varepsilon$ for $\tau_1 \ne
  \varepsilon$, $\varepsilon\rightarrow\tau_2 = \omega$, and
  $\tau_1\rightarrow\omega = \omega$.

  \begin{longtable}{cc}
    \( {(\mathrm{A}):}\; \inferrule{\tau \in \B}{A_\tau \sim_\alpha
      \tau} \) & \( {(\mathrm{H}):}\; \inferrule{ }{H \sim_\alpha o}
    \)
    \\
    & \\
    \( {(\mathrm{K\omega}):}\; \inferrule{t \leadsto_{<\alpha} \top}{K
      t \sim_\alpha \omega} \) & \( {(\mathrm{K\varepsilon}):}\;
    \inferrule{t \leadsto_{<\alpha} \bot}{K t \sim_\alpha \varepsilon}
    \)
    \\
    & \\
    \multicolumn{2}{c}{
      \(
      {(\mathrm{F}):}\; \inferrule{t_1 \sim_{<\alpha} \tau_1 \\
        t_2 \sim_{<\alpha} \tau_2}
      {F t_1 t_2 \sim_\alpha \tau_1\rightarrow\tau_2} \)
    }
    \\
    & \\
    \multicolumn{2}{c}{
      \(
      {(\mathrm{F'}):}\; \inferrule{t_1 \sim_{<\alpha} \tau_1 \\
        \lambda z . t_2 \sim_{<\alpha} \tau_2 \\
        f,x \notin FV(t_1,t_2)}
      {\lambda f . \Xi t_1 (\lambda x . t_2[z/f x]) \sim_\alpha
        \tau_1\rightarrow\tau_2} \)
    }
    \\
    & \\
    \multicolumn{2}{c}{
      \(
      {(\mathrm{F''}):}\; \inferrule{t_1 \sim_{<\alpha} \tau_1 \\
        t_2 \sim_{<\alpha} \tau_2 \in \{ \omega, \varepsilon \} \\ f \notin FV(t_1,t_2)}
      {\lambda f . \Xi t_1 t_2 \sim_\alpha \tau_1\rightarrow\tau_2} \)
    }
    \\
    & \\
    \(
    {(\mathrm{F\omega}):}\; \inferrule{t_1 \sim_{<\alpha}
      \varepsilon}
    {\lambda f . \Xi t_1 t_2 \sim_\alpha \omega} \) &
    \(
    {(\mathrm{F\omega'}):}\; \inferrule{t_1 \sim_{<\alpha}
      \varepsilon}
    {\Xi t_1 \sim_\alpha \omega} \)
  \end{longtable}
\end{ssdefinition}

The above definition depends on the definitions of~$R_\beta$,
$\sim_\beta$ and~$\succ_\beta$ for~$\beta < \alpha$. The next
definition of~$\succ_\alpha$ depends on the definitions of~$R_\beta$
and~$\sim_\beta$ for~$\beta \le \alpha$, and on~$\succ_{<\alpha}$.

\begin{ssdefinition}\label{r:def_succ}
  We define the relation $t \succ_\alpha \rho$ for canonical terms
  $\rho$ by the following conditions:
  \begin{itemize}
  \item $\rho \succ_\alpha \rho$ if the canonical type of~$\rho$
    is~$o$ or a base type,
  \item $t \succ_\alpha \rho$ if the canonical type of $\rho$ is
    $\tau_1\to\tau_2$ and $t$ is a term such that for any $t_1 \in
    \T_{\tau_1}$ we have $t t_1 \leadsto_{<\alpha}
    \F(\rho)(t_1)$. Note that we allow $\tau_1 = \omega$ but not
    $\tau_1 = \varepsilon$.
  \end{itemize}

  In particular, $\top \succ_\alpha \top$ and $\bot \succ_\alpha \bot$
  by the above definition. For $\rho \in \{\top,\bot\}$ we give
  additional postulates. For $\alpha \ge 0$ we postulate $t
  \succ_\alpha \top$ for all terms $t$ such that at least one of the
  following holds:
  \begin{enumerate}
  \item[$(A_\tau^\top)$] $t \equiv A_{\tau} c$ where $\tau \in \B$ and
    $c \in \Sigma_\tau$,
  \item[$(\Xi^\top)$] $t \equiv \Xi t_1 t_2$ where $t_1$, $t_2$ are
    terms such that there exists $\tau$ s.t. $t_1 \sim_\alpha \tau$
    and for all $t_3 \in \T_\tau$ we have $t_2 t_3 \leadsto_{<\alpha}
    \top$,
  \item[$(L^\top)$] $t \equiv L t_1$ and $t_1 \sim_{\alpha} \tau$ for
    some type~$\tau$.
  \end{enumerate}
  Finally, when $\alpha \ge 0$ we postulate $t \succ_\alpha \bot$ for
  all terms $t$ such that:
  \begin{enumerate}
  \item[$(\Xi^\bot)$] $t \equiv \Xi t_1 t_2$ and there exists a type
    $\tau$ such that:
    \begin{itemize}
    \item $t_1 \sim_\alpha \tau$, and
    \item for every term $t_3 \in \T_\tau$ we have $t_2 t_3
      \leadsto_{<\alpha} \top$ or $t_2 t_3 \leadsto_{<\alpha} \bot$,
    \item there exists a term $t_3 \in \T_\tau$ with $t_2 t_3
      \leadsto_{<\alpha} \bot$.
    \end{itemize}
  \end{enumerate}
\end{ssdefinition}

The intuitive interpretation of $\Xi t_1 t_2$ is restricted
quantification $\forall x . t_1 x \supset t_2 x$, but $t_1$ is
required to represent a type, if $\Xi t_1 t_2$ is to have a logical
value. In illative combinatory logic the notions of being
(representing) a type and being eligible to stand as a quantifier
range are equivalent. It turns out that the types of $\I_\omega^c$ are
just the types defined by $\Tc^+$. This explains putting $t_1
\sim_\alpha \tau$ in some of the cases above.

During the course of the transfinite inductive definition some
previously untyped terms $t$ will obtain types, e.g. a statement of
the form $F A_{\tau_1} A_{\tau_2} t$ will become true at some stage
$\alpha$. At that point we need to decide which term among the
canonical terms of type $\tau_1\to\tau_2$ behaves exactly
like~$t$. The whole correctness proof rests on the fact that this
decision is always possible. That we may choose such a canonical term
implies that quantifying over only canonical terms of a certain
type~$\tau$ is equivalent to quantifying over all terms of
type~$\tau$. This justifies restricting quantification to canonical
terms in the above definition of $t \succ_\alpha \top$.

\medskip

Let us now give some examples illustrating the above definitions.

\begin{ssexample}
  Suppose we have a base type~$\tau$ and~$\widehat{\mathrm{Id}} \in
  \D_{\tau\to\tau}$ is the identity function on~$\D_\tau$. Let
  $\mathrm{Id} = \pi_{\tau}^{-1} \circ \widehat{\mathrm{Id}} \circ
  \pi_{\tau}$, i.e., $\mathrm{Id}(c) = c$ for any $c \in
  \Sigma_\tau$. There is a constant $\mathrm{id} \in
  \Sigma_{\tau\to\tau}$ such that $\F(\mathrm{id}) = \mathrm{Id}$. We
  show $\lambda x . x \succ_1 \mathrm{id}$. Let $c \in \Sigma_\tau =
  \T_\tau$. We have $(\lambda x . x) c \contr_{<1} c$, because $R_{<1}
  = \bigcup_{n<1}R_n = R_0$ and~$R_0$ contains the rules of
  $\beta$-reduction. We also have $c \succ_{<1} c \equiv
  \F(\mathrm{id})(c)$ by the first part of
  Definition~\ref{r:def_succ}. Therefore $(\lambda x . x) c
  \leadsto_{<1} c$. Since $c \in \Sigma_\tau$ was arbitrary, we obtain
  $\lambda x . x \succ_{1} \mathrm{id}$ by the second part of
  Definition~\ref{r:def_succ}.

  Now we show that $\lambda y x . x \succ_2 \lambda y
  . \mathrm{id}$. We have $\lambda y . \mathrm{id} \in
  \T_{\omega\to\tau\to\tau}$. So let $t \in \T_\omega$. We have
  $(\lambda y x . x) t \contr_{<2} \lambda x . x$. We already proved
  that $\lambda x . x \succ_{1} \mathrm{id}$. Note that $\F(\lambda y
  . \mathrm{id})$ is the constant function from~$\T_\omega$
  to~$\T_{\tau\to\tau}$ whose value is always~$\mathrm{id}$. This
  implies that $(\lambda y x. x) t \leadsto_{<2} \mathrm{id} \equiv
  \F(\lambda y . \mathrm{id})(t)$ for any $t \in \T_\omega$. Hence
  $\lambda y x. x \succ_2 \lambda y . \mathrm{id}$.

  Let $\rho \in \Sigma_{((\omega\to\tau)\to\tau)\to\tau}$ be such that
  $\F(\rho)(f) \equiv \F(f)(\lambda y . \mathrm{id})$ for $f \in
  \Sigma_{(\omega\to\tau)\to\tau}$. As another example we will show
  that $\lambda z . z (\lambda y x . x) \succ_4 \rho$. So suppose $f
  \in \Sigma_{(\omega\to\tau)\to\tau} =
  \T_{(\omega\to\tau)\to\tau}$. We have $(\lambda z . z (\lambda y x
  . x)) f \contr_{<4} f (\lambda y x . x)$. We proved in the previous
  paragraph that $\lambda y x . x \succ_2 \lambda y . \mathrm{id}$. By
  the second part of Definition~\ref{r:def_reductions} we obtain $f
  (\lambda y x . x) \contr_{=3} \F(f)(\lambda y . \mathrm{id})$. Hence
  $(\lambda z . z (\lambda y x . x)) f \reduces_{<4} \F(f)(\lambda y
  . \mathrm{id})$ for any $f \in
  \Sigma_{(\omega\to\tau)\to\tau}$. Obviously we have $\F(f)(\lambda y
  . \mathrm{id}) \succ_0 \F(f)(\lambda y . \mathrm{id})$ by the first
  part of Definition~\ref{r:def_succ}, because the range of~$\F(f)$ is
  included in~$\Sigma_\tau$. Recalling that $\F(\rho)(f) \equiv
  \F(f)(\lambda y . \mathrm{id})$ for any $f \in
  \Sigma_{(\omega\to\tau)\to\tau}$ we obtain $\lambda z . z (\lambda y
  x . x) \succ_4 \rho$ by the second part of
  Definition~\ref{r:def_succ}.
\end{ssexample}

\begin{sslemma}\label{r:lem_monotonous}
  For $\alpha \le \beta$ we have the following inclusions: $R_{\alpha}
  \subseteq R_{\beta}$, $\sim_\alpha\; \subseteq \;\sim_\beta$,
  and~$\succ_\alpha\; \subseteq \;\succ_\beta$.
\end{sslemma}

\begin{proof}
  Follows easily from definitions.
\end{proof}

It follows from Lemma~\ref{r:lem_monotonous} by appealing to the
well-known Knaster-Tarski fixpoint theorem that there exists an
ordinal~$\zeta$ such that $\succ_\zeta\; = \;\succ_{<\zeta}$ and
$R_\zeta = R_{<\zeta}$. This simple fact may also be shown directly as
follows. Suppose~$\zeta$ is an ordinal with cardinality greater than
$(\T_\omega \cup \{\Box\})^4$ and there is no $\alpha < \zeta$ such
that $R_\alpha = R_{<\alpha}$ and $\succ_\alpha\; =
\;\succ_{<\alpha}$. Then for each $\alpha < \zeta$ either $R_{\alpha}
\setminus R_{<\alpha}$ or $\succ_\alpha \setminus \succ_{<\alpha}$ is
non-empty. Because $R_\alpha \subseteq \T_\omega \times \T_\omega$ and
$\succ_\alpha \subseteq \T_\omega \times \T_\omega$, we may thus
define, using the axiom of choice, an injection~$f$ from~$\zeta$ to
$(\T_\omega \cup \{\Box\})^4$ (recall that in set theory an
ordinal~$\zeta$ is the set of all ordinals less than~$\zeta$). If
$R_\alpha \setminus R_{<\alpha}$ is non-empty, then let $f(\alpha) =
\langle t_1, t_2, \Box, \Box \rangle$ where $\langle t_1, t_2 \rangle
\in R_\alpha \setminus R_{<\alpha}$ is chosen
arbitrarily. Analogously, if $\succ_\alpha \setminus \succ_{<\alpha}$
is non-empty, then let $f(\alpha) = \langle \Box, \Box, t_1, t_2
\rangle$ where $\langle t_1, t_2 \rangle \in\; \succ_\alpha \setminus
\succ_{<\alpha}$ is chosen arbitrarily. Since $R_\alpha \subseteq
R_{<\beta}$ and $\succ_\alpha\;\subseteq\;\succ_{<\beta}$ for $\alpha
< \beta$, we have $f(\alpha) \ne f(\beta)$, so~$f$ really is an
injection. But this implies that the cardinality of~$\zeta$ is not
greater than the cardinality of~$(\T_\omega \cup
\{\Box\})^4$. Contradiction.

Let~$\zeta$ be an ordinal such that $R_\zeta = R_{<\zeta}$ and
$\succ_\zeta\; = \;\succ_{<\zeta}$. We may assume without loss of
generality that also $\sim_\zeta\; = \;\sim_{<\zeta}$. In what follows
we will use the notations $R$, $\succ$, $\leadsto$, etc. for
$R_\zeta$, $\succ_\zeta$, $\leadsto_\zeta$, etc.

\medskip

Finally, we are ready to define the model $\M$ for $\I_\omega^c$.

\begin{ssdefinition}\label{r:def_model}
  The one-state classical illative Kripke model $\M$ is defined as
  follows. We take the combinatory algebra $\C$ of $\M$ to be the set
  of equivalence classes of $=_R$. We define the interpretation $I$ of
  $\M$ by $I(c) = [c]_R$. We define the set $\Ts$ of true elements of
  $\M$ by $\Ts = \{ d \in \C \;|\; \exists t \,.\, d = [t]_R \wedge t
  \leadsto \top \}$.
\end{ssdefinition}

\subsection{Correctness proof}

In this subsection we prove that the preceding lengthy definition of
$\M$ is actually correct, i.e., that $\M$ is a classical illative
Kripke model for $\I_\omega^c$.

Below we will silently use the following simple lemma, without
mentioning it explicitly every time.

\begin{sslemma}\label{r:lem_xi_unique}
  If $\lambda \vec{x} . \Xi t_1 t_2 \reduces \lambda \vec{x} . t$ then
  $t \equiv \Xi t_1' t_2'$ where $t_1 \reduces t_1'$ and $t_2 \reduces
  t_2'$. An analogous result holds when $\lambda \vec{x} . A_\tau t_1
  \reduces \lambda \vec{x} . t$ for $\tau \in \B$, and when $\lambda
  \vec{x} . L t_1 \reduces \lambda \vec{x} . t$. Here the reduction
  $\reduces$ may stand for any of $\reduces_{\le\alpha}$,
  $\reduces_{<\alpha}$, etc.
\end{sslemma}

\begin{proof}
  This follows from the fact that there are no reduction rules which
  involve $\Xi$, $L$, or $A_\tau$ for $\tau \in \B$, so the reductions
  may happen only inside $t_1$ and $t_2$.
\end{proof}

Note that together with our convention stated in
Notation~\ref{r:convention_kt} regarding the meaning of~$H t_1$,
Lemma~\ref{r:lem_xi_unique} implies that if $H t_1 \reduces t$ then $t
\equiv H t_1'$ where $t_1 \reduces t_1'$.

The proof of the following lemma illustrates a pattern common to many
of the proofs below. We give this single proof in full, but when later
an argument follows this same pattern we treat only some of the cases
to spare the reader excessive tedious details.

\begin{sslemma}\label{r:lem_nu_context}
  If $x_1,\ldots,x_n$ are variables, $n \ge 1$, and $C$ is a context,
  then the following conditions hold:
  \begin{enumerate}
  \item if $t \reduces_{\le\alpha} t'$ and $t \equiv C[x_1\ldots x_n]$
    then $t' \equiv C'[x_1 \ldots x_n]$ and $C[t'']
    \reduces_{\le\alpha} C'[t'']$ for any term~$t''$, \label{r:lcsp_1}
  \item if $C[x_1 \ldots x_n] \succ_\alpha \rho$ then $C[t]
    \succ_\alpha \rho$ for any term~$t$, \label{r:lcsp_2}
  \item if $C[x_1 \ldots x_n] \sim_\alpha \tau$ then $C[t] \sim_\alpha
    \tau$ for any term~$t$. \label{r:lcsp_3}
  \end{enumerate}
\end{sslemma}

\begin{proof}
  Induction on $\alpha$.

  First, we show~(\ref{r:lcsp_1}) by induction on the length of the
  reduction $C[x_1 \ldots x_n] \reduces_{\le\alpha} t'$. The only
  interesting case is when $c C[x_1 \ldots x_n] \contr_{\le\alpha}
  \rho_2$ by virtue of $C[x_1 \ldots x_n] \succ_{<\alpha} \rho_1$. But
  then by part~(\ref{r:lcsp_2}) of the IH we have $C[t'']
  \succ_{<\alpha} \rho_1$, so $c C[t''] \contr_{\le\alpha} \rho_2$.

  Next we shall verify~(\ref{r:lcsp_3}). If $C[x_1 \ldots x_n]
  \sim_\alpha \tau$ is obtained by rule~$(\mathrm{A})$
  or~$(\mathrm{H})$, then $C \equiv A_\tau$ for $\tau \in \B$ or $C
  \equiv H$, and the claim is obvious.

  If $C[x_1 \ldots x_n] \sim_\alpha \tau$ is obtained by
  rule~$(\mathrm{K}\omega)$ or~$(\mathrm{K}\varepsilon)$ then $\tau
  \in \{\omega, \varepsilon\}$ and $C[x_1 \ldots x_n]
  \leadsto_{<\alpha} c$ for $c \in \{\top, \bot\}$, i.e., $C[x_1 \ldots
  x_n] \reduces_{<\alpha} t' \succ_{<\alpha} c$. By
  part~(\ref{r:lcsp_1}) of the IH we obtain $t' \equiv C'[x_1 \ldots
  x_n]$ where $C[t] \reduces_{<\alpha} C'[t]$. Then by
  part~(\ref{r:lcsp_2}) of the IH we have $C'[t] \succ_{<\alpha}
  c$. Hence $C'[t] \leadsto_{<\alpha} c$, so $C'[t] \sim_\alpha \tau$
  by rule~$(\mathrm{K}\omega)$ or~$(\mathrm{K}\varepsilon)$.

  If $C[x_1 \ldots x_n] \sim_\alpha \tau$ is obtained by
  rule~$(\mathrm{F}\omega)$ then $\tau = \omega$, $C \equiv \lambda f
  . \Xi C_1 C_2$ and $C_1[x_1\ldots x_n] \sim_{<\alpha}
  \varepsilon$. By part~(\ref{r:lcsp_3}) of the IH we obtain $C_1[t]
  \sim_{<\alpha} \varepsilon$, and thus $C[t] \equiv \lambda f . \Xi
  C_1[t] C_2[t] \sim_\alpha \omega$ by rule~$(\mathrm{F}\omega)$.

  If $C[x_1 \ldots x_n] \sim_\alpha \tau$ is obtained by
  rule~$(\mathrm{F}\omega')$ then $\tau = \omega$, $C \equiv \Xi C_1$
  and $C_1[x_1\ldots x_n] \sim_{<\alpha} \varepsilon$. By
  part~(\ref{r:lcsp_3}) of the IH we obtain $C_1[t] \sim_{<\alpha}
  \varepsilon$, and thus $C[t] \equiv \Xi C_1[t] \sim_\alpha \omega$
  by rule~$(\mathrm{F}\omega')$.

  If $C[x_1 \ldots x_n] \sim_\alpha \tau$ is obtained by
  rule~$(\mathrm{F})$ then $\tau = \tau_1\to\tau_2$ and $C \equiv
  \lambda f . \Xi C_1 (\lambda x . C_2 (f x))$ where $C_1[x_1\ldots
  x_n] \sim_{<\alpha} \tau_1$ and $C_2[x_1 \ldots x_n] \sim_{<\alpha}
  \tau_2$. But then by part~(\ref{r:lcsp_3}) of the~IH we have $C_1[t]
  \sim_{<\alpha} \tau_1$ and $C_2[t] \sim_{<\alpha} \tau_2$, which
  implies $C[t] \equiv \lambda f . \Xi C_1[t] (\lambda x . C_2[t] (f
  x)) \sim_\alpha \tau$.

  If $C[x_1 \ldots x_n] \sim_\alpha \tau$ is obtained by
  rule~$(\mathrm{F'})$ then $\tau = \tau_1\to\tau_2$ and $C \equiv
  \lambda f . \Xi C_1 (\lambda x . C_2[z/f x])$ where $C_1[x_1\ldots
  x_n] \sim_{<\alpha} \tau_1$ and $\lambda z . C_2[x_1 \ldots x_n]
  \sim_{<\alpha} \tau_2$. But then by part~(\ref{r:lcsp_3}) of the~IH
  we have $C_1[t] \sim_{<\alpha} \tau_1$ and $\lambda z . C_2[t] \sim_{<\alpha}
  \tau_2$, which implies $C[t] \equiv \lambda f . \Xi C_1[t] (\lambda
  x . C_2[z/f x][t] \sim_\alpha \tau$ (recall that by our convention
  regarding contexts, the free variables of~$t$ are assumed not to
  become bound in~$C[t]$).

  Finally, if $C[x_1 \ldots x_n] \sim_\alpha \tau$ is obtained by
  rule~$(\mathrm{F''})$ then $\tau = \tau_1\to\tau_2$ and $C \equiv
  \lambda f . \Xi C_1 C_2$ where $C_1[x_1\ldots x_n] \sim_{<\alpha}
  \tau_1$ and $C_2[x_1 \ldots x_n] \sim_{<\alpha} \tau_2 \in
  \{\omega,\varepsilon\}$. But then by parts~(\ref{r:lcsp_1})
  and~(\ref{r:lcsp_3}) of the~IH we have $C_1[t] \sim_{<\alpha}
  \tau_1$ and $C_2[t] \sim_{<\alpha} \tau_2$, which implies $C[t]
  \equiv \lambda f . \Xi C_1[t] C_2[t] \sim_\alpha \tau$.

  Now we check condition~(\ref{r:lcsp_2}). Suppose $C[x_1 \ldots x_n]
  \succ_\alpha \rho$ for a canonical term~$\rho$. If $C[x_1 \ldots
  x_n] \equiv \rho$ then the claim is obvious because canonical terms
  are closed, so $C[t] \equiv C \equiv C[x_1 \ldots x_n] \equiv
  \rho$. If the canonical type of~$\rho$ is $\tau_1\to\tau_2$ then by
  definition for any $t_1 \in \T_{\tau_1}$ we have $C[x_1 \ldots x_n]
  t_1 \leadsto_{<\alpha} \F(\rho)(t_1)$. By parts~(\ref{r:lcsp_2})
  and~(\ref{r:lcsp_3}) of the IH and by the definition of
  $\leadsto_{<\alpha}$ we obtain $C[t] t_1 \leadsto_{<\alpha}
  \F(\rho)(t_1)$. Hence $C[t] \succ_\alpha \rho$.

  Suppose $\rho \equiv \top$. If $C[x_1 \ldots x_n] \not\succ_0 \top$
  then one of the conditions~$(A_\tau^\top)$, $(\Xi^\top)$ or
  $(L^\top)$ in Definition~\ref{r:def_succ} must
  hold. If~$(A_\tau^\top)$ holds then the claim is obvious, because
  $C[x_1 \ldots x_n]$ is closed.

  If~$(\Xi^\top)$ holds then $C \equiv \Xi C_1 C_2$ and there exists
  $\tau$ such that $C_1[x_1 \ldots x_n] \sim_\alpha \tau$ and for all
  $t' \in \T_\tau$ we have $C_2[x_1 \ldots x_n] t' \leadsto_{<\alpha}
  \top$. By claim~(\ref{r:lcsp_3}), which has already been verified in
  this inductive step, we obtain $C_1[t] \sim_\alpha \tau$. By
  parts~(\ref{r:lcsp_1}) and~(\ref{r:lcsp_2}) of the IH we conclude
  that for all $t' \in \T_\tau$ we have $C_2[t] t' \leadsto_{<\alpha}
  \top$. Therefore $C[t] = \Xi C_1[t] C_2[t] \succ_\alpha \top$.

  If condition~$(L^\top)$ holds then $C \equiv L C_1$ and $C_1[x_1
  \ldots x_n] \sim_\alpha \tau$ for some type~$\tau$. By
  calim~(\ref{r:lcsp_3}), which has already been verified in this
  inductive step, we obtain $C_1[t] \sim_\alpha \tau$. Therefore $C[t]
  \succ_\alpha \top$.

  It remains to verify the case $C[x_1 \ldots x_n] \succ_\alpha
  \bot$. Assuming $C[x_1 \ldots x_n] \not\succ_0 \bot$, the
  condition~$(\Xi^\bot)$ must hold. Then the claim again follows by
  applying the already verified condition~(\ref{r:lcsp_3}) and
  parts~(\ref{r:lcsp_1}) and~(\ref{r:lcsp_2}) of the inductive
  hypothesis.
\end{proof}

\begin{sscorollary}\label{r:cor_subst}
  If $t_1 \contr_{=\alpha} t_1'$ and the free variables of~$t_2$ do
  not become bound in~$t_1[x/t_2]$, then $t_1[x/t_2] \contr_{=\alpha}
  t_1'[x/t_2]$.
\end{sscorollary}

\begin{proof}
  If $\alpha = 0$ then this is obvious. If $\alpha > 0$ then assume
  without loss of generality that $c t_1 \contr_{=\alpha} \rho_2
  \equiv t_1'$ by virtue of $t_1 \succ_{<\alpha} \rho_1$. But then by
  part~(\ref{r:lcsp_2}) of Lemma~\ref{r:lem_nu_context} we have
  $t_1[x/t_2] \succ_{<\alpha} \rho_1$, so $c t_1[x/t_2]
  \contr_{=\alpha} \rho_2 \equiv t_1'[x/t_2]$, since the canonical
  term~$\rho_2$ is closed.
\end{proof}

\begin{sslemma}\label{r:lem_k_sim}
  If $K t \sim_\alpha \tau$ then $\tau = \omega$ or $\tau =
  \varepsilon$.
\end{sslemma}

\begin{proof}
  Induction on $\alpha$. The non-obvious case is when $K t \equiv
  \lambda f . \Xi t_1 (\lambda x . t_2[z/fx]) \sim_\alpha
  \tau_1\to\tau_2$ is obtained by rule~$(\mathrm{F'})$, and $t_1
  \sim_{<\alpha} \tau_1$ for $\tau_1 \ne \varepsilon$, and $\lambda z
  . t_2 \sim_{<\alpha} \tau_2$. But then $t \equiv \Xi t_1 (\lambda x
  . t_2[z/fx])$ and $z \notin FV(t_2)$. Since $K t_2 \sim_{<\alpha}
  \tau_2$ by the inductive hypothesis we conclude $\tau_2 = \omega$ or
  $\tau_2 = \varepsilon$. In either case $\tau = \omega$ or $\tau =
  \varepsilon$.
\end{proof}

The next lemma and Lemma~\ref{r:lem_context} are the two key technical
lemmas justifying the correctness of our model construction.

\begin{sslemma}\label{r:lem_commute}
  For all ordinals $\alpha$, $\beta$ the following conditions hold:
  \begin{enumerate}
  \item $R_\alpha$ and $R_\beta$ commute, i.e., if $t
    \reduces_{\le\alpha} t_1$ and $t \reduces_{\le\beta} t_2$ then
    $t_1 \reduces_{\le\beta} t'$ and $t_2 \reduces_{\le\alpha} t'$ for
    some term~$t'$, \label{r:comm_1} \vspace{-0.5em}
  \item if $t_1 \succ_\alpha \rho$ and $t_1 \reduces_{\le\beta}
    t_2$ then $t_2 \succ_\alpha \rho$, \label{r:comm_2}
  \item if $t \succ_\alpha \rho_1$, $t \succ_\beta \rho_2$ and
    $\rho_1, \rho_2 \in \T_\tau$ then $\rho_1 \equiv
    \rho_2$, \label{r:comm_3}
  \item if $t_1 \sim_\alpha \tau$ and $t_1 \reduces_{\le\beta} t_2$
    then $t_2 \sim_\alpha \tau$, \label{r:comm_4}
  \item if $t \sim_\alpha \tau_1$ and $t \sim_\beta \tau_2$ then
    $\tau_1 = \tau_2$, \label{r:comm_5}
  \item if $t \sim_\alpha \omega$ then $t r \leadsto_{<\alpha} \top$
    for all $r$, and if $t \sim_\alpha \varepsilon$ then $t r
    \leadsto_{<\alpha} \bot$ for all $r$. \label{r:comm_6}
  \end{enumerate}
\end{sslemma}

\begin{proof}
  Induction on pairs $\langle \alpha, \beta \rangle$ ordered
  lexicographically. Together with every condition we show its dual,
  i.e., the condition with~$\alpha$ and~$\beta$ exchanged. We give
  proofs only for the original conditions, but it can be easily seen
  that in every case the dual condition follows by exactly the same
  proof with~$\alpha$ and~$\beta$ exchanged. Note that for a proof of
  a condition to be a proof of its dual, it suffices that we never use
  the inductive hypothesis with~$\beta$ increased.

  %% NOTE: we may never increase \beta, because we need to prove the
  %% dual conditions

  First note that conditions~(\ref{r:comm_1}) and~(\ref{r:comm_2})
  imply that if $t_1 \leadsto_{\alpha} \rho$ and $t_1
  \reduces_{\le\beta} t_2$, then $t_2 \leadsto_{\alpha} \rho$. Indeed,
  if $t_1 \reduces_{\le\alpha} t_1' \succ_{\alpha} \rho$ and $t_1
  \reduces_{\le\beta} t_2$, then by~(\ref{r:comm_1}) we have $t_2
  \reduces_{\le\alpha} t_2'$ and $t_1' \reduces_{\le\beta}
  t_2'$. Hence by~(\ref{r:comm_2}) it follows that $t_2'
  \succ_{\alpha} \rho$, so $t_2 \leadsto_{\alpha} \rho$.

  Instead of~(\ref{r:comm_1}) we prove a stronger claim that
  $\widehat{R_\alpha}$ and $\widehat{R_\beta}$
  commute. Condition~(\ref{r:comm_1}) follows from this claim by a
  simple tiling argument, similar to the proof of the Hindley-Rosen
  lemma.

  %% NOTE: lexicographic order: $\langle \alpha_1, \beta_1 \rangle <
  %% \langle \alpha_2, \beta_2 \rangle$ iff $\alpha_1 < \alpha_2$, or
  %% $\alpha_1 = \alpha_2$ and $\beta_1 < \beta_2$ (I always confuse
  %% such things).
  %% NOTE: We need the dual conditions to verify (1)

  If $\alpha = \beta = 0$ then the claim is obvious, because
  $\widehat{R_0} = R_0$ is the ordinary
  $\lambda\beta\eta$-calculus. We therefore check that~$R_{0}$
  commutes with~$\widehat{R_\alpha}$ for $\alpha > 0$. We show that if
  $t \contr_{=\alpha} t_1$ and $t \contr_{\beta\eta} t_2$ then there
  exists $t_3$ such that $t_1 \contr_{\beta\eta}^\equiv t_3$ and $t_2
  \reduces_{=\alpha} t_3$. The claim then follows by a simple diagram
  chase. First suppose $t \equiv (\lambda x . r_1) r_2 \contr_\beta
  r_1[x/r_2] \equiv t_2$ and $r_1 \contr_{=\alpha} r_1'$. Then by
  Corollary~\ref{r:cor_subst} we have $r_1[x/r_2] \contr_{=\alpha}
  r_1'[x/r_2]$. Also obviously $(\lambda x . r_1') r_2 \contr_{\beta}
  r_1'[x/r_2]$. If $t \equiv (\lambda x . r_1) r_2 \contr_\beta
  r_1[x/r_2] \equiv t_2$ and $r_2 \contr_{=\alpha} r_2'$ then the
  claim is obvious. Suppose $t \equiv \lambda x . r x \contr_\eta r$
  where $x \notin FV(r)$. The only interesting case is when $r \equiv
  c$ and $c x \contr_{=\alpha} \rho_2$ by virtue of~$x \succ_{<\alpha}
  \rho_1$. But then by part~(2) of Lemma~\ref{r:lem_nu_context} we
  have $\rho' \succ_{<\alpha} \rho_1$ for $\rho' \ne \rho_1$ with
  $\rho'$ of the same canonical type as~$\rho_1$. This is, however,
  impossible by part~(\ref{r:comm_3}) of the~IH. Without loss of
  generality, the only remaining case is $t \equiv c t'
  \contr_{=\alpha} t_1 \equiv \F(c)(\rho)$, $t' \succ_{<\alpha} \rho$,
  $t_2 \equiv c t_2'$, and $t' \contr_{\beta\eta} t_2'$. By
  part~(\ref{r:comm_2}) of the IH we obtain $t_2' \succ_{<\alpha}
  \rho$. Therefore $t_2 \equiv c t_2' \contr_{=\alpha} \F(c)(\rho)
  \equiv t_1$.

  We now check that $\widehat{R_\alpha}$ commutes with
  $\widehat{R_\beta}$ for $\alpha, \beta > 0$. It suffices to show
  that if $t \contr_{=\alpha} t_1$ and $t \contr_{=\beta} t_2$ then
  there exists $t_3$ such that $t_1 \contr_{=\beta}^\equiv t_3$ and
  $t_2 \contr_{=\alpha}^\equiv t_3$. If the redexes do not overlap
  then this is obvious. Suppose they overlap at the root, i.e., $t
  \equiv c t'$, $c t' \contr_{=\alpha} t_1 \equiv \F(c)(\rho_1)$ where
  $t' \succ_{<\alpha} \rho_1$, and $c t' \contr_{=\beta} t_2 \equiv
  \F(c)(\rho_2)$ where $t' \succ_{<\beta} \rho_2$. But then $\rho_1$
  and $\rho_2$ are canonical terms of the same type, which is
  determined by the type of $c$. So by part~(\ref{r:comm_3}) of the IH
  we obtain $\rho_1 \equiv \rho_2$. Hence $t_1 \equiv t_2$. If the
  overlap does not happen at the root, then without loss of generality
  $t \equiv c t'$, $c t' \contr_{=\alpha} t_1 \equiv \F(c)(\rho)$
  where $t' \succ_{<\alpha} \rho$, $t_2 \equiv c t_2'$, and $t'
  \contr_{=\beta} t_2'$. By part~(\ref{r:comm_2}) of the IH we obtain
  $t_2' \succ_{<\alpha} \rho$, so $t_2 \equiv c t_2' \contr_{=\alpha}
  \F(c)(\rho) \equiv t_1$.

  Now we shall prove~(\ref{r:comm_4}). If $t_1 \equiv A_\tau
  \sim_\alpha \tau$ for $\tau \in \B$ or $t_1 \equiv H$, then the
  claim is obvious. If $t_1 \equiv K t_1' \sim_\alpha \omega$ and
  $t_1' \leadsto_{<\alpha} \top$, then $t_2 = K t_2'$, $t_1'
  \reduces_{\le\beta} t_2'$, and by parts~(\ref{r:comm_1})
  and~(\ref{r:comm_2}) of the IH we have $t_2' \leadsto_{<\alpha}
  \top$. Hence $t_2 \sim_\alpha \omega$. If $t_1 \equiv K t_1'
  \sim_\alpha \varepsilon$ and $t_1' \leadsto_{<\alpha} \bot$ the
  proof is analogous.

  If $t_1 \sim_\alpha \tau_1\to\tau_2$ follows by~$(\mathrm{F})$ then
  $t_1 \equiv \lambda f . \Xi t_1^1 (\lambda x . t_1^2 (f x))$ where
  $t_1^1 \sim_{<\alpha} \tau_1$ and $t_1^2 \sim_{<\alpha}
  \tau_2$. Without loss of generality, we may assume $t_1
  \to_{\le\beta} t_2$, i.e., the reduction $t_1 \reduces_{\le\beta}
  t_2$ consists of a single step. Then $t_2 \equiv \lambda f . \Xi
  t_2^1 s$ with $t_1^1 \to_{\le\beta}^\equiv t_2^1$ and $\lambda x
  . t_1^2 (f x) \to_{\le\beta}^\equiv s$. By the IH we have $t_2^1
  \sim_{<\alpha} \tau_1$. If $t_1^2 \equiv \lambda z . s_1$ and $s
  \equiv \lambda x . s_2[z/f x]$ then $t_2 \sim_\alpha
  \tau_1\to\tau_2$ by~$(\mathrm{F'})$. It is impossible that $t_1^2 (f
  x)$ is a redex with~$t_1^2$ a constant. Indeed, then $f x
  \succ_{<\beta} \rho$ for some canonical~$\rho$. Using the definition
  of~$\succ$ and noting that a term of the form $f x w_1 \ldots w_k$
  is not a $\to_\gamma$-redex for any~$\gamma$ because~$f$ is a
  variable, we may conclude that $f x w_1 \ldots w_n \succ_{\gamma}
  \rho'$ for some~$\gamma,w_1,\ldots,w_n$ and some canonical~$\rho'$
  of type~$o$ or base type. But this contradicts the definition
  of~$\succ_\gamma$. Therefore, the only remaining possibility is $s
  \equiv \lambda x . t_2^2 (f x)$ with $t_1^2 \to_{\le\beta}^\equiv
  t_2^2$. By the IH we obtain $t_2^2 \sim_{<\alpha} \tau_2$. Therefore
  $t_2 \sim_\alpha \tau_1\to\tau_2$ by~$(\mathrm{F})$.

  If $t_1 \sim_\alpha \tau_1\to\tau_2$ follows by~$(\mathrm{F'})$ then
  $t_1 \equiv \lambda f . \Xi t_1^1 (\lambda x . t_1^2[z/fx])$ where
  $t_1^1 \sim_{<\alpha} \tau_1$ and $\lambda z . t_1^2 \sim_{<\alpha}
  \tau_2$. Without loss of generality, we may assume $t_1
  \to_{\le\beta} t_2$, i.e., the reduction $t_1 \reduces_{\le\beta}
  t_2$ consists of a single step. By Lemma~\ref{r:lem_xi_unique} we
  have $t_2 \equiv \lambda f . \Xi t_2^1 s_2$ where $t_1^1
  \to_{\le\beta}^\equiv t_2^1$ and $\lambda x . t_1^2[z/fx]
  \to_{\le\beta}^\equiv s_2$. We show that $s_2 \equiv \lambda x
  . t_2^2[z/fx]$ with $t_1^2 \reduces_{\le\beta} t_2^2$. Suppose the
  contraction in $\lambda x . t_1^2[z/fx]$ occurs at the root. Then
  this must be an $\eta$-contraction, and because $x \notin FV(t_1^2)$
  we have $t_1^2 \equiv z$. But then $\lambda z . z \sim_{<\alpha}
  \tau_2$. By inspecting the definition of~$\sim_{<\alpha}$ this is
  seen to be impossible. Hence the contraction does not occur at the
  root, and thus it follows from Lemma~\ref{r:lem_xi_unique} and
  Lemma~\ref{r:lem_nu_context} that $s_2 \equiv \lambda x
  . t_2^2[z/fx]$ with $t_1^2 \reduces_{\le\beta} t_2^2$. By the IH we
  obtain $t_2^1 \sim_{<\alpha} \tau_1$ and $\lambda z . t_2^2
  \sim_{<\alpha} \tau_2$. Thus $t_2 \sim_\alpha \tau_1\to\tau_2$.

  If $t_1 \sim_\alpha \tau_1\to\tau_2$ follows by~$(\mathrm{F''})$
  then $t_1 \equiv \lambda f . \Xi t_1^1 t_1^2$ with $t_1^1
  \sim_{<\alpha} \tau_1$ and $t_1^2 \sim_{<\alpha} \tau_2$, with $f
  \notin FV(t_1^2)$. Since $t_1^2 \not\equiv f$ we have $t_2 \equiv
  \lambda f . \Xi t_2^1 t_2^2$ with $t_1^1 \reduces_{\le\beta} t_2^1$
  and $t_1^2 \reduces_{\le\beta} t_2^2$. By the IH we obtain $t_2^1
  \sim_{<\alpha} \tau_1$ and $t_2^2 \sim_{<\alpha} \tau_2$. Thus $t_2
  \sim_\alpha \tau_1\to\tau_2$.

  If $t_1 \sim_\alpha \omega$ follows by~$(\mathrm{F}\omega)$ then
  $t_1 \equiv \lambda f . \Xi t_1^1 t_1^2$ and $t_1^1 \sim_{<\alpha}
  \varepsilon$. Without loss of generality we assume $t_1
  \to_{\le\beta} t_2$. There are two possibilities.
  \begin{itemize}
  \item $t_2 \equiv \lambda f . \Xi t_2^1 t_2^2$ with $t_1^1
    \reduces_{\le\beta} t_2^1$ and $t_1^2 \reduces_{\le\beta}
    t_2^2$. Then $t_2^1 \sim_{<\alpha} \varepsilon$ by the~IH, so $t_2
    \sim_\alpha \omega$ by~$(\mathrm{F}\omega)$.
  \item $t_2 \equiv \Xi t_1^1$. Then $t_2 \sim_\alpha \omega$ follows
    by~$(\mathrm{F}\omega')$.
  \end{itemize}

  If $t_1 \sim_\alpha \omega$ follows by~$(\mathrm{F}\omega')$ then
  $t_1 \equiv \Xi t_1'$, $t_1' \sim_{<\alpha} \varepsilon$ and $t_2
  \equiv \Xi t_2'$ with $t_1' \reduces_{\le\beta} t_2'$. Then $t_2'
  \sim_{<\alpha} \varepsilon$ by the~IH. Thus $t_2 \sim_\alpha
  \omega$.

  We show~(\ref{r:comm_2}). If $t_1 \equiv \rho$ then $t_1$ is in
  $R_\beta$-normal form, so there is nothing to prove. If $t_1
  \not\equiv \rho$, $t_1 \succ_{\alpha} \rho$ and $t_1
  \reduces_{\le\beta} t_2$, where $\rho \in \T_{\tau_1\to\tau_2}$,
  then by definition for all $\rho_1 \in \T_{\tau_1}$ we have $t_1
  \rho_1 \leadsto_{<\alpha} \rho_2$, where $\rho_2 \equiv
  \F(\rho)(\rho_1)$. But then by parts~(\ref{r:comm_1})
  and~(\ref{r:comm_2}) of the inductive hypothesis $t_2 \rho_1
  \leadsto_{<\alpha} \rho_2$, so $t_2 \succ_\alpha \rho$. Therefore
  suppose $t_1 \succ_\alpha \top$. When $t_1 \succ_\alpha \bot$ the
  argument is similar. If $\alpha = 0$ then the claim is obvious,
  because the right sides of the identities in the postulates for $t_1
  \succ_0 \top$ are normal forms. If $\alpha > 0$ then assume $t_1
  \reduces_{\le\beta} t_2$, $t_1 \equiv \Xi t_1^1 t_1^2$ and
  condition~$(\Xi^\top)$ in the definition of~$t_1 \succ_\alpha \top$
  is satisfied, i.e., there exists~$\tau$ s.t. $t_1^1 \sim_\alpha
  \tau$ and for all $t_3 \in \T_\tau$ we have $t_1^2 t_3
  \leadsto_{<\alpha} \top$. When any of the other conditions in the
  definition of~$t_1 \succ_\alpha \top$ is satisfied instead
  of~$(\Xi^\top)$, then the proof is analogous. By
  Lemma~\ref{r:lem_xi_unique} we have $t_2 \equiv \Xi t_2^1 t_2^2$
  where $t_1^1 \reduces_{\le\beta} t_2^1$ and $t_1^2
  \reduces_{\le\beta} t_2^2$. By~(\ref{r:comm_4}), which has already
  been verified in this inductive step, we obtain $t_2^1 \sim_\alpha
  \tau$. It therefore suffices to check that for all $t_3 \in \T_\tau$
  we have $t_2^2 t_3 \leadsto_{<\alpha} \top$. But for $t_3 \in
  \T_\tau$ obviously $t_1^2 t_3 \leadsto_{<\alpha} \top$, so $t_2^2
  t_3 \leadsto_{<\alpha} \top$ by parts~(\ref{r:comm_1})
  and~(\ref{r:comm_2}) of the IH.

  We show~(\ref{r:comm_6}). Suppose $t \sim_\alpha \omega$. When $t
  \sim_\alpha \varepsilon$ the argument is similar. If $t \sim_\alpha
  \omega$ is obtained by rule $(\mathrm{K\omega})$ then the claim is
  obvious. If $t \sim_\alpha \omega$ is obtained by~$(\mathrm{F})$
  then $t \equiv \lambda f . \Xi t_1 (\lambda x . t_2 (f x))$ with $f,
  x \notin FV(t_1,t_2)$, $t_1 \sim_{<\alpha} \tau_1$ and $t_2
  \sim_{<\alpha} \tau_2$. Because $\tau=\omega$ we must have
  $\tau_1=\varepsilon$ or $\tau_2=\omega$. Since $t r \to_\beta \Xi
  t_1 (\lambda x . t_2 (r x))$ it suffices to show $\Xi t_1 (\lambda x
  . t_2 (r x)) \succ_{<\alpha} \top$. If $\tau_1=\varepsilon$ then
  this follows from~$(\Xi^\top)$ because $\T_\varepsilon =
  \emptyset$. So assume $\tau_2=\omega$. Let $\gamma < \alpha$ be such
  that $t_1 \sim_\gamma \tau_1$ and $t_2 \sim_\gamma \omega$. Let $t_3
  \in \T_{\tau_1}$. By part~(\ref{r:comm_6}) of the~IH we have
  $(\lambda x . t_2 (r x)) t_3 \to_\beta t_2 (r t_3)
  \leadsto_{<\gamma} \top$. Hence $\Xi t_1 t_2 \succ_{<\alpha} \top$
  by~$(\Xi^\top)$. If $t \sim_\alpha \omega$ is obtained
  by~$(\mathrm{F''})$ then the argument is analogous to the case
  for~$(\mathrm{F})$. If the derivation of $t \sim_\alpha \omega$ is
  by $(\mathrm{F\omega})$ then $t \equiv \lambda f . \Xi t_1 t_2$ with
  $t_1 \sim_{<\alpha} \varepsilon$. Then $ t r \to_\beta \Xi t_1
  (t_2[f/r]) \succ_{<\alpha} \top$ by definition. When $t \sim_\alpha
  \omega$ is obtained by~$(\mathrm{F\omega'})$ the argument is
  analogous to the case for~$(\mathrm{F\omega})$. The only other
  possiblity is that $t \sim_\alpha \omega$ is obtained by rule
  $(\mathrm{F'})$. Then $t \equiv \lambda f . \Xi t_1 (\lambda x
  . t_2[z/fx])$, $t_1 \sim_{<\alpha} \tau_1$ and $\lambda z . t_2
  \sim_{<\alpha} \tau_2$. It suffices to verify that $\Xi t_1 (\lambda
  x . (\lambda z . t_2) (r x)) \succ_{<\alpha} \top$, because for $t'
  \equiv \Xi t_1 (\lambda x . t_2[z/rx])$ we have $\Xi t_1 (\lambda x
  . (\lambda z . t_2) (r x)) \to_\beta t'$ and $t r \to_\beta t'$, so
  then $t' \succ_{<\alpha} \top$ by part~(\ref{r:comm_2}) of the IH,
  which implies $t r \leadsto_{<\alpha} \top$. But the argument to
  show $\Xi t_1 (\lambda x . (\lambda z . t_2) (r x)) \succ_{<\alpha}
  \top$ is analogous to the case for~$(\mathrm{F})$.

  We show~(\ref{r:comm_5}). Suppose $t \sim_\alpha \tau_1$ and $t
  \sim_\beta \tau_2$. If $t \equiv A_\tau$ for~$\tau \in \B$ or $t
  \equiv H$ then the claim is obvious. So suppose $t \not\equiv
  A_\tau$ for $\tau \in \B$ and $t \not\equiv H$. First assume that
  both $t \sim_\alpha \tau_1$ and $t \sim_\beta \tau_2$ are obtained
  by rule $(\mathrm{F'})$. Hence $\tau_1 = \tau_1^1\to\tau_1^2$,
  $\tau_2 = \tau_2^1\to\tau_2^2$, and $t \equiv \lambda f . \Xi t_1
  (\lambda x . t_2[z/fx])$ where $t_1 \sim_{<\alpha} \tau_1^1$,
  $\lambda z . t_2 \sim_{<\alpha} \tau_1^2$, $t_1 \sim_{<\beta}
  \tau_2^1$ and $\lambda z . t_2 \sim_{<\beta} \tau_2^2$. By the IH we
  obtain $\tau_1^1 = \tau_2^1$ and $\tau_1^2 = \tau_2^2$. Hence
  $\tau_1 = \tau_2$. If one of $t \sim_\alpha \tau_1$ or $t \sim_\beta
  \tau_2$ is obtained by~$(\mathrm{F\omega})$ and the other
  by~$(\mathrm{F})$, $(\mathrm{F'})$ or~$(\mathrm{F''})$, or one
  by~$(\mathrm{F})$ and the other by~$(\mathrm{F'})$, or both are
  obtained by~$(\mathrm{F})$, etc., then the argument is similar. If
  one is obtained by $(\mathrm{K\omega})$ and the other by
  $(\mathrm{K\varepsilon})$, then the claim follows from
  parts~(\ref{r:comm_2}) and~(\ref{r:comm_3}) of the IH. The only
  other possibility is, without loss of generality, when $t
  \sim_\alpha \tau_1$ is obtained by $(\mathrm{K\omega})$ or
  $(\mathrm{K\varepsilon})$ and $t \sim_\beta \tau_2$ by
  $(\mathrm{F})$, $(\mathrm{F'})$, $(\mathrm{F''})$ or
  $(\mathrm{F\omega})$. Then $t \equiv K t'$. So by
  Lemma~\ref{r:lem_k_sim} we have $\tau_1, \tau_2 \in \{\omega,
  \varepsilon\}$. For instance, suppose $\tau_1 = \omega$ and $\tau_2
  = \varepsilon$. By~(\ref{r:comm_6}) and its dual, which we have
  already verified in this inductive step, for all $t_3$ we have $t
  t_3 \leadsto_{<\alpha} \top$ and $t t_3 \leadsto_{<\beta} \bot$. By
  parts~(\ref{r:comm_1}) and~(\ref{r:comm_2}) of the IH this implies
  the existence of $t_4$ such that $t_4 \succ_{<\alpha} \top$ and $t_4
  \succ_{<\beta} \bot$, which contradicts part~(\ref{r:comm_3}) of
  the~IH.

  It remains to verify~(\ref{r:comm_3}). If $\tau \in \B$ then this is
  obvious. Suppose $\tau=\tau_1\to\tau_2 \in \Tc_1$. Note that for all
  $t_1 \in \T_{\tau_1}$ we have $\F(\rho_1)(t_1) \equiv
  \F(\rho_2)(t_1)$. This follows from the definition of $\succ_\alpha$
  for $\tau=\tau_1\to\tau_2 \in \Tc_1$, from parts~(\ref{r:comm_1}),
  (\ref{r:comm_2}) and~(\ref{r:comm_3}) of the IH, and from the fact
  that canonical terms are in normal form. Now, if $\tau_1 = \omega$
  then $\rho_1 \equiv \lambda x . \rho_1'$ and $\rho_2 \equiv \lambda
  x . \rho_2'$. Thus for any $t_1$ we have $\rho_1' \equiv
  \F(\rho_1)(t_1) \equiv \F(\rho_2)(t_1) \equiv \rho_2'$, so $\rho_1
  \equiv \rho_2$. If $\tau_1 \ne \omega$ then the claim is immediate,
  because $\T_{\tau_1\to\tau_2} = \Sigma_{\tau_1\to\tau_2}$ for
  $\tau_1 \ne \omega$ was defined to contain exactly one constant for
  every function from $\T_{\tau_1}$ to~$\T_{\tau_2}$.

  The last remaining case is $\tau = o$. Thus, suppose $t
  \succ_{\alpha} \top$ and $t \succ_{\beta} \bot$. It is easily seen
  that this is possible only when the conditions~$(\Xi^\top)$
  and~$(\Xi^\bot)$ are satisfied. So we have $t \equiv \Xi t_1 t_2$
  and there exists~$\tau_1$ such that $t_1 \sim_\alpha \tau_1$ and for
  all $t' \in \T_{\tau_1}$ we have $t_2 t' \leadsto_{<\alpha}
  \top$. There also exists~$\tau_2$ and $t_3 \in \T_{\tau_2}$ such
  that $t_1 \sim_\beta \tau_2$ and $t_2 t_3 \leadsto_{<\beta}
  \bot$. But by~(\ref{r:comm_5}) we have $\tau_1 = \tau_2$. Hence $t_2
  t_3 \leadsto_{<\alpha} \top$ and $t_2 t_3 \leadsto_{<\beta} \bot$,
  which contradicts the inductive hypothesis.
\end{proof}

\begin{sscorollary}\label{r:corollary_t_conv}
  If $t =_{\le\alpha} t'$ then $t \leadsto_\alpha \top$ is equivalent
  to $t' \leadsto_\alpha \top$.
\end{sscorollary}

\begin{proof}
  Follows from conditions~(\ref{r:comm_1}) and~(\ref{r:comm_2}) in
  Lemma~\ref{r:lem_commute}.
\end{proof}

\begin{sscorollary}\label{r:corollary_leadsto_consistent}
  If $t \leadsto_\alpha \rho_1$ and $t \leadsto_\alpha \rho_2$ where
  $\rho_1$, $\rho_2$ are canonical terms with the same canonical type,
  then $\rho_1 \equiv \rho_2$.
\end{sscorollary}

\begin{proof}
  Follows from conditions~(\ref{r:comm_1})-(\ref{r:comm_3}) in
  Lemma~\ref{r:lem_commute}.
\end{proof}

\begin{sslemma}\label{r:lem_extensional}
  Let $t_1$ and $t_2$ be terms. If for all terms $t_0$ we have $t_1
  t_0 =_{\le\alpha} t_2 t_0$ then $t_1 =_{\le\alpha} t_2$. In
  particular, the combinatory algebra of $\M$, as defined in
  Definition~\ref{r:def_model}, is extensional.
\end{sslemma}

\begin{proof}
  If $t_1 t_0 =_{\le\alpha} t_2 t_0$ for all terms $t_0$ then in
  particular $t_1 x =_{\le\alpha} t_2 x$ where $x$ is variable which
  does not occur in $t_1$ and $t_2$. Hence $t_1
  \leftidx{_\eta}{\from}{} \lambda x . t_1 x =_{\le\alpha} \lambda x
  . t_2 x \to_\eta t_2$. Therefore $t_1 =_{\le\alpha} t_2$.
\end{proof}

The \emph{rank} of a type $\tau$, denoted $\rank(\tau)$, is defined as
follows. If $\tau \in \B \cup \{o, \omega, \varepsilon\}$ then
$\rank(\tau) = 1$. Otherwise $\tau = \tau_1\to\tau_2 \in \Tc_1$ and we
set $\rank(\tau) = \max\{\rank(\tau_1)+1,\rank(\tau_2)\}$. By the rank
of a canonical term we mean the rank of its canonical type.

We write $t \gg_\alpha t'$ if there exists an $n$-ary context~$C$,
terms $t_1, \ldots, t_n$, and canonical terms $\rho_1, \ldots,
\rho_n$, such that $t_i \succ_\alpha \rho_i$ for $i=1,\ldots,n$, $t
\equiv C[t_1,\ldots,t_n]$ and $t' \equiv C[\rho_1,\ldots,\rho_n]$. If
the maximal rank of $\rho_1,\ldots,\rho_n$ is at most $k$ then we
write $t \gg_\alpha^k t'$, and if it is less than $k$ we write $t
\gg_\alpha^{<k} t'$.

Recall that whenever we write $C[t_1,\ldots,t_n]$ we assume that the
free variables of $t_1,\ldots,t_n$ do not become bound in
$C[t_1,\ldots,t_n]$.

\begin{sslemma}\label{r:lem_succ_abstraction}
  If $t \succ_\alpha \rho$ and $x_1,\ldots,x_n \notin FV(t)$ then
  $\lambda x_1 \ldots x_k . t \succ_{\alpha + k} \lambda x_1 \ldots
  x_k . \rho$.
\end{sslemma}

\begin{proof}
  Easy induction on $k$.
\end{proof}

\begin{sslemma}\label{r:lem_f_gg_0}
  If $t \gg^n F r_1' r_2'$ then $t \equiv F r_1 r_2$ with $r_1 \gg^n
  r_1'$ and $r_2 \gg^n r_2'$.
\end{sslemma}

\begin{proof}
  This follows from $F r_1' r_2' \equiv \lambda f . \Xi r_1' (\lambda
  x . r_2' (f x))$ and from the fact that canonical terms are closed
  and do not contain~$\Xi$.
\end{proof}

\begin{sslemma}\label{r:lem_f_gg}
  If $t \gg^n \lambda f . \Xi r_1' (\lambda x . r_2'[z/fx])$ with $x,f
  \notin FV(r_2')$ then one of the following holds:
  \begin{itemize}
  \item $t \equiv \lambda f . \Xi r_1 (\lambda x . r_2[z/fx])$, $x,f
    \notin FV(r_2)$, $r_1 \gg^n r_1'$ and $r_2 \gg^n r_2'$, or
  \item $t \equiv \lambda f . \Xi r_1 r_2$, $z \notin FV(r_2')$, $r_1
    \gg^n r_1'$ and $r_2 \gg^n \lambda z . r_2'$.
  \end{itemize}
\end{sslemma}

\begin{proof}
  Let $q \equiv \lambda x . r_2'[z/fx]$. Since $t \gg^n \lambda f
  . \Xi r_1' q$ there exist contexts $C_1, C_2$, terms $t_1, \ldots,
  t_k$, and canonical terms $\rho_1, \ldots, \rho_k$, such that $t_i
  \succ_\alpha \rho_i$ for $i=1,\ldots,k$, $t \equiv \lambda f . \Xi
  C_1[t_1,\ldots,t_k] C_2[t_1,\ldots,t_k]$, $r_1' \equiv
  C_1[\rho_1,\ldots,\rho_k]$ and $q \equiv
  C_2[\rho_1,\ldots,\rho_k]$. We take $r_1 \equiv
  C_1[t_1,\ldots,t_k]$. If $C_2 \equiv \lambda x . (C_2')[z / (f x)]$,
  then $C_2'[\rho_1,\ldots,\rho_k] \equiv r_2'$ and we take $r_2
  \equiv C_2'[t_1,\ldots,t_k]$. Otherwise $z \notin FV(r_2')$ and
  $\lambda x . r_2' \equiv \rho_i$ for some $1 \le i \le k$. Then
  $C_2[t_1,\ldots,t_k] \equiv t_i \succ \rho_i$ and the second point
  in the statement of the lemma holds.
\end{proof}

Recall that we use the notations $R$, $\succ$, $\leadsto$, $\gg$,
etc. without subscripts to denote $R_\zeta$, $\succ_\zeta$,
$\leadsto_\zeta$, $\gg_\zeta$, etc., where~$\zeta$ is the ordinal
introduced just before Definition~\ref{r:def_model}. For this ordinal
we have $\succ_\zeta\, =\, \succ_{<\zeta}$, $R_{\zeta} = R_{<\zeta}$,
etc.

\begin{sslemma}\label{r:lem_context}
  If $t_1$, $t_2$, $t_3$ are terms, $\rho$ is a canonical term, and
  $\tau$ is a type, then for every ordinal~$\alpha$ and every natural
  number~$n$ the following conditions hold:
  \begin{enumerate}
  \item if $t_1 \gg^n t_2 \succ_\alpha \rho$ then $t_1 \succ
    \rho$, \label{r:con_1}
  \item if $t_1 \gg^n t_2 \sim_\alpha \tau$ then $t_1 \sim
    \tau$, \label{r:con_2}
  \item if $t_1 \gg^n t_2 \reduces_{\le\alpha} t_2'$ then $t_1
    \reduces_{R} t_1' \gg^n t_2'$. \label{r:con_3}
  \end{enumerate}
\end{sslemma}

\begin{proof}
  Induction on pairs $\pair{n}{\alpha}$ ordered lexicographically,
  i.e., $\pair{n_1}{\alpha_1} < \pair{n_2}{\alpha_2}$ iff $n_1 < n_2$,
  or $n_1 = n_2$ and $\alpha_1 < \alpha_2$.

  First we verify condition~(\ref{r:con_2}). Suppose $t_1 \gg^n t_2
  \sim_\alpha \tau$. If $t_2 \sim_\alpha \tau$ is obtained by rule
  $(\mathrm{A})$ or $(\mathrm{H})$ then $t_2 \equiv A_\tau$ for $\tau
  \in \B$ or $t_2 \equiv H$, so $t_1 \equiv t_2$ and the claim is
  obvious.

  If $t_2 \sim_\alpha \tau$ is obtained by rule $(\mathrm{K\omega})$
  or $(\mathrm{K\varepsilon})$ then $t_2 \equiv K t_2'$, $\tau \in
  \{\omega, \varepsilon\}$ and $t_2' \leadsto_{<\alpha} c$ where $c
  \in \{\top, \bot\}$, i.e., $t_2' \reduces_{<\alpha} t_2''
  \succ_{<\alpha} c$ for some $t_2''$. Hence $t_1 \equiv K t_1' \gg^n
  K t_2' \equiv t_2$, and thus $t_1' \gg_\zeta^n t_2'
  \reduces_{<\alpha} t_2'' \succ_{<\alpha} c$. By part~(\ref{r:con_3})
  the~IH there exists~$t_1''$ such that $t_1' \reduces_{R} t_1'' \gg^n
  t_2'' \succ_{<\alpha} c$. By part~(\ref{r:con_1}) of the IH we
  obtain $t_1' \leadsto c$. Hence $t_1 \sim \tau$.

  If $t_2 \sim_\alpha \tau$ is obtained by rule $(\mathrm{F\omega})$
  then $\tau = \omega$ and $t_2 \equiv \lambda f . \Xi r_1'
  r_2'$. Then we must have $t_1 \equiv \lambda f . \Xi r_1 r_2 \gg^n
  \lambda f . \Xi r_1' r_2'$ where $r_1 \gg^n r_1' \sim_{<\alpha}
  \varepsilon$. So by part~(\ref{r:con_2}) of the inductive hypothesis
  $r_1 \sim \varepsilon$. Therefore $t_1 \sim \omega = \tau$ by rule
  $(\mathrm{F\omega})$. If $t_2 \sim_\alpha \tau$ is obtained by
  rule~$(\mathrm{F\omega'})$ then the argument is similar.

  If $t_2 \sim_\alpha \tau$ is obtained by rule
  $(\mathrm{F})$ then $t_2 \equiv F r_1' r_2'$ and by
  Lemma~\ref{r:lem_f_gg_0} we obtain $t_1 \equiv F r_1 r_2 \gg^n F r_1'
  r_2'$ where $r_1 \gg^n r_1'$ and $r_2 \gg^n r_2'$. We have
  $\tau=\tau_1\to\tau_2$, $r_1 \gg^n r_1' \sim_{<\alpha} \tau_1$ and
  $r_2 \gg^n r_2' \sim_{<\alpha} \tau_2$. By part~(\ref{r:con_2}) of
  the IH we obtain $r_1 \sim \tau_1$ and $r_2 \sim \tau_2$. Therefore
  $t_1 \equiv F r_1 r_2 \sim \tau$ by rule $(\mathrm{F})$.

  If $t_2 \sim_\alpha \tau$ is obtained by rule~$(\mathrm{F'})$ then
  $t_2 \equiv \lambda x . \Xi r_1' (\lambda x . r_2'[z/fx])$,
  $\tau=\tau_1\to\tau_2$, $r_1' \sim_{<\alpha} \tau_1$ and $\lambda z
  . r_2' \sim_{<\alpha} \tau_2$. By Lemma~\ref{r:lem_f_gg} there are
  two cases.
  \begin{itemize}
  \item $t_1 \equiv \lambda f . \Xi r_1 (\lambda x . r_2[z/fx])$, $x,f
    \notin FV(r_2)$, $r_1 \gg^n r_1'$ and $r_2 \gg^n r_2'$. Then $r_2
    \gg^n r_1' \sim_{<\alpha} \tau_1$ and $\lambda z . r_2 \gg^n
    \lambda z . r_2' \sim_{<\alpha} \tau_2$. By part~(\ref{r:con_2})
    of the IH we obtain $r_1 \sim \tau_1$ and $\lambda z . r_2 \sim
    \tau_2$. Therefore $t_1 \equiv F r_1 r_2 \sim \tau$ by rule
    $(\mathrm{F'})$.
  \item $t_1 \equiv \lambda f . \Xi r_1 r_2$, $z \notin FV(r_2')$,
    $r_1 \gg^n r_1'$ and $r_2 \gg^n \lambda z . r_2'$. By
    part~(\ref{r:con_2}) of the IH we obtain $r_1 \sim \tau_1$ and
    $r_2 \sim \tau_2$. Since $K r_2' \sim_{<\alpha} \tau_2$, by
    Lemma~\ref{r:lem_k_sim} we have $\tau_2 \in
    \{\omega,\varepsilon\}$. Therefore $t_1 \equiv F r_1 r_2 \sim
    \tau$ by rule $(\mathrm{F''})$.
  \end{itemize}

  The remaining case is when $t_2 \sim_\alpha \tau$ is obtained by
  rule $(\mathrm{F''})$. Then $t_2 \equiv \lambda f . \Xi r_1' r_2'$,
  $t_1 \equiv \lambda f . \Xi r_1 r_2$, $r_1 \gg^n r_1' \sim_{<\alpha}
  \tau_1$ and $r_2 \gg^n r_2' \sim_{<\alpha} \tau_2 \in
  \{\omega,\varepsilon\}$. By part~(\ref{r:con_2}) of the IH we obtain
  $r_2 \sim \tau_1$ and $r_1 \sim \tau_2$. Hence $t_1 \sim
  \tau_1\to\tau_2$.

  Now we verify condition~(\ref{r:con_1}). If $t_2 \equiv \rho$ then
  $t_1 \gg \rho$. By~(\ref{r:canon_1}) in Fact~\ref{r:fact_canonical}
  we have $\rho \equiv \lambda x_1 \ldots x_n . c$, so by definition
  of~$\gg$, there exist a unary context $C$, a term $t'$, and a
  canonical term $\rho'$ such that $t_1 \equiv C[t']$, $\rho \equiv
  C[\rho']$ and $t' \succ \rho'$. If $C \equiv \rho$ then the claim is
  obvious. Otherwise $C \equiv \lambda x_1 \ldots x_k . \Box$ where $k
  \le n$, $\rho' \in \T_\tau$, and $\rho \in \T_{\omega^k\to\tau}$,
  by~(\ref{r:canon_2}) in Fact~\ref{r:fact_canonical}. By
  Lemma~\ref{r:lem_succ_abstraction} we obtain $t_1 \equiv C[t']
  \equiv \lambda x_1 \ldots x_k . t' \succ \lambda x_1 \ldots x_k
  . \rho' \equiv C[\rho'] \equiv \rho$.
  %% NOTE: assuming $t_1 \gg_\alpha \rho$, we can only conclude $t
  %% \succ_{\alpha+k} \rho$ by the above argument

  Next assume that $\rho \in \T_{\tau}$ where $\tau=\tau_1\to\tau_2
  \in \Tc_1$. Thus for all $t_3 \in \T_{\tau_1}$ there exists $t_2'$
  such that $t_2 t_3 \reduces_{<\alpha} t_2' \succ_{<\alpha}
  \F(\rho)(t_3)$. Then obviously $t_1 t_3 \gg^n t_2 t_3
  \reduces_{<\alpha} t_2'$, so by part~(\ref{r:con_3}) of the
  inductive hypothesis there exists $t_1'$ such that $t_1 t_3
  \reduces_{R} t_1' \gg^n t_2' \succ_{<\alpha} \F(\rho)(t_3)$. Using
  part~(\ref{r:con_1}) of the IH we obtain $t_1 t_3 \reduces_{R} t_1'
  \succ \F(\rho)(t_3)$. This implies $t_1 \succ \rho$.

  The remaining case to check is $\rho \in \T_o$. Suppose $\rho \equiv
  \top$, so $t_1 \gg^n t_2 \succ_{\alpha} \top$. If $\rho \equiv
  \bot$, i.e., $t_1 \gg^n t_2 \succ_{\alpha} \bot$, then proof is
  similar. We consider all possible forms of~$t_2$ according to the
  definition of $t_2 \succ_\alpha \top$. If $t_2 \equiv A_\tau c$ for
  $\tau \in \B$ then $t_1 \equiv t_2$, because if $c$ is a canonical
  constant of a base type~$\tau$ then the condition $t \succ c$
  implies $t \equiv c$. If $t_2 \equiv \top$ then $t_1 \succ t_2
  \equiv \top$ and the claim is obvious. Suppose
  condition~$(\Xi^\top)$ in the definition of~$t_2 \succ_\alpha \top$
  is satisfied. Then $t_1 \equiv \Xi r_1 r_2 \gg^n \Xi r_1' r_2'
  \equiv t_2$ where $r_1 \gg^n r_1'$ and $r_2 \gg^n r_2'$. By
  definition of $\succ_\alpha$ there exists $\tau$ such that $r_1'
  \sim_\alpha \tau$ and for all $t_3 \in \T_\tau$ we have $r_2' t_3
  \leadsto_{<\alpha} \top$, i.e., $r_2' t_3 \reduces_{<\alpha} t_3'
  \succ_{<\alpha} \top$. Since $r_1 \gg^n r_1' \sim_\alpha \tau$ we
  conclude that $r_1 \sim \tau$ by condition~(\ref{r:con_2}) which we
  have already verified in this inductive step. Because for all $t_3
  \in \T_\tau$ we have $r_2 t_3 \gg^n r_2' t_3 \reduces_{<\alpha} t_3'
  \succ_{<\alpha} \top$, so by part~(\ref{r:con_3}) of the IH for all
  $t_3 \in \T_\tau$ there exists~$t_3''$ such that $r_2 t_3
  \reduces_{R} t_3'' \gg^n t_3' \succ_{<\alpha} \top$. Hence $r_2 t_3
  \leadsto \top$ by applying part~(\ref{r:con_1}) of the IH. Therefore
  $t_1 \succ \top$ by the definition of~$\succ$. Finally, assume the
  condition~$(L^\top)$ in the definition of $t_2 \succ_\alpha \top$ is
  satisfied. Then $t_2 \equiv L t_2'$ with $t_2' \sim_\alpha \tau$ for
  some type~$\tau$. Since $t_1 \gg^n t_2$ we must have $t_1 \equiv L
  t_1'$ with $t_1' \gg^n t_2' \sim_\alpha \tau$. By
  condition~(\ref{r:con_2}), which we have already verified in this
  inductive step, we obtain $t_1' \sim \tau$. Therefore $t_1 \equiv L
  t_1' \succ \top$.

  It remains to prove~(\ref{r:con_3}). It suffices to consider a
  single reduction step, i.e., to show that $t_1 \gg^n t_2
  \contr_{\le\alpha} t_2'$ implies $t_1 \reduces_{R} t_1' \gg^n
  t_2'$. We have $t_1 \equiv C[r_1,\ldots,r_k]$ and $t_2 \equiv
  C[\rho_1,\ldots,\rho_k]$ where $r_i \succ \rho_i$ and $\rank(\rho_i)
  \le n$, for $i=1,\ldots,k$. Denote by $C_0[\rho_1,\ldots,\rho_k]$
  the contracted redex in~$t_2$, where the boxes in~$C_0$ correspond
  to appropriate boxes in~$C$. By~$C_e$ we denote the surrounding
  context satisfying $C \equiv C_e[C_0,\Box_1,\ldots,\Box_k]$. It
  follows from the definition of $R_\alpha$ that there are four
  possibilities: $C_0 \equiv \lambda x . C_1 x$ where $x \notin
  FV(C_1)$, $C_0 \equiv (\lambda x . C_1) C_2$, $C_0 \equiv c_0 C_1$
  for $c_0 \in \Sigma_{\tau_1\to\tau_2}$, or $C_0 \equiv \Box_i C_1$
  for some $1 \le i \le k$. In the first two cases we have $t_2's
  \equiv C_e[C_0'[\rho_1,\ldots,\rho_k], \rho_1,\ldots,\rho_k]$ where
  $C_0 \contr_{\le\alpha} C_0'$, so we may just take $t_1' \equiv
  C_e[C_0'[r_1,\ldots,r_k], r_1,\ldots,r_k]$.

  Otherwise the contraction in $t_2$ produces some canonical term
  $\rho$, i.e., $C_0[\rho_1,\ldots,\rho_k] \contr_{\le\alpha} \rho$. It
  suffices to prove:
  \begin{itemize}
  \item[$(\star)$] there exists $t$ such that $C_0[r_1,\ldots,r_k]
    \reduces_{R} t \succ \rho$, and if $t \not\equiv
    \rho$ then $\rank(\rho) \le n$.
  \end{itemize}
  Indeed, if $(\star)$ holds then simply take $t_1' \equiv
  C_e[t,r_1,\ldots,r_k]$. We have $t_1 \equiv
  C_e[C_0[r_1,\ldots,r_k],r_1,\ldots,r_k] \reduces_{R}
  C_e[t,r_1,\ldots,r_k] \equiv t_1'$ and $t_2' \equiv
  C_e[\rho,r_1,\ldots,r_k]$. Now it is easy to see that $t_1' \gg^n
  t_2'$: if $t \equiv \rho$ then we take
  $C_e[\rho,\Box_1,\ldots,\Box_k]$ as the context required by the
  definition of $\gg^n$, otherwise we take $C_e$ noting that $t \succ
  \rho$ and $\rank(\rho) \le n$.

  If $C_0 \equiv c_0 C_1$ then $C_1[\rho_1,\ldots,\rho_k]
  \succ_{<\alpha} \rho'$ where $\F(c)(\rho') \equiv \rho$. We conclude
  $C_1[r_1,\ldots,r_k] \succ \rho'$ by part~(\ref{r:con_1}) of the IH
  and the fact that $C_1[r_1,\ldots,r_k] \gg^n
  C_1[\rho_1,\ldots,\rho_k]$. Therefore $C_0[r_1,\ldots,r_k] \equiv c
  C_1[r_1,\ldots,r_k] \contr_{R} \rho$ and we are done.

  Suppose $C_0 \equiv \Box_i C_1$ where $1 \le i \le k$. First assume
  that $\rho_i$ is a canonical constant of type $\tau_1\to\tau_2$. As
  in the previous paragraph we have $C_1[\rho_1,\ldots,\rho_k]
  \succ_{<\alpha} \rho'$ where $\F(\rho_i)(\rho') \equiv \rho$, so
  $C_1[r_1,\ldots,r_k] \succ \rho'$ by part~(\ref{r:con_1}) of the
  IH. Obviously $\rank(\rho) = \rank(\tau_2) \le
  \rank(\tau_1\to\tau_2) = \rank(\rho_i) \le n$ and $\rank(\rho') =
  \rank(\tau_1) < \rank(\tau_1) + 1 \le \rank(\tau_1\to\tau_2) =
  \rank(\rho_i) \le n$. Let $r \equiv C_1[r_1,\ldots,r_k]$. We have $r
  \succ \rho'$ and $\rank(\rho') < n$, so $r_i r \gg^{<n} r_i \rho'$
  where the context required by the definition of $\gg^{<n}$ is $r_i
  \Box$. Since $r_i \succ \rho_i$ and the canonical type of $\rho_i$
  is a function type, we conclude by definition of $\succ$ that $r_i
  \rho' \leadsto \F(\rho_i)(\rho') \equiv \rho$. Note that we may have
  $r_i \equiv \rho_i$, but then the condition $r_i \rho' \leadsto
  \rho$ is satisfied anyway, by definition of $\F$. Therefore there
  exists $t'$ such that $r_i r \gg^{<n} r_i \rho' \reduces_{R} t'
  \succ \rho$. By part~(\ref{r:con_3}) of the inductive hypothesis
  there exists $t$ such that $r_i r \reduces_\zeta t \gg_\zeta^{<n} t'
  \succ \rho$. Applying part~(\ref{r:con_1}) of the IH we obtain $t
  \succ \rho$. Hence $C_0[r_1,\ldots,r_k] \equiv r_i
  C_1[r_1,\ldots,r_k] \equiv r_i r \reduces_{R} t \succ \rho$ where
  $\rank(\rho) \le n$, so~$(\star)$ holds.

  Now suppose that $\rho_i \equiv \lambda x_1 \ldots x_m . c$ for $m >
  0$. We have $C_0[r_1,\ldots,r_k] \equiv r_i C_1[r_1,\ldots,r_k]$
  with \mbox{$r_i \succ \rho_i$}. By the definition of $\succ$ we
  conclude that there exists $t$ such that $r_i C_1[r_1,\ldots,r_k]
  \reduces_{R} t \succ \lambda x_2 \ldots x_m . c \equiv
  \rho$. Obviously we also have $\rank(\rho) \le \rank(\rho_i) \le
  n$. Thus~$(\star)$ holds.
\end{proof}

\begin{sscorollary}\label{r:corollary_succ_stable}
  If $t \succ \rho_1$ and $C[\rho_1] \leadsto \rho_2$, then $C[t]
  \leadsto \rho_2$.
\end{sscorollary}

The above corollary states that our definition of $\succ$ is
correct. If $t \succ \rho_1$ then~$t$ behaves exactly like~$\rho_1$ in
every context~$C$ such that~$C[\rho_1]$ has an ``interesting''
interpretation.

\medskip

The following final lemmas show that the conditions on~$\Ts$ required
for a classical illative model are satisfied by~$\M$.

\begin{sslemma}\label{r:lem_h_t_f}
  If $H t \leadsto_\alpha \top$ then $t \leadsto_{<\alpha} \top$ or $t
  \leadsto_{<\alpha} \bot$.
\end{sslemma}

\begin{proof}
  Keeping in mind the convention regarding the meaning of $H t$, we
  note that if $H t \leadsto_\alpha \top$ then $H t
  \reduces_{\le\alpha} L (K t') \succ_\alpha \top$ where $t
  \reduces_{\le\alpha} t'$. Thus it suffices to show that for any
  term~$t$, if $L (K t) \succ_\alpha \top$ then $t \leadsto_\alpha
  \top$ or $t \leadsto_{\alpha+1} \bot$. Assume $L (K t) \succ_\alpha
  \top$. Then the condition~$(L^\top)$ must hold, so $K t \sim_\alpha
  \tau$ for some type~$\tau$. By Lemma~\ref{r:lem_k_sim} we have $\tau
  = \omega$ or $\tau=\varepsilon$. Assume $\tau=\omega$. The other
  case is analogous. By~(\ref{r:comm_6}) in Lemma~\ref{r:lem_commute}
  we have $K t t \leadsto_{<\alpha} \top$. Since $K t t \to_\beta t$, by
  Corollary~\ref{r:corollary_t_conv} we have $t \leadsto_{<\alpha} \top$.
\end{proof}

\begin{sslemma}\label{r:lem_canonical_succ}
  If $\rho \in \T_\tau$ and $\tau\ne\omega$ then $\rho \succ \rho$.
\end{sslemma}

\begin{proof}
  Induction on the size of~$\tau$.
\end{proof}

\begin{sslemma}\label{r:lem_sim_correct}
  If $t \sim_\alpha \tau$ then for all $t_0 \in \T_\tau$ we have $t
  t_0 \leadsto \top$.
\end{sslemma}

\begin{proof}
  Induction on $\alpha$. If $t \sim_\alpha \tau$ is obtained by
  rule~$(\mathrm{A})$, $(\mathrm{H})$, $(\mathrm{K\omega})$ or
  $(\mathrm{K\varepsilon})$, then the claim is obvious. If
  $\tau=\omega$ then the claim follows from~(\ref{r:comm_6}) in
  Lemma~\ref{r:lem_commute}. If $\tau=\varepsilon$ then the claim is
  also obvious. So we may assume $\tau=\tau_1\to\tau_2 \notin
  \{\omega,\varepsilon\}$. Then the only remaining cases are when $t
  \sim \tau$ is obtained by~$(\mathrm{F})$ or~$(\mathrm{F'})$. Then $t
  =_\beta F t_1 t_2$, $\tau = \tau_1\to\tau_2$, $t_1 \sim_{<\alpha}
  \tau_1$ and $t_2 \sim_{<\alpha} \tau_2$. Suppose $t_0 \in
  \T_{\tau_1\to\tau_2}$. Then for all $r_1 \in \T_{\tau_1}$ there
  exists $r_2 \in \T_{\tau_2}$ such that $t_0 r_1 \reduces_R r_2$, by
  Definition~\ref{r:def_succ}, because if $\tau_1\ne\omega$ then $r_1
  \succ r_1$ by Lemma~\ref{r:lem_canonical_succ}. Also, we have $F t_1
  t_2 t_0 =_{\le 0} \Xi t_1 \lambda y . t_2 (t_0 y)$. Hence $(\lambda
  y . t_2 (t_0 y)) r_1 \reduces_R t_2 r_2$. Because $t_2
  \sim_{<\alpha} \tau_2$, we have $t_2 r_2 \leadsto \top$ by the IH,
  so $(\lambda y . t_2 (t_0 y)) r_1 \leadsto \top$. Therefore $\Xi t_1
  \lambda y . t_2 (t_0 y) \succ \top$ by
  condition~$(\Xi_i^\top)$. Hence, by
  Corollary~\ref{r:corollary_t_conv}, we obtain $F t_1 t_2 t' \leadsto
  \top$.
\end{proof}

\begin{sslemma}\label{r:lem_succ_type}
  If $t_1 \sim_\alpha \tau$, $\tau \ne \omega$, $\tau \ne \varepsilon$
  and $t_1 t_2 \leadsto \top$, then $t_2 \leadsto \rho$ for some $\rho
  \in \T_\tau$.
\end{sslemma}

\begin{proof}
  Induction on $\alpha$. If $t_1 \sim_\alpha \tau$ is obtained by
  rule~$(\mathrm{A})$ then $t_1 \equiv A_\tau$ for $\tau \in \B$, and
  $A_\tau t_2 \reduces_{R} t' \succ_\top$. So $t' \equiv A_\tau
  t_2'$ where $t_2 \reduces_{R} t_2'$. By
  Definition~\ref{r:def_succ} we have $t_2' \equiv c$ for $c \in
  \T_\tau$. Hence $t_2 \leadsto c$. If $t_1 \sim_\alpha \tau$ is obtained by
  rule~$(\mathrm{H})$ then $t_1 \equiv H$ and $t_2 \leadsto c \in
  \{\top, \bot\}$ by Lemma~\ref{r:lem_h_t_f}.

  The only remaining case is when $t_1 \sim_\alpha \tau =
  \tau_1\to\tau_2$ is obtained by~$(\mathrm{F})$
  or~$(\mathrm{F'})$. Then $t_1 =_\beta F r_1 r_2 \sim_\alpha \tau =
  \tau_1\to\tau_2$ where $r_1 \sim_{<\alpha} \tau_1$, $r_2
  \sim_{<\alpha} \tau_2$. We may assume $\tau_1 \ne \varepsilon$,
  $\tau_2 \ne \omega$ and $\tau_2 \ne \varepsilon$, since otherwise
  $\tau = \omega$ or $\tau = \varepsilon$. By
  Corollary~\ref{r:corollary_t_conv} we have $\Xi r_1 \lambda y . r_2
  (t_2 y) \leadsto \top$, so $\Xi r_1' r_2' \succ \top$ where $r_1
  \reduces_{R} r_1'$, $\lambda y . r_2 (t_2 y) \reduces_{R} r_2'$. By
  inspecting Definition~\ref{r:def_succ} we see that the only possible
  way for $\Xi r_1' r_2' \succ \top$ to hold is when
  condition~$(\Xi^\top)$ is satisfied, i.e., there exists $\tau'$ such
  that $r_1' \sim \tau'$ and for all $t_3 \in \T_{\tau'}$ we have
  $r_2' t_3 \leadsto \top$. By~(\ref{r:comm_4}) in
  Lemma~\ref{r:lem_commute} we have $r_1' \sim_{<\alpha} \tau_1$, so
  it follows from~(\ref{r:comm_5}) in Lemma~\ref{r:lem_commute} that
  $\tau' = \tau_1$. Therefore for any $t_3 \in \T_{\tau_1}$ we have
  $r_2' t_3 \leadsto \top$. Since $r_2 (t_2 t_3) =_{\le 0} (\lambda y
  . r_2 (t_2 y)) t_3 \reduces_{R} r_2' t_3$, we obtain by
  Corollary~\ref{r:corollary_t_conv} that $r_2 (t_2 t_3) \leadsto
  \top$ for any $t_3 \in \T_{\tau_1}$. Because $r_2 \sim_{<\alpha}
  \tau_2$ where $\tau_2 \ne \omega$ and $\tau_2 \ne \varepsilon$, we
  conclude by the inductive hypothesis that the following condition
  holds:
  \begin{itemize}
  \item[$(\star)$] for all $t_3 \in \T_{\tau_1}$ there exists $\rho_2
    \in \T_{\tau_2}$ such that $t_2 t_3 \leadsto \rho_2$.
  \end{itemize}
  Note that $\rho_2$ depends on~$t_3$.

  If $\tau_1 \ne \omega$ then $\T_{\tau_1\to\tau_2}$ contains a
  constant for every set-theoretical function from $\T_{\tau_1}$ to
  $\T_{\tau_2}$. In particular it contains a constant~$c$ such that
  for every $\rho_1 \in \T_{\tau_1}$ we have $\F(c)(\rho_1) \equiv
  \rho_2$ where $\rho_2 \in \T_{\tau_2}$ is a term depending
  on~$\rho_1$ such that $t_2 \rho_1 \leadsto \rho_2$. Such a~$\rho_2$
  exists by~$(\star)$. Therefore by definition of~$\succ$ we have $t_2
  \succ c \in \T_\tau$.

  If $\tau_1 = \omega$ then it suffices to show that there exists a
  single $\rho' \in \T_{\tau_2}$ such that for all~$t_3$ we have $t_2
  t_3 \leadsto \rho'$. Indeed, if this holds then $t_2 \succ K \rho'
  \in \T_{\omega\to\tau_2} = \T_\tau$. Let~$x$ be a
  variable. Obviously $x \in \T_\omega$, so by $(\star)$ there exists
  $\rho' \in \T_{\tau_2}$ such that $t_2 x \leadsto \rho'$, i.e., $t_2
  x \reduces_{R} t' \succ \rho'$ for some term~$t'$. Taking $C \equiv
  t_2 \Box$, we conlude by conditon~(\ref{r:lcsp_1}) in
  Lemma~\ref{r:lem_nu_context} that $t' \equiv C'[x]$ where $C[t_3]
  \reduces_{R} C'[t_3]$ for any term~$r$. By
  condition~(\ref{r:lcsp_2}) in Lemma~\ref{r:lem_nu_context} we have
  $C'[t_3] \succ \rho'$ for any term $t_3$. Therefore for any~$t_3$
  there exists~$t_3'$ such that $t_2 t_3 \reduces_{R} t_3' \succ
  \rho'$, i.e., $t_2 t_3 \leadsto \rho'$. This $\rho'$ depends only
  on~$x$, but not on~$t_3$, so our claim has been established.
\end{proof}

\begin{sslemma}\label{r:lem_cikm_1_3_4}
  The following conditions are satisfied.
  \begin{itemize}
  \item If $L t_1 \leadsto \top$ and for all $t_3$ such that $t_1 t_3
    \leadsto \top$ we have $t_2 t_3 \leadsto \top$, then $\Xi t_1 t_2
    \leadsto \top$.
  \item If $L t_1 \leadsto \top$ and for all $t_3$ such that $t_1 t_3
    \leadsto \top$ we have $H (t_2 t_3) \leadsto \top$, then $H (\Xi
    t_1 t_2) \leadsto \top$.
  \item If $L t_1 \leadsto \top$, and either $L t_2 \leadsto \top$ or
    there is no~$t_3$ such that $t_1 t_3 \leadsto \top$, then $L (F
    t_1 t_2) \leadsto \top$.
  \end{itemize}
\end{sslemma}

\begin{proof}
  Suppose $L t_1 \leadsto \top$. By definitions we have $t_1
  \reduces_{R} t_1' \sim \tau$ for some type $\tau$.

  Assume that for all~$t_3$ such that $t_1 t_3 \leadsto \top$ we have
  $t_2 t_3 \leadsto \top$. Let $t_0 \in \T_\tau$. Then by
  Lemma~\ref{r:lem_sim_correct} we obtain $t_1' t_0 \leadsto
  \top$. Because $t_1 t_0 =_{R} t_1' t_0$, by
  Corollary~\ref{r:corollary_t_conv} we conclude $t_1 t_0 \leadsto
  \top$. Then by assumption $t_2 t_0 \leadsto \top$. Therefore
  by~$(\Xi_i^\top)$ we obtain $\Xi t_1' t_2 \succ \top$. Hence $\Xi
  t_1 t_2 \leadsto \top$.

  Assume that for all $t_3$ such that $t_1 t_3 \leadsto \top$ we have
  $H (t_2 t_3) \leadsto \top$, so $t_2 t_3 \leadsto \top$ or $t_2 t_3
  \leadsto \bot$ by Lemma~\ref{r:lem_h_t_f}. If for all $t_3$ such
  that $t_1 t_3 \leadsto \top$ we have $t_2 t_3 \leadsto \top$, then
  $\Xi t_1 t_2 \leadsto \top$ by the previous paragraph. Otherwise
  using Lemma~\ref{r:lem_sim_correct},
  Corollary~\ref{r:corollary_t_conv} and~$(\Xi^\bot)$ we may conclude
  $\Xi t_1 t_2 \leadsto \bot$ by an argument analogous to the previous
  paragraph. In any case $H (\Xi t_1 t_2) \leadsto \top$
  by~$(L^\top)$, and $(\mathrm{K\omega})$ or
  $(\mathrm{K\varepsilon})$.

  Assume $L t_2 \leadsto \top$. Then $t_2 \reduces_R t_2' \sim
  \tau'$. Then $F t_1' t_2' \sim \tau\to\tau'$, so $L (F t_1 t_2)
  \reduces_R L (F t_1' t_2') \succ \top$.

  Finally, assume there is no~$t_3$ such that $t_1 t_3 \leadsto
  \top$. Then there is no~$t_3$ such that $t_1't_2 \leadsto \top$. By
  Lemma~\ref{r:lem_succ_type} and~(\ref{r:comm_6}) in
  Lemma~\ref{r:lem_commute} we must have $\tau=\varepsilon$. Then $L
  (F t_1 t_2) \leadsto \top$ by~$(\mathrm{F\omega})$, $(L^\top)$ and
  Corollary~\ref{r:corollary_t_conv}.
\end{proof}

\begin{sslemma}\label{r:lem_cikm_2}
  If $\Xi t_1 t_2 \leadsto \top$ then for all terms~$t_3$ such that
  $t_1 t_3 \leadsto \top$ we have $t_2 t_3 \leadsto \top$.
\end{sslemma}

\begin{proof}
  If $\Xi t_1 t_2 \leadsto \top$ then $\Xi t_1 t_2 \reduces_{R} \Xi
  t_1' t_2' \succ \top$ where $t_1 \reduces_{R} t_1'$ and $t_2
  \reduces_{R} t_2'$. The only possibility for $\Xi t_1' t_2' \succ
  \top$ to hold is that condition~$(\Xi^\top)$ holds for $\Xi t_1'
  t_2'$. Thus $t_1' \sim \tau$ for some type $\tau$. Suppose $t_1 t_3
  \leadsto \top$. By Corollary~\ref{r:corollary_t_conv} we have $t_1'
  t_3 \leadsto \top$. Because $t_2 t_3 \reduces_{R} t_2' t_3$, it
  suffices to show that $t_2' t_3 \leadsto \top$. If $\tau = \omega$
  then this is obvious by definition of~$(\Xi^\top)$. We cannot have
  $\tau = \varepsilon$, since if $t_1' \sim \varepsilon$ then
  by~(\ref{r:comm_6}) in Lemma~\ref{r:lem_commute} and by
  Corollary~\ref{r:corollary_leadsto_consistent} there is no $t$ such
  that $t_1' t \leadsto \top$. If $t_1' \sim \tau \ne \omega$ and
  $\tau \ne \varepsilon$, then we use Lemma~\ref{r:lem_succ_type} to
  conclude that there exist $t_3'$ and $\rho \in \T_\tau$ such that
  $t_3 \reduces_{R} t_3' \succ \rho$. Because~$(\Xi_i^\top)$ holds for
  $\Xi t_1' t_2'$, $t_1' \sim \tau$ and $\rho \in \T_\tau$, we have
  $t_2' \rho \leadsto \top$. Since $t_3' \succ \rho$, taking $C \equiv
  t_2' \Box$ we conclude by Corollary~\ref{r:corollary_succ_stable}
  that $t_2' t_3' \leadsto \top$, so $t_2' t_3 \leadsto \top$.
\end{proof}

\begin{sstheorem}
  The systems $\I_\omega^c$ and $\I_\omega$ are strongly consistent,
  i.e., $\Xi H I$ is not derivable in them.
\end{sstheorem}

\begin{proof}
  We verify that the structure~$\M$ constructed in
  Definition~\ref{r:def_model} is a one-state classical illative model
  for~$\I_\omega^c$. It follows from Lemma~\ref{r:lem_extensional}
  that the combinatory algebra of $\M$ is
  extensional. Corollary~\ref{r:corollary_t_conv} implies that $[t]_R
  \in \Ts$ is equivalent to $t \leadsto \top$. We need to check the
  conditions stated in
  Fact~\ref{r:fact_classical_illat_model}. Conditions~(\ref{r:cikm_1}),
  (\ref{r:cikm_3}) and~(\ref{r:cikm_4}) follow from
  Lemma~\ref{r:lem_cikm_1_3_4}. Condition~(\ref{r:cikm_2}) follows
  from Lemma~\ref{r:lem_cikm_2}. Conditions~(\ref{r:cikm_5}),
  (\ref{r:cikm_6}) and~(\ref{r:cikm_7}) follow from definitions.

  It is also easy to see that $\not\forces_\M \Xi H I$. Indeed,
  otherwise we would have $\Xi H I \leadsto \top$, which is possible
  only when~$(\Xi^\top)$ is satisfied for $\Xi H I$. Thus $H \sim
  \tau$ for some type~$\tau$, and for all $t \in \T_\tau$ we have $I t
  \leadsto \top$, so $t \leadsto \top$ by
  Corollary~\ref{r:corollary_t_conv}. It is easily verified by
  inspecting the definitions that we must have $\tau = o$. But then
  $\bot \leadsto \top$ which is impossible by
  Corollary~\ref{r:corollary_leadsto_consistent}.

  Therefore, by the soundness part of
  Theorem~\ref{r:thm_ikm_complete}, the term $\Xi H I$ is not
  derivable in~$\I_\omega^c$, and hence neither in~$\I_\omega$, which
  is a subsystem of~$\I_\omega^c$.
\end{proof}

\section{The embedding}\label{r:sec_embedding}

In this section a syntactic translation from the terms
of~$\mathrm{PRED2}_0$ into the terms of~$\I_0$ is defined and proven
complete for~$\I_0$. The translation is a slight extension of that
from~\cite{illat01}. The method of the completeness proof is by model
construction analogous to that in the previous section. Relinquishing
quantification over predicates and restricting arguments of functions
to base types allows us to significantly simplify this construction
and to extend it to more than one state.

We use the notation~$\Tc$ for the set of types
of~$\mathrm{PRED2}_0$. Recall that~$\Tc$ is defined by the grammar
$\Tc \;::=\; o \;|\; \B \;|\; \B \rightarrow \Tc$, where~$\B$ is a
specific set of base types. We assume that~$\B$ corresponds exactly to
the base types used in the definiton of~$\I_0$. We fix a signature for
$\mathrm{PRED2}_0$, and by $\Sigma_\tau$ denote the set of constants
of type $\tau$ in this signature. We always assume that all variables
of $\mathrm{PRED2}_0$ are present in the set of variables of~$\I_0$.

Recall that by~$\T(\Sigma)$ we denote the set of type-free lambda
terms over a set of primitive constants $\Sigma$, which is assumed to
contain $\Xi$, $L$ and $A_\tau$ for each $\tau \in \B$. We also assume
that $\Sigma$ contains every constant $c \in \Sigma_\tau$ for any
$\tau \in \Tc$. For the sake of uniformity, we will sometimes use the
notation $A_o$ for $H$. For every composite type
$\tau=\tau_1\rightarrow\tau_2 \in \Tc$ we inductively define $A_\tau =
F A_{\tau_1} A_{\tau_2}$. We use the same notational conventions
concerning $K t$, $H t$, etc. as in Section~\ref{r:sec_construction}.

\begin{definition}
  We define inductively a map $\transl{-}$ from the terms of
  $\mathrm{PRED2}_0$ to $\T(\Sigma)$ as follows:
  \begin{itemize}
  \item $\transl{x} = x$ for a variable $x$,
  \item $\transl{c} = c$ for a constant $c$,
  \item $\transl{t_1 t_2} = \transl{t_1} \transl{t_2}$,
  \item $\transl{\varphi \supset \psi} = \transl{\varphi} \supset \transl{\psi}$,
  \item $\transl{\forall x . \varphi} = \Xi A_\tau \lambda x . \transl{\varphi}$ for
    $x \in V_\tau$.
  \end{itemize}

  We extend the map to finite sets of formulas by defining
  $\transl{\Delta}$ to be the image of~$\transl{-}$ on~$\Delta$. We
  also define a mapping~$\Gamma$ from sets of formulas to subsets
  of~$\T(\Sigma)$, which is intended to provide a context for a set of
  formulas. For a finite set of formulas~$\Delta$ we define
  $\Gamma(\Delta)$ to contain the following:
  \begin{itemize}
  \item $A_\tau x$ for all $x \in FV(\Delta)$ s.t. $x \in V_\tau$, and
    all types $\tau$,
  \item $A_\tau c$ for all $c \in \Sigma_\tau$, and all types $\tau$,
  \item $L A_\tau$ for all $\tau \in \B$,
  \item $A_\tau y$ for all $\tau \in \B$ and some $y \in V_\tau$ such
    that $y \notin FV(\Delta)$.
  \end{itemize}
\end{definition}

\begin{lemma}\label{r:lem_types_non_empty}
  For any $\tau \in \Tc$ and any $\Delta$ there exists a term $t$ such
  that $\Gamma(\Delta) \proves_{\I_0} A_\tau t$.
\end{lemma}

\begin{proof}
  First note that by a straightforward induction on the size of~$\tau$
  we obtain $\Gamma(\Delta) \proves L A_\tau$ for any type~$\tau$.

  We prove the lemma by induction on the size of~$\tau$. If $\tau \in
  \B$ then $A_\tau y \in \Gamma(\Delta)$ for some variable~$y$. If
  $\tau = o$ then notice that e.g. $\proves H (L H)$. If $\tau =
  \tau_1\to\tau_2$ then we need to prove that $\Gamma(\Delta) \proves
  F A_{\tau_1} A_{\tau_2} t$ for some term~$t$. Because
  $\Gamma(\Delta) \proves L A_{\tau_1}$, it suffices to show that
  $\Gamma(\Delta), A_{\tau_1} x \proves A_{\tau_2} (t x)$ for some
  term $t$ and some $x \notin FV(\Gamma(\Delta), t)$. By the inductive
  hypothesis there exists a term~$t_2$ such that $\Gamma(\Delta)
  \proves A_{\tau_2} t_2$. So just take $x \notin FV(\Gamma(\Delta),
  t_2)$ and $t \equiv K t_2$.
\end{proof}

\begin{theorem}
  The embedding is sound, i.e., $\Delta \proves_{\mathrm{PRED2}_0}
  \varphi$ implies $\transl{\Delta}, \Gamma(\Delta, \varphi)
  \proves_{\I_0} \transl{\varphi}$.
\end{theorem}

\begin{proof}
  Induction on the length of derivation of $\Delta
  \proves_{\mathrm{PRED2}_0} \varphi$, using
  Lemma~\ref{r:lem_illat_admissible}. The only interesting case is
  with modus-ponens, as from the inductive hypothesis we may only
  directly derive the judgement $\transl{\Delta}, \Gamma(\Delta,
  \psi), \Gamma(\varphi) \proves_{\I_0} \transl{\psi}$. To get rid of
  $\Gamma(\varphi)$ on the left, we note that if $t \in
  \Gamma(\varphi) \setminus \Gamma(\Delta, \psi)$ then $t \equiv
  A_\tau x$ for $x \in FV(\varphi) \setminus FV(\Delta, \psi)$. Now,
  by Lemma~\ref{r:lem_types_non_empty} there exists $t'$ such that
  $\Gamma(\Delta, \psi) \proves_{\I_0} A_\tau t'$. It is not difficult
  to show by induction on the length of derivation that
  $\transl{\Delta}, \Gamma(\Delta, \psi), \Gamma(\varphi)[x/t']
  \proves_{\I_0} \transl{\psi}$, i.e., that we may change $A_\tau x$ on
  the left to $A_\tau t'$. To eliminate $A_\tau t'$ altogether, it
  remains to notice that if $\Gamma, t_1 \proves_{\I_0} t_2$ and
  $\Gamma \proves_{\I_0} t_1$ then $\Gamma \proves_{\I_0} t_2$.

  If we had extended our semantics for $\mbox{PRED2}_0$ a bit by
  allowing non-constant domains, then we could also give a relatively
  simple semantic proof by transforming any illative Kripke model for
  $\I_0$ to a Kripke model for $\mbox{PRED2}_0$, and appealing to the
  completeness part of Theorem~\ref{r:thm_ikm_complete}.
\end{proof}

The rest of this section is devoted to proving that the embedding is
also complete.

\medskip

Let $\N$ be a Kripke model for $\mbox{PRED2}_0$. We will now construct
an illative Kripke model $\M$ such that $\M$ will ``mirror'' $\N$,
i.e., exactly the translations of true statements in a state of $\N$
will be true in the corresponding state of $\M$. This construction is
the crucial step in the completeness proof. It is similar to the
construction given in Section~\ref{r:sec_construction}. For the rest
of this section we assume a fixed~$\N$.

We define a set of primitive constants $\Sigma^+$ and the sets
$\Sigma_\tau$ of canonical constants of type $\tau$, just like in
Definition~\ref{r:def_canonical}, but restricting ourselves only to
the types in~$\Tc$ (i.e.~the types of~$\mbox{PRED2}_0$). Note that
there is a bijection~$\delta_\tau$ between~$\Sigma_\tau$
and~$\D_\tau^\N$. We often drop the subscript in~$\delta_\tau$. We
also include in~$\Sigma^+$ an infinite set~$\Sigma^\nu$ of
\emph{external constants}. Note that $\Sigma^+$ is disjoint from the
signature~$\Sigma$ of~$\M$ which we defined earlier. The terms
over~$\Sigma$ form the syntax. The terms over~$\Sigma^+$ are used to
build the model. To every constant $c \in \Sigma$ corresponds exactly
one constant $c^+ \in \Sigma^+$ such that $\valuation{c}{\N}{} =
\delta(c^+)$. This correspondence, however, need not be injective, as
there may be another constant $c' \in \Sigma$, $c' \ne c$, such that
$\valuation{c'}{\N}{} = \delta(c^+)$.

Let $\Sc$ be the set of states of $\N$. By $\top \in \Sigma_o$ we
denote the constant such that $\varsigma_\N(\delta(\top)) = \Sc$, and
by $\bot \in \Sigma_o$ the constant such that
$\varsigma_\N(\delta(\bot)) = \emptyset$. In what follows $\rho$,
$\rho'$, etc., stand for~$\top$ or~$\bot$. Note that~$\Sigma_o$ may
contain other elements in addition to~$\top$ and~$\bot$. In this
section we use $t$, $t_1$, $t_2$, etc., for \emph{closed} terms,
unless otherwise stated.

\begin{definition}\label{r:def_reductions_2}
  We construct a reduction system $R$ as follows. The terms of $R$ are
  the type-free lambda-terms over $\Sigma^+$. The reduction rules of $R$
  are as follows:
  \begin{itemize}
  \item rules of $\beta$- and $\eta$-reduction,
  \item $c c_1 \contr c_2$ for $c \in \Sigma_{\tau_1\to\tau_2}$, $c_1
    \in \Sigma_{\tau_1}$ and $c_2 \in \Sigma_{\tau_2}$ such that
    $\F(c)(c_1) = c_2$.
  \end{itemize}
  It is easy to see that $R$ has the Church-Rosser property.
\end{definition}

\begin{definition}\label{r:def_sim_succ_2}
  For each ordinal $\alpha$ and each state $s \in \Sc$ we inductively
  define a relation~$\succ_\alpha^s$ between terms and~$\top$
  or~$\bot$. The notations $\succ_{<\alpha}^s$,
  $\leadsto_{<\alpha}^s$, etc., have analogous meaning to those in
  Section~\ref{r:sec_construction}.

  We postulate $t \succ_{\alpha}^s \top$ for $\alpha \ge 0$ and all
  closed terms~$t$ such that:
  \begin{enumerate}
  \item $t \equiv c$ for some $c \in \Sigma_o$ such that $s \in
   \varsigma_\N(\delta(c))$, or
  \item $t \equiv L A_{\tau}$ for some $\tau \in \B$, or
  \item $t \equiv L H$, or
  \item $t \equiv A_{\tau} c$ for $\tau \in \B$ and $c \in
    \Sigma_\tau$, or
  \item $t \equiv H c$ for $c \in \Sigma_o$.
  \end{enumerate}

  When $\alpha > 0$ we postulate $t \succ_\alpha^s \top$ for all
  closed terms~$t$ such that one of the following holds:
  \begin{enumerate}
  \item[$(\Xi_\top)$] $t \equiv \Xi A_\tau t_1$ where $\tau \in \B
    \cup \{o\}$ and~$t_1$ is such that for all $s' \ge s$ and all $c
    \in \Sigma_\tau$ we have $t_1 c \leadsto_{<\alpha}^{s'} \top$,
  \item[$(\Ps_\top)$] $t \equiv \Xi (K t_1) t_2$ where
    \begin{itemize}
    \item $t_1 \leadsto_{<\alpha}^s \top$ or $t_1 \leadsto_{<\alpha}^s
      \bot$, and
    \item for all $s' \ge s$ such that $t_1 \leadsto_{<\alpha}^{s'}
      \top$ we have $t_2 \reduces_R K t_2'$ with $t_2'
      \succ_{<\alpha}^{s'} \top$,
    \end{itemize}
  \item[$(H_\top)$] $t \equiv H t_1$, and $t_1 \leadsto_{<\alpha}^s
    \top$ or $t_1 \leadsto_{<\alpha}^s \bot$.
  \end{enumerate}

  Finally, we postulate $t \succ_\alpha^s \bot$ for $\alpha \ge 0$ and
  all closed terms~$t$ such that one of the following holds:
  \begin{enumerate}
  \item[$(c_\bot)$] $t \equiv c \in \Sigma_o$ and $s \notin
    \varsigma_\N(\delta(c))$,
  \item[$(\Xi_\bot)$] $t \equiv \Xi A_\tau t_1$ and $\tau \in \B \cup
    \{o\}$, and
    \begin{itemize}
    \item for all $c \in \Sigma_\tau$ and all $s' \ge s$ we have $t_1
      c \leadsto_{<\alpha}^{s'} \top$ or $t_1 c
      \leadsto_{<\alpha}^{s'} \bot$,
    \item there exist a constant $c \in \Sigma_\tau$ and a state $s'
      \ge s$ such that $t_1 c \leadsto_{<\alpha}^{s'} \bot$,
    \end{itemize}
  \item[$(\Ps_\bot)$] $t \equiv \Xi (K t_1) (K t_2)$, and
    \begin{itemize}
    \item $t_1 \leadsto_{<\alpha}^s \top$ or $t_1 \leadsto_{<\alpha}^s
      \bot$, and
    \item for all $s' \ge s$ such that $t_1
      \leadsto_{<\alpha}^{s'} \top$ we have $t_2
      \leadsto_{<\alpha}^{s'} \top$ or $t_2 \leadsto_{<\alpha}^{s'}
      \bot$.
    \item there exists $s' \ge s$ such that $t_1
      \leadsto_{<\alpha}^{s'} \top$ and $t_2 \leadsto_{<\alpha}^{s'}
      \bot$.
    \end{itemize}
  \end{enumerate}
\end{definition}

In~\cite{Czajka2013JSL} this definition is incorrect. In fact,
Lemma~5.9 of~\cite{Czajka2013JSL} is false, because of the presence of
type~$\varepsilon$. To correct this we need to separately consider the
case when~$\Xi$ encodes implication, which is done here by means of
the rules~$(\Ps_\top)$ and~$(\Ps_\bot)$. This change requires
reworking the subsequent correctness proof.

With the corrected definition, it is not obvious that for $\alpha \le
\beta$ we have $\succ_\alpha^s\; \subseteq \;\succ_\beta^s$. We will
show this only in Lemma~\ref{r:lem_monotone}. However, for $\alpha \le
\beta$ we obviously have $\succ_{<\alpha}^s\; \subseteq
\;\succ_{<\beta}^s$, and consequently $\leadsto_{<\alpha}^s\;
\subseteq \;\leadsto_{<\beta}^s$.

\begin{lemma}\label{r:lem_succ_sim_red}
  If $t_1 \succ_\alpha^s \rho$ and $t_1 \reduces_R t_2$ then $t_2
  \succ_\alpha^s \rho$.
\end{lemma}

\begin{proof}
  This follows by an easy induction on~$\alpha$, using the
  Church-Rosser property of~$R$.
\end{proof}

\begin{corollary}\label{r:corollary_t_conv_2}
  If $t =_R t'$ then $t \leadsto_\alpha^s \rho$ is equivalent to $t'
  \leadsto_\alpha^s \rho$.
\end{corollary}

\begin{corollary}\label{r:corollary_t_conv_3}
  If $t \leadsto_\alpha^s \top$ and $t \leadsto_\alpha^s \bot$ then
  there exists~$t'$ such that $t' \succ_\alpha^s \top$ and $t'
  \succ_\alpha^s \bot$.
\end{corollary}

\begin{lemma}\label{r:lem_upward_closed}
  For all ordinals $\alpha$ and all $s \in \Sc$ we have:
  \begin{enumerate}
  \item if $t \succ_\alpha^s \top$ and $s' \ge s$ then $t
    \succ_\alpha^{s'} \top$,
  \item if $t \succ_\alpha^s \bot$ and $s' \ge s$ then $t
    \succ_\alpha^{s'} \top$ or $t \succ_\alpha^{s'} \bot$.
  \end{enumerate}
\end{lemma}

\begin{proof}
  Induction on $\alpha$.
  \begin{enumerate}
  \item Follows directly from the inductive hypothesis.
  \item The only non-obvious cases are with~$(\Xi_\bot)$
    and~$(\Ps_\bot)$. Suppose $t \equiv \Xi A_\tau t_1 \succ_\alpha^s
    \bot$ with:
    \begin{itemize}
    \item for all $c \in \Sigma_\tau$ and all $s'' \ge s$ we have $t_1
      c \leadsto_{<\alpha}^{s''} \top$ or $t_1 c
      \leadsto_{<\alpha}^{s''} \bot$,
    \item there exist a constant $c \in \Sigma_\tau$ and a state $s''
      \ge s$ such that $t_1 c \leadsto_{<\alpha}^{s''} \bot$,
    \end{itemize}
    Let $s' \ge s$. The first condition obviously still holds
    with~$s'$ substituted for~$s$. If the second condition does not
    hold, then by the first condition:
    \begin{itemize}
    \item for all $c \in \Sigma_\tau$ and all $s'' \ge s'$ we have
      $t_1 c \leadsto_{<\alpha}^{s''} \top$.
    \end{itemize}
    This implies $t \succ_\alpha^{s'} \top$. The argument
    for~$(\Ps_\bot)$ is analogous.
  \end{enumerate}
\end{proof}

\begin{corollary}\label{r:corollary_upward_closed}
  If $t \leadsto_\alpha^s \top$ then $t \leadsto_\alpha^{s'} \top$ for
  $s' \ge s$.
\end{corollary}

\begin{remark}\label{r:rem_problem_hol}
  The necessity of the above corollary is precisely the reason why it
  is not easy to extend this construction to the case of full
  higher-order intuitionistic logic, i.e., when we have functions and
  predicates of all types and more than one state. In that case we
  would need separate reduction systems~$R_\alpha^s$ for each~$s$
  and~$\alpha$, similarily to what is done in
  Section~\ref{r:sec_construction}. But then it would not be the case
  that $R_\alpha^s \subseteq R_\alpha^{s'}$ for $s' \ge s$. Roughly
  speaking, this is because $t \succ_\alpha^s \bot$ is interpreted as
  ``$t$ is not true in state~$s$ basing on what we know at
  stage~$\alpha$'', and not as ``$t$ is false in state~$s$''. Thus we
  may have $t \succ_\alpha^s \bot$ and $t \succ_\alpha^{s'} \top$ for
  some $s' \ge s$. This by itself is not yet a fatal obstacle, because
  we really only care about $t \leadsto_\alpha^s \top$ being
  monotonous w.r.t. state ordering. However, the condition $t
  \succ_\alpha^s \bot$ would be used to define~$R_\alpha^s$, which
  would make $R_\alpha^s$ non-monotonous w.r.t~$s$. Thus $t
  \leadsto_\alpha^s \top$ would not be monotonous either, as it is
  equivalent to $t \reduces_{R_\alpha^s} t' \succ_\alpha^s
  \top$. Hence the corollary would fail. This explains why we do not
  simply give a single construction generalizing both the present one
  and the one from Section~\ref{r:sec_construction}.
\end{remark}

\begin{lemma}\label{r:lem_monotone}
  For all ordinals $\alpha$ and all $s \in \Sc$ we have:
  \begin{enumerate}
  \item if $t \succ_{<\alpha}^s \rho$ then $t \succ_\alpha^s \rho$,
  \item if $t \succ_\alpha^s \top$ then $t \not\succ_\alpha^s \bot$.
  \end{enumerate}
\end{lemma}

\begin{proof}
  Induction on~$\alpha$. First note that the inductive hypothesis and
  Corollary~\ref{r:corollary_t_conv_3} imply:
  \begin{itemize}
  \item[$(\star)$] if $t \leadsto_{<\alpha}^s \top$ then $t
    \not\leadsto_{<\alpha} \bot$.
  \end{itemize}

  Now, we check the conditions (1) and (2).
  \begin{enumerate}
  \item The problem is with the universal quantification
    in~$(\Ps_\top)$ and~$(\Ps_\bot)$. For instance,
    consider~$(\Ps_\top)$, i.e., $t \equiv \Xi (K t_1) t_2
    \succ_{\beta}^s \top$ for some $\beta < \alpha$, with:
    \begin{itemize}
    \item $t_1 \leadsto_{<\beta}^s \top$ or $t_1 \leadsto_{<\beta}^s
      \bot$,
    \item for all $s' \ge s$ such that $t_1 \leadsto_{<\beta}^{s'}
      \top$ we have $t_2 \reduces_R K t_2'$ with $t_2'
      \succ_{<\beta}^{s'} \top$.
    \end{itemize}
    Of course, we have $t_1 \leadsto_{<\alpha}^s \top$ or $t_1
    \leadsto_{<\alpha}^s \top$. Suppose $s' \ge s$ and $t_1
    \leadsto_{<\alpha}^{s'} \top$. If $t_1 \leadsto_{<\beta}^s \top$,
    then $t_1 \leadsto_{<\beta}^{s'} \top$ by
    Lemma~\ref{r:lem_upward_closed}. Thus $t_2 \reduces_R K t_2'$ with
    $t_2' \succ_{<\beta}^{s'} \top$, so also $t_2'
    \succ_{<\alpha}^{s'} \top$. If $t_1 \leadsto_{<\beta}^s \bot$,
    then $t_1 \leadsto_{<\beta}^{s'} \top$ or $t_1
    \leadsto_{<\beta}^{s'} \top$ by
    Lemma~\ref{r:lem_upward_closed}. The case $t_1
    \leadsto_{<\beta}^{s'} \top$ has just been considered. So suppose
    $t_1 \leadsto_{<\beta}^{s'} \bot$. Then $t_1
    \leadsto_{<\alpha}^{s'} \bot$ which contradicts~$(\star)$.
  \item The claim is immediate for $\alpha = 0$. Suppose $t
    \succ_{\alpha}^s \top$ and $t \succ_{\alpha}^s \bot$. Then either
    $t \equiv \Xi A_\tau t_1$ or $t \equiv \Xi (K t_1) (K t_2)$.

    Assume $t \equiv \Xi A_\tau t_1$. Then, because $t \succ_\alpha^s
    \bot$, there exist $c \in \Sigma_\tau$ and $s' \ge s$ such that
    $t_1c \leadsto_{<\alpha}^{s'} \bot$. On the other hand, because
    $t \succ_\alpha^s \top$, we have $t_1c \leadsto_{<\alpha}^{s'}
    \top$. This contradicts~$(\star)$.

    If $t \equiv \Xi (K t_1) (K t_2)$ then the argument is analogous.
  \end{enumerate}
\end{proof}

It follows from Lemma~\ref{r:lem_monotone}, by a simple cardinality
argument, that there exists an ordinal~$\zeta$ such that
\mbox{$\succ_\zeta^s\; = \;\succ_{<\zeta}^s$} for all $s \in \Sc$. We
use the notations~$\succ^s$ and~$\leadsto^s$ without subscripts
for~$\succ_\zeta^s$ and~$\leadsto_\zeta^s$.

\begin{definition}\label{r:def_model_2}
  The structure~$\M$ is defined as follows. We define the extensional
  combinatory algebra~$\C$ of~$\M$ to be the set of equivalence
  classes of~$=_R$ on closed terms. We take the set~$\Sc$ of states
  of~$\N$ to be the set of states of~$\M$ as well. For $c \in \Sigma$
  we define the interpretation~$I$ of~$\M$ by $I(c) = [c^+]_R$, where
  $c^+ \in \Sigma^+$ corresponds to the element
  $\valuation{c}{\N}{}$. The function~$\varsigma_\M$ is given by
  $\varsigma_\M(d) = \{ s \in \Sc \;|\; \exists t . d = [t]_R \wedge t
  \leadsto^s \top \}$, where~$t$ is required to be closed.
\end{definition}

\begin{lemma}\label{r:lem_extensional_2}
  Let $t_1$ and $t_2$ be closed terms. If for all closed $t_3$ we have
  $t_1 t_3 =_R t_2 t_3$, then $t_1 =_R t_2$.
\end{lemma}

\begin{proof}
  If $t_1 t_3 =_R t_2 t_3$ for all closed $t_3$, then in particular
  $t_1 \nu =_R t_2 \nu$ for an external constant~$\nu$ not occuring in
  $t_1$ and $t_2$. By the Church-Rosser property of $R$ there exists
  $t$ such that $t_1 \nu \reduces_R t$ and $t_2 \nu \reduces_R
  t$. Because there are no rules in $R$ involving $\nu$, and $\nu$
  cannot be produced by any of the reductions, it is easy to verify by
  induction on the number of reduction steps that $t \equiv C'[\nu]$,
  $t_1 \nu \equiv C_1[\nu]$, $t_2 \nu \equiv C_2[\nu]$, $C_1
  \reduces_R C'$ and $C_2 \reduces_R C'$, where $\nu$ does not occur
  in $C_1$, $C_2$ or $C'$. Hence $t_1 x \equiv C_1[x] =_R C_2[x]
  \equiv t_2 x$ for a variable $x$, and thus $\lambda x . t_1 x =_R
  \lambda x . t_2 x$. Because $R$ contains the rule of
  $\eta$-reduction, we conclude that $t_1 =_R t_2$.
\end{proof}

\begin{lemma}\label{r:lem_cred}
  Let~$C$ be a context and let $\rho \in \{\top, \bot\}$. If $C[\rho]
  \reduces_R t$ then there exists a context $C'$ such that $C
  \reduces_R C'$ and $t = C'[\rho]$.
\end{lemma}

\begin{proof}
  Because there are no rules in $R$ involving $\rho$, the claim is
  easy to verify by induction on the number of reduction steps.
%% NOTE: the fact that propositions cannot be arguments is used in the
%% proof
\end{proof}

The following lemma is a much simplified analogon of
Lemma~\ref{r:lem_context}.

\begin{lemma}\label{r:lem_succ_stable_2}
  If $t \succ^s \rho_1$ and $C[\rho_1] \leadsto_\alpha^s \rho_2$ then
  $C[t] \leadsto^s \rho_2$.
\end{lemma}

\begin{proof}
  Induction on $\alpha$.

  Suppose $t \succ^s \rho_1$ and $C[\rho_1] \leadsto_\alpha^s
  \rho_2$. By Lemma~\ref{r:lem_cred} we have $C \reduces_R C'$ where
  $C'[\rho_1] \succ_\alpha^s \rho_2$. It suffices to show that $C'[t]
  \succ^s \rho_2$.

  First assume $\alpha = 0$. The claim is obvious if $C'$ does not
  contain $\Box$, so assume it does. Then by inspecting the
  definitions we see that there are the following two possibilities.
  \begin{itemize}
  \item If $C' \equiv \Box$ and $\rho_1 \equiv \rho_2$ then the claim
    is obvious.
  \item If $C' \equiv H \Box$ and $\rho_2 \equiv \top$, then either $t
    \succ \top$ or $t \succ \bot$. Thus $H t \succ \top$ by
    condition~$(H_\top)$.
  \end{itemize}

  Now let $\alpha > 0$. If $C' \equiv \Xi A_\tau C_1$ and $\rho_2 =
  \top$ then for all $c \in \Sigma_\tau$ and all $s' \ge s$ we have
  $C_1[\rho_1] c \leadsto_{<\alpha}^{s'} \top$. We conclude by the
  inductive hypothesis that for all $c \in \Sigma_\tau$ and all $s'
  \ge s$ we have $C_1[t] c \leadsto^{s'} \top$. Hence $C'[t] \succ^s
  \top$.

  If $C' \equiv \Xi (K C_1) C_2$ and $\rho_2 = \top$ then
  \begin{itemize}
  \item $C_1[\rho_1] \leadsto_{<\alpha}^s \top$ or $C_1[\rho_1]
    \leadsto_{<\alpha}^s \bot$, and
  \item for all $s' \ge s$ such that $C_1[\rho_1]
    \leadsto_{<\alpha}^{s'} \top$ we have $C_2 \reduces_R K C_2'$ with
    $C_2'[\rho_1] \succ_{<\alpha}^{s'} \top$.
  \end{itemize}
  By Lemma~\ref{r:lem_upward_closed} for all $s' \ge s$ we have:
  \begin{itemize}
  \item[$(\star)$] $C_1[\rho_1] \leadsto_{<\alpha}^{s'} \top$ or
    $C_1[\rho_1] \leadsto_{<\alpha}^{s'} \bot$.
  \end{itemize}
  By the inductive hypothesis $C_1[t] \leadsto^s \top$ or $C_1[t]
  \leadsto^s \bot$. Let $s' \ge s$ be such that $C_1[t] \leadsto^{s'}
  \top$. By~$(\star)$ we have $C_1[\rho_1] \leadsto_{<\alpha}^{s'}
  \top$, because if $C_1[\rho_1] \leadsto_{<\alpha}^{s'} \bot$ then
  $C_1[t] \leadsto^{s'} \bot$ by the inductive hypothesis, which
  contradicts $C_1[t] \leadsto^{s'} \top$ by part~2 of
  Lemma~\ref{r:lem_monotone}. Hence, $C_2 \reduces_R K C_2'$ with
  $C_2'[\rho_1] \succ_{<\alpha}^{s'} \top$, and by the inductive
  hypothesis $C_2'[t] \leadsto_{<\alpha}^{s'} \top$. This implies
  $C'[t] \succ_\alpha^s \top$.

  In all other cases the proof is similar.
\end{proof}

This finishes the more difficult part of the construction correctness
proof. As in Section~\ref{r:sec_construction} it remains to prove
several simple lemmas implying that $\M$ satisfies the conditions
imposed on an illative Kripke model for $\I_0$. For convenience we
reformulate the definition of an illative Kripke model for~$\I_0$ in
terms of the notions used to construct $\M$.

\begin{fact}\label{r:fact_illat_model}
  If the following conditions hold, then $\M$ is an illative Kripke
  model for~$\I_0$.
  \begin{enumerate}
  \item If $t_1 =_R t_2$ then $t_1 \leadsto^s \top$ is equivalent to
    $t_2 \leadsto^s \top$. \label{r:mikm_1}
  \item If $t \leadsto^s \top$ then $t \leadsto^{s'} \top$ for all $s'
    \ge s$. \label{r:mikm_2}
  \item If for all $t_3$ we have $t_1 t_3 =_R t_2 t_3$ then $t_1 =_R
    t_2$. \label{r:mikm_3}
  \item If $L t_1 \leadsto^s \top$ and for all $s' \ge s$ and all
    $t_3$ such that $t_1 t_3 \leadsto^{s'} \top$ we have $t_2 t_3
    \leadsto^{s'} \top$, then $\Xi t_1 t_2 \leadsto^s
    \top$. \label{r:mikm_4}
    \vspace{-1.2em}
  \item If $\Xi t_1 t_2 \leadsto^s \top$ then for all $t_3$ such that
    $t_1 t_3 \leadsto^s \top$ we have $t_2 t_3 \leadsto^s
    \top$. \label{r:mikm_5}
  \item If $L t_1 \leadsto^s \top$ and for all $s' \ge s$ and all
    $t_3$ such that $t_1 t_3 \leadsto^{s'} \top$ we have $H (t_2 t_3)
    \leadsto^{s'} \top$, then $H (\Xi t_1 t_2) \leadsto^s
    \top$. \label{r:mikm_6}
  \item If $t \leadsto^s \top$ then $H t \leadsto^s
    \top$. \label{r:mikm_7}
  \item $L H \leadsto^s \top$, \label{r:mikm_8}
  \item $L A_\tau \leadsto^s \top$ for $\tau \in \B$. \label{r:mikm_9}
  \end{enumerate}
\end{fact}

\begin{proof}
  Condition~(\ref{r:mikm_1}) ensures that $s \in \varsigma_\M([t]_R)$ is
  equivalent to $t \leadsto^s \top$. Condition~(\ref{r:mikm_2})
  implies that for any $d \in \M$ the set $\varsigma_\M(d)$ is
  upward-closed. Condition~(\ref{r:mikm_3}) implies that the
  combinatory algebra of~$\M$ is extensional. The remaining conditions
  are a reformulation of the conditions imposed on~$\varsigma$ in an
  illative Kripke model for~$\I_0$.
\end{proof}

\begin{lemma}\label{r:lem_h_t_f_2}
  $H t \leadsto^s \top$ iff $t \leadsto^s \top$ or $t \leadsto^s
  \bot$.
\end{lemma}

\begin{proof}
  Follows directly from definitions.
\end{proof}

\begin{lemma}\label{r:lem_l}
  If $L t \leadsto^s \top$ then exactly one of the following
  holds:
  \begin{itemize}
  \item $t \reduces_R A_{\tau}$ for some $\tau \in \B$,
  \item $t \reduces_R H$,
  \item $t \reduces_R K t'$ with $t' \leadsto^s \top$ or $t'
    \leadsto^s \bot$.
  \end{itemize}
\end{lemma}

\begin{proof}
  Easy inspection of the rules in the definition of
  $\succ_\alpha^s$. That the conditions are exclusive is a consequence
  of the Church-Rosser property of~$R$.
\end{proof}

\begin{lemma}\label{r:lem_ikm_1_3}
  The following conditions are satisfied.
  \begin{itemize}
  \item If $L t_1 \leadsto^s \top$ and for all $s' \ge s$ and all
    $t_3$ such that $t_1 t_3 \leadsto^{s'} \top$ we have $t_2 t_3
    \leadsto^{s'} \top$, then $\Xi t_1 t_2 \leadsto^s \top$.
  \item If $L t_1 \leadsto^s \top$ and for all $s' \ge s$ and all
    $t_3$ such that $t_1 t_3 \leadsto^{s'} \top$ we have $H (t_2 t_3)
    \leadsto^{s'} \top$, then $H (\Xi t_1 t_2) \leadsto^s \top$.
  \end{itemize}
\end{lemma}

\begin{proof}
  Suppose $L t_1 \leadsto^s \top$ and for all $s' \ge s$ and all $t_3$
  such that $t_1 t_3 \leadsto^{s'} \top$ we have $t_2 t_3
  \leadsto^{s'} \top$. We consider possible cases according to
  Lemma~\ref{r:lem_l}.
  \begin{itemize}
  \item $t_1 \reduces_R A_{\tau}$ for $\tau \in \B \cup \{o\}$. If $c
    \in \Sigma_\tau$ and $s' \ge s$ then $A_\tau c \succ^{s'} \top$,
    so also $t_1 c \leadsto^{s'} \top$ by
    Corollary~\ref{r:corollary_t_conv_2}, and thus $t_2 c
    \leadsto^{s'} \top$. Hence $\Xi A_\tau t_2 \succ^s \top$
    by~$(\Xi_\top)$. Therefore, $\Xi t_1 t_2 \leadsto^s \top$.
  \item $t_1 \reduces_R K t_1'$ with $t_1' \leadsto^s \top$ or $t_1'
    \leadsto^s \bot$. Let $s' \ge s$ be such that $t_1' \leadsto^{s'}
    \top$. Then $t_1 t_3 \leadsto^{s'} \top$ for arbitrary
    closed~$t_3$, so $t_2 t_3 \leadsto^{s'} \top$ for any
    closed~$t_3$, in particular for $t_3 \equiv \nu$ an external
    constant not occuring in~$t_2$. We have $t_2 \nu \reduces_R t_2'
    \succ^{s'} \top$. It is easy to see by inspecting the definitions
    that~$\nu$ cannot occur in~$t_2'$. Thus we also have $t_2 x
    \reduces_R t_2'$. Therefore $t_2 \from_\eta \lambda x . t_2 x
    \reduces_R K t_2'$. So if there exists $s' \ge s$ such that $t_1'
    \leadsto^{s'} \top$ then $t_2 =_R K t_2'$, and for every such $s'
    \ge s$ we have $t_2' \succ^{s'} \top$. Thus $\Xi (K t_1') (K t_2')
    \succ^{s} \top$, so $\Xi t_1 t_2 \leadsto^s \top$, by
    Corollary~\ref{r:corollary_t_conv_2}. If there does not exist $s'
    \ge s$ such that $t_1' \leadsto^{s'} \top$, then also $\Xi t_1 t_2
    \leadsto^s \top$.
  \end{itemize}

  The second claim is verified in a similar manner using
  Lemma~\ref{r:lem_l}, Lemma~\ref{r:lem_h_t_f_2},
  Corollary~\ref{r:corollary_upward_closed} and
  Corollary~\ref{r:corollary_t_conv_2}.
\end{proof}

\begin{lemma}\label{r:lem_ikm_2}
  If $\Xi t_1 t_2 \leadsto^s \top$ then for all $s' \ge s$ and all
  terms~$t_3$ such that $t_1 t_3 \leadsto^{s'} \top$ we have $t_2 t_3
  \leadsto^{s'} \top$.
\end{lemma}

\begin{proof}
  Suppose $\Xi t_1 t_2 \leadsto^s \top$. Then $\Xi t_1 t_2 \reduces_R
  \Xi t_1' t_2' \succ^s \top$ with $t_i \reduces_R t_i'$. There are
  three cases.
  \begin{itemize}
  \item $t_1' \equiv A_\tau$ where $\tau \in \B$ and for all $s' \ge
    a$ and all $c \in \Sigma_\tau$ we have $t_2' c \leadsto^{s'}
    \top$. Assume $s' \ge s$ and $t_1 t_3 \leadsto^{s'} \top$. Then
    also $A_\tau t_3 \leadsto^{s'} \top$ by
    Corollary~\ref{r:corollary_t_conv_2}. This is only possible when
    $t_3 \in \Sigma_\tau$. This implies $t_2' t_3 \leadsto^{s'} \top$,
    so also $t_2 t_3 \leadsto^{s'} \top$ because $t_2 \reduces_R
    t_2'$.
  \item $t_1' \equiv H$ and for all $s' \ge s$ and all $\rho \in
    \{\top,\bot\} \subseteq \Sigma_\tau$ we have $t_2' \rho
    \leadsto^{s'} \top$. Assume $s' \ge s$ and $t_1 t_3 \leadsto^{s'}
    \top$. Then also $H t_3 \leadsto^{s'} \top$ by
    Corollary~\ref{r:corollary_t_conv_2}. By Lemma~\ref{r:lem_h_t_f_2}
    either $t_3 \leadsto^{s'} \top$ or $t_3 \leadsto^{s'} \bot$. In
    any case, we may use Lemma~\ref{r:lem_succ_stable_2} to conclude
    $t_2 t_3 \leadsto^{s'} \top$.
  \item $t_1' \equiv K t_1''$ and for all $s' \ge s$ such that $t_1''
    \leadsto^{s'} \top$ we have $t_2' \reduces_R K t_2''$ with $t_2''
    \succ^{s'} \top$. Assume $s' \ge s$ and $t_1 t_2 \leadsto^{s'}
    \top$. Then $t_1'' \leadsto^{s'} \top$ by
    Corollary~\ref{r:corollary_t_conv_2}. So also $t_2 t_3
    \leadsto^{s'} \top$ by Corollary~\ref{r:corollary_t_conv_2},
    because $t_2 \reduces_R K t_2''$ with $t_2'' \succ^{s'} \top$.
  \end{itemize}
\end{proof}

\begin{corollary}
  The structure $\M$ constructed in Definition~\ref{r:def_model_2} is
  an illative Kripke model for $\I_0$.
\end{corollary}

\begin{proof}
  It suffices to check the conditions of
  Fact~\ref{r:fact_illat_model}. Condition~(\ref{r:mikm_1}) follows
  from
  Corollary~\ref{r:corollary_t_conv_2}. Condition~(\ref{r:mikm_2}) is
  a consequence of
  Corollary~\ref{r:corollary_upward_closed}. Condition~(\ref{r:mikm_3})
  follows from
  Lemma~\ref{r:lem_extensional_2}. Conditions~(\ref{r:mikm_4})
  and~(\ref{r:mikm_6}) follow from
  Lemma~\ref{r:lem_ikm_1_3}. Lemma~\ref{r:lem_ikm_2} implies
  condition~(\ref{r:mikm_5}). Conditions~(\ref{r:mikm_7}),
  (\ref{r:mikm_8}) and~(\ref{r:mikm_9}) are obvious from definitions.
\end{proof}

\begin{lemma}\label{r:lem_a_tau_c}
  If $\tau \in \Tc$ and $c \in \Sigma_\tau$ then for all states $s$ we
  have $A_\tau c \leadsto^s \top$.
\end{lemma}

\begin{proof}
  Straightforward induction on the size of~$\tau$.
\end{proof}

It remains to prove that the values in $\N$ of formulas of
$\mbox{PRED2}_0$ are faithfully represented by the values of their
translations in $\M$. From this completeness will directly follow.

\begin{definition}
  Recall that for $c \in \Sigma^+$, we denote by $\delta(c)$ the
  element of $\N$ corresponding to $c$, if there is one. We say that
  an $\M$-valuation $\widetilde{w}$ \emph{mirrors} an $\N$-valuation
  $w$, if for every variable $x$ there exists $c \in \Sigma^+$ such
  that $w(x) = \delta(c)$ and $\widetilde{w}(x) = [c]_R$. In other
  words, $\widetilde{w}$ is the valuation assigning to each variable
  $x$ the equivalence class of the constant corresponding to the
  element $w(x)$. Note that given $w$ the valuation $\widetilde{w}$ is
  uniquely determined.
\end{definition}

To avoid confusion, from now on we use $q_1$, $q_2$, etc. for terms of
$\mbox{PRED2}_0$. By $t_1$, $t_2$, etc. we denote closed terms from
$\T(\Sigma^+)$. We use $c$, $c_1$, $c_2$, etc. for constants from
$\Sigma^+$.

\begin{lemma}\label{r:lem_val_delta}
  For any $\N$-valuation $w$ and any term $q$ of $\mbox{PRED2}_0$
  which is not a formula, we have
  $\valuation{\transl{q}}{\M}{\widetilde{w}} = [c]_R$ for some $c \in
  \Sigma^+$ such that $\delta(c) = \valuation{q}{\N}{w}$.
\end{lemma}

\begin{proof}
  Induction on the size of $q$. If $q$ is a constant then $\transl{q}
  = q$ and $\valuation{\transl{q}}{\M}{\widetilde{w}} =
  \valuation{q}{\M}{\widetilde{w}} = I_\M(q) = [c]_R$ for some $c
  \in \Sigma^+$ such that $\delta(c) = \valuation{q}{\N}{}$. If $q =
  x$ is a variable of type $\tau \in \B$ then $\transl{q} = x$. So
  $\valuation{\transl{q}}{\M}{\widetilde{w}} = \widetilde{w}(x) =
  [c]_R$ for $c \in \Sigma^+$ such that $w(x) = \delta(c)$, by
  definition of $\widetilde{w}$.

  Otherwise $q \equiv q_1 q_2$. Neither $q_1$ nor $q_2$ is a formula,
  so by the inductive hypothesis
  $\valuation{\transl{q_1}}{\M}{\widetilde{w}} = [c_1]_R$ and
  $\valuation{\transl{q_2}}{\M}{\widetilde{w}} = [c_2]_R$ where
  $\delta(c_1) = \valuation{q_1}{\N}{w}$ and $\delta(c_2) =
  \valuation{q_2}{\N}{w}$. We have $\transl{q} = \transl{q_1}
  \transl{q_2}$, so $\valuation{\transl{q}}{\M}{\widetilde{w}} =
  \valuation{\transl{q_1}}{\M}{\widetilde{w}} \cdot_\M
  \valuation{\transl{q_2}}{\M}{\widetilde{w}} = [c_1]_R \cdot_\M
  \ [c_2]_R = [c_1 c_2]_R$. Let $c \in \Sigma^+$ be such that
  $\delta(c) = \delta(c_1) \cdot_\N \delta(c_2)$. In~$R$ there
  is a reduction rule $c_1 c_2 \to c$ because $\F(c_1)(c_2)
  = c$. Thus $[c_1 c_2]_R = [c]_R$. We also have $\delta(c)
  = \valuation{q_1}{\N}{w} \cdot_\N \valuation{q_2}{\N}{w} =
  \valuation{q_1 q_2}{\N}{w} = \valuation{q}{\N}{w}$.
\end{proof}

\begin{lemma}\label{r:lem_forcing}
  For any formula $\phi$ of $\mbox{PRED2}_0$, any state $s$, and any
  $\N$-valuation $w$ we have:
  \[
  s, w \forces_\N \phi \mathrm{\ \ iff\ \ } s, \widetilde{w}
  \forces_\M \transl{\phi}
  \]
\end{lemma}

\begin{proof}
  Induction on the size of $\phi$.

  If $\phi$ is a variable or a constant, then our claim follows easily
  from definitions. If $\phi = q_1 q_2$, then neither $q_1$ nor $q_2$
  is a formula, so by Lemma~\ref{r:lem_val_delta} we have
  $\valuation{\transl{q_1}}{\M}{\widetilde{w}} = [c_1]_R$ and
  $\valuation{\transl{q_2}}{\M}{\widetilde{w}} = [c_2]_R$ where $c_1,
  c_2 \in \Sigma^+$ and $\delta(c_1) = \valuation{q_1}{\N}{w}$,
  $\delta(c_2) = \valuation{q_2}{\N}{w}$. We have $[c_1]_R \cdot
  [c_2]_R = [c_1 c_2]_R = [c]_R$ for $c \in \Sigma^+$ such that
  $\delta(c) = \delta(c_1) \cdot_\N \delta(c_2) = \valuation{t_1
    t_2}{\N}{w}$. The claim now follows from the definition
  of~$\succ_0^s$.

  If $\phi = \varphi \supset \psi$ then $\transl{\phi} =
  \transl{\varphi} \supset \transl{\psi}$. Suppose $s, \widetilde{w}
  \forces_\M \transl{\varphi} \supset \transl{\psi}$. Let $s' \ge s$
  be such that $s', w \forces_\N \varphi$. By the inductive hypothesis
  $s', \widetilde{w} \forces_\M \transl{\varphi}$. Note that we also
  have $s', \widetilde{w} \forces_\M \transl{\varphi} \supset
  \transl{\psi}$. By condition~(\ref{r:cond_supset_illat_02}) in
  Fact~\ref{r:fact_supset_illat} we obtain $s', \widetilde{w}
  \forces_\M \transl{\psi}$, which implies $s', w \forces_\N \psi$ by
  the IH. From Definition~\ref{r:def_kripke_pred2} it now follows that
  $s, w \forces_\N \varphi \supset \psi$. The other direction is
  analogous.

  If $\phi = \forall x . \varphi$ where $x \in V_{\tau}$, $\tau \in \B
  \cup \{o\}$, then $\transl{\forall x . \varphi} = \Xi A_{\tau}
  \lambda x . \transl{\varphi}$.

  Suppose $s, \widetilde{w} \forces_\M \transl{\forall x . \varphi}$,
  i.e., $s, \widetilde{w} \forces_\M \Xi A_{\tau} \lambda x
  . \transl{\varphi}$. Let $s' \ge s$, $d \in \D_{\tau}^\N$, and $u =
  w[x/d]$. There exists $c \in \Sigma^+$ such that $\widetilde{u}(x) =
  [c]_R$ and $\delta(c) = d$. The constant $c$ is a canonical constant
  of type $\tau \in \B \cup \{o\}$, so $s', \widetilde{w} \forces_\M
  A_{\tau} c$, by definition of~$\M$. We also have $s', \widetilde{w}
  \forces_\M \Xi A_{\tau} \lambda x . \transl{\varphi}$, so we
  conclude that $s', \widetilde{w} \forces_\M (\lambda x
  . \transl{\varphi}) c$. This implies $s', \widetilde{u} \forces_\M
  \transl{\varphi}$, and hence $s', w[x/d] \forces_\N \varphi$ by the
  IH. Therefore $s, w \forces_\N \forall x . \varphi$, by
  Definition~\ref{r:def_kripke_pred2}.

  For the other direction, we need to show that if $s, w \forces_\M
  \forall x . \varphi$ then $s, \widetilde{w} \forces_\M \Xi A_\tau
  \lambda x . \transl{\varphi}$, where $\tau \in \B \cup \{o\}$. If
  $v$ is an $\M$-valuation and $t \in \T(\Sigma^+)$, then by~$t^v$ we
  denote the term~$t$ with every free variable~$x$ substituted for a
  representant of the equivalence class~$v(x)$. By induction on the
  size of~$t$ one may easily verify that $\valuation{t}{\M}{v} =
  \valuation{t^v}{\M}{}$, but Lemma~\ref{r:lem_extensional_2} is
  needed for the case of lambda-abstraction. Hence $s, v \forces_\M t$
  is equivalent to $t^v \leadsto^s \top$. Now the condition $s,
  \widetilde{w} \forces_\M \Xi A_\tau \lambda x . \transl{\varphi}$
  may be reformulated as $\Xi A_\tau (\lambda x
  . \transl{\varphi})^{\widetilde{w}} \leadsto^s \top$. Therefore it
  suffices to prove, assuming $s, w \forces_\N \forall x . \varphi$,
  that for all canonical constants $c \in \Sigma_\tau$ of type~$\tau
  \in \B \cup \{o\}$ and all $s' \ge s$ we have $(\lambda x
  . \transl{\varphi})^{\widetilde{w}} c \leadsto^{s'} \top$. Let $u =
  w[x/\delta(c)]$. We have $\widetilde{u} = \widetilde{w}[x/c]$. Hence
  $(\lambda x . \transl{\varphi})^{\widetilde{w}} c \leadsto^{s'}
  \top$ is equivalent to $\transl{\varphi}^{\widetilde{u}}
  \leadsto^{s'} \top$, which is the same as $s, \widetilde{u}
  \forces_\M \transl{\varphi}$. Because $s, w \forces_\N \forall x
  . \varphi$, $s' \ge s$ and $u = w[x / \delta(c)]$, we conclude that
  $s', u \forces_\N \varphi$. By the inductive hypothesis we obtain
  $s, \widetilde{u} \forces_\M \transl{\varphi}$ which completes the
  proof.
\end{proof}

\begin{theorem}
  The embedding is complete, i.e., $\transl{\Delta}, \Gamma(\Delta,
  \varphi) \proves_{\I_0} \transl{\varphi}$ implies $\Delta
  \proves_{\mathrm{PRED2}_0} \varphi$.
\end{theorem}

\begin{proof}
  Suppose $\Delta \notproves_{\mathrm{PRED2}_0} \varphi$. Let $\N$ be
  a Kripke model, $v$ an $\N$-valuation and $s$ a state of $\N$ such
  that $s, v \forces_\N \Delta$, but $s, v \notforces_\N \varphi$. We
  use the construction in Definition~\ref{r:def_model_2} to obtain an
  illative Kripke model $\M$. By Lemma~\ref{r:lem_forcing} the
  condition $s, v \forces_\N \psi$ is equivalent to $s, \widetilde{v}
  \forces_\M \transl{\psi}$. Therefore $s, \widetilde{v} \forces_\M
  \transl{\Delta}$ but $s, \widetilde{v} \notforces_\M$. Using
  Lemma~\ref{r:lem_a_tau_c}, it is a matter of routine to verify that
  also $s, \widetilde{v} \forces_\M \Gamma(\Delta, \varphi)$. By the
  soundness part of Theorem~\ref{r:thm_ikm_complete} this implies
  $\transl{\Delta}, \Gamma(\Delta, \varphi) \notproves_{\I_0}
  \transl{\varphi}$.
\end{proof}

\section{Remarks and open problems}

\begin{remark}
  In this paper we use lambda-calculus with
  $\beta\eta$-equality. Lambda-calculus with $\beta$-equality or
  combinatory logic with weak equality could be used instead. The
  proofs and definitions would only need minor adjustments.
\end{remark}

\begin{remark}
  It is clear that the methods presented here may be used to prove
  completeness of the embedding of propositional second-order logic
  into an extension of~$\I P$ from~\cite{illat01}. This extension
  of~$\I P$ is essentially $\I_0$ but with rules $P_i$, $P_e$, $P_H$
  from Lemma~\ref{r:lem_illat_admissible} instead of the more general
  rules for~$\Xi$. Whether such an extension is complete for
  second-order propositional logic was posed as an open problem in
  \cite{illat01}.

  The open problem related to~$\I_0$ given in~\cite{illat01} was
  whether full second-order predicate logic may be faithfully embedded
  into it. We do not know the answer to this question. One problem
  with extending our methods was already noted in
  Remark~\ref{r:rem_problem_hol}. It is not straightforward to extend
  our construction to obtain a model with quantification over
  predicates and more than one state. Another obstacle is that our
  construction of a model for~$\I_\omega^c$ crucially depends on the
  fact that the model of higher-order logic being transformed is a
  full model. Thus the construction cannot be used to show
  completeness of an embedding of higher-order logic
  into~$\I_\omega^c$. Informally speaking, a full model is needed to
  ensure that no ``essentially new'' functions may be ``created'' at
  later stages $\alpha$ of the inductive definition.

  In \cite{illat02} and \cite{illat03} two indirect
  propositions-as-types translations of first-order propositional and
  predicate logic were shown complete for two illative systems $\I F$
  and $\I G$, which are stronger than $\I P$ and $\I \Xi$,
  respectively. It is interesting whether our methods may be used to
  obtain these results, or improve on them.
\end{remark}

\begin{remark}
  In~\cite{Czajka2011} we presented an algebraic treatment of a
  combination of untyped combinatory logic with first-order classical
  logic. The model construction and the completeness proof there
  follow essentially the same pattern as those presented here, but
  they are much simpler. The system in~\cite{Czajka2011} contains an
  additional constant $\mbox{Cond}$ which allows for branching on
  formulas. It is not difficult to see that we could add such a
  constant to~$\I_\omega^c$ and our model construction would still go
  through.
\end{remark}

%% NOTE: the problem with defining T_\alpha^s by means of reduction
%% when we have more than one state is that then R_\alpha^s
%% \not\subseteq R_\alpha^{s'} for $s' \ge s$. It would be immediate
%% that if $t \contr_\alpha^s \top$ then $t \contr_\alpha^{s'}
%% \top$, but the implication $t \reduces_\alpha^s \top$ => $t
%% \reduces_\alpha^{s'} \top$ would not be obvious. Here it would
%% probably hold, because there is no way we can ``use'' a truth
%% value.

\begin{remark}
  The construction from Section~\ref{r:sec_construction} could also be
  used to show that classical many-sorted first-order logic may be
  faithfully embedded into~$\I_\omega^c$, but we omit this proof as it
  is analogous to that from Section~\ref{r:sec_embedding}. We do not
  know whether~$\I_\omega^c$ is conservative over stronger systems of
  logic, or whether~$\I_\omega$ is conservative over intuitionistic
  first-order logic.
\end{remark}

\bibliography{hoicl}{}
\bibliographystyle{alpha}

\end{document}